\documentclass[11pt,reqno]{amsart}

\usepackage{amsthm, mathrsfs, amsmath, amstext, amsxtra, amsfonts, dsfont, amssymb}
\usepackage{xcolor}
\usepackage{lmodern}
\usepackage[colorlinks, linkcolor=red, citecolor=blue, urlcolor=blue, hypertexnames=false]{hyperref}
\usepackage{graphicx}				
\numberwithin{equation}{section}
\usepackage{geometry}
\newcommand{\be}{\begin{equation}}
\newcommand{\ee}{\end{equation}}

\newtheorem{theorem}{Theorem}[section]
\newtheorem{lemma}[theorem]{Lemma}
\newtheorem{proposition}[theorem]{Proposition}
\newtheorem{corollary}[theorem]{Corollary}
\theoremstyle{definition}

\def\sideremark#1{\ifvmode\leavevmode\fi\vadjust{\vbox to0pt{\vss
 \hbox to 0pt{\hskip\hsize\hskip1em
\vbox{\hsize2cm\tiny\raggedright\pretolerance10000
 \noindent #1\hfill}\hss}\vbox to8pt{\vfil}\vss}}}%

{\end{list}}

\title[On the standing waves for the XFEL]{On the standing waves for the X-ray free electron laser Schr\"{o}dinger equation}
\author{Daomin Cao}
\address{School of Mathematics and Information Sciences, Guangzhou University, Guangzhou 510006; Institute of Applied Mathematics, Chinese Academy of Sciences, Beijing 100190, P.R. China}
\email{dmcao@amt.ac.cn}
\author{Binhua Feng}
\address{Department of Mathematics, Northwest Normal University, Lanzhou, 730070, China}
\email{binhuaf@163.com}
\author{Tingjian Luo}
\address{School of Mathematics and Information Sciences, Guangzhou University, Guangzhou 510006, China}
\email{luotj@gzhu.edu.cn}


\begin{document}
	
\begin{abstract} In this paper, we are concerned with the standing waves for the following nonlinear Schr\"{o}dinger equation
\[
i\partial_{t}\psi=-\Delta \psi+b^2(x_1^2+x_2^2)\psi+\frac{\lambda_1}{|x|}\psi+ \lambda_2(|\cdot|^{-1}\ast |\psi|^2)\psi-
 \lambda_3|\psi|^p \psi,~~~
(t,x)\in \mathbb{R}^+\times \mathbb{R}^3,
\]
where $0<p<4$. This equation arises as an effective single particle model in X-ray Free Electron Lasers. We mainly study the existence and stability/instability properties of standing waves for this equation, in two cases: the first one is that no magnetic potential is involved, (i.e. $b=0$ in the equation) and the second one is that $b\neq 0$. To be precise, in the first case, by considering a minimization problem on a suitable Pohozaev manifold, we prove the existence of radial solutions, and show further that the corresponding standing waves are strongly unstable by blow-up in finite time. Moreover, by making use of the ideas of these proofs, we are able to prove the existence and instability of normalized solutions, whose proofs seem to be new, compared with the studies of normalized solutions in the existing literature. This study also indicates that there is a close connection  between the study of the strong instability and the one of the existence of normalized solutions. In the second case, the situation is more difficult to be treated, due to the additional term of the partial harmonic potential. We manage to prove the existence of stable standing waves for $p\in (0,4)$ and with some assumptions on the coefficients, solutions are obtained as global minimizers if $p\in (0,\frac{4}{3}]$, and as local minimizers if $p\in [\frac{4}{3}, 4)$. In the mass-critical and supercritical cases $p\in [\frac{4}{3}, 4)$, we also establish the variational characterization of the ground states on a new manifold which is different from the one neither of the Nehari type nor of the Pohozaev type, and then prove the existence of ground states. Finally under some assumptions on the coefficients, we prove that the ground state standing waves are strongly unstable.

\end{abstract}	
	
\maketitle
\newpage
\tableofcontents

\section{Introduction}

In this paper we consider a first-principles model for beam-matter interaction in X-ray free electron lasers
(XFEL) \cite{ch,fr}. The fundamental model for XFEL
is the following nonlinear Schr\"{o}dinger equation (NLS)
\begin{equation}\label{e0}
\left\{
\begin{array}{l}
i\hbar \partial_{t}\psi=(i \hbar\nabla -A)^2\psi+\frac{\lambda_1}{|x|}\psi+\lambda_2 (|\cdot|^{-1}\ast |\psi|^2)\psi-\lambda_3
 |\psi|^p \psi,\ (t,x)\in \mathbb{R}^+\times \mathbb{R}^3, \\
\psi(0,x) = \psi_0 (x) .%
\end{array}%
\right.
\end{equation}
The coefficients $\lambda_1, \lambda_2, \lambda_3 \in \mathbb{R}$ and the exponent $p\in (0,4)$. $\hbar$
is supposed to be a scaled version of the Planck constant, which, without loss of
generality, shall be set to be $1$ in the sequel. A solution $\psi$ of this NLS can be considered as the wave function of an electron beam, interacting
self-consistently through the repulsive Coulomb (Hartree) force with strength  $\lambda_2$,
the attractive local Fock approximation with strength  $\lambda_3$ and interacting repulsively with an atomic nucleus, located
at the origin, of interaction strength  $\lambda_1$, see \cite{ch,fr}.
Since XFEL is more
powerful by several orders of magnitude than more conventional lasers, the systematic investigation of many
of the standard assumptions and approximations has attracted increased attention.

However, to the best of our knowledge, there are only a few mathematical results on equation \eqref{e0}, despite the important significance of the involved applications (see \cite{ch,fr} and their references given there).
Recently, when the magnetic potential $A$ depends only on time $t$, Antonelli et al. in \cite{an} investigated the asymptotic behavior of solution of equation \eqref{e0} with $0<p<\frac{4}{3}$ in the highly oscillating regime. Antonelli et al. in \cite{aah} considered similar problems by numerical simulations. The first author has extended the study of Antonelli et al. in \cite{an} to power nonlinearity with exponent $0<p <4$ and proved the stronger form of
convergence, see \cite{fna2017}.
In \cite{f18dcds}, the first author and Zhao systematically investigated the local and global well-posedness, the $H_A^2$ regularity for equation \eqref{e0} with general time-dependent electromagnetic field and Coulomb potentials. Based on these results, optimal bilinear control problem for \eqref{e0} has been investigated in \cite{fcontrol}.

In this paper, we consider the potential $A$ defined by
\begin{equation}\label{magnetic}
A(x)=b(-x_2,x_1,0),~~\mbox{for}~~x=(x_1,x_2,x_3)\in \mathbb{R}^3,\quad b\in \mathbb{R},
\end{equation}
which is the vector potential of the constant magnetic field
 $B=rot A(x)=(0,0,2b)$. Due to $divA(x)=0$, it follows that
\[
(i \nabla -A)^2\psi=-\Delta \psi-2iA(x)\cdot \nabla \psi+|A(x)|^2\psi=-\Delta \psi-2i \frac{\partial \psi}{\partial \theta}+b^2\rho^2 \psi,
\]
where the following cylindrical coordinates $(\rho,\theta,z)$ in $\mathbb{R}^3$ are used
\[
x_1=\rho\cos \theta,~~~x_2=\rho\sin \theta,~~~x_3=z.
\]
Then \eqref{e0} leads to the following equations:
 \begin{equation}\label{equation partial}
\left\{
\begin{array}{l}
i\partial_{t}\psi=-\Delta \psi+b^2(x_1^2+x_2^2)\psi+\frac{\lambda_1}{|x|}\psi+ \lambda_2(|\cdot|^{-1}\ast |\psi|^2)\psi-
 \lambda_3|\psi|^p \psi,~~
(t,x)\in \mathbb{R}^+\times \mathbb{R}^3, \\
\psi(0,x) = \psi_0 (x) .%
\end{array}%
\right.
\end{equation}%
The Cauchy problem for \eqref{equation partial} is locally well-posed in the energy space $X$, see the details in Proposition \ref{pro2.1}, or Theorem 9.1.5 in \cite{ca2003} and Theorem 3.1 in \cite{f18dcds}. Here, the energy space $X$ for \eqref{equation partial} is defined by
\[
X:=\left\{u \in H^1(\mathbb{R}^3): ~~\int_{\mathbb{R}^{3}}(x_1^2+x_2^2)|u|^2dx <\infty\right\},
\]
with the norm
\[
\|u\|_{X}^2:=\|\nabla u\|_{L^2}^2+\|u\|_{L^2}^2+\int_{\mathbb{R}^3}(x_1^2+x_2^2)|u(x)|^2dx.
\]

In this paper, we will systematically study the standing waves of \eqref{equation partial}, which are of the form $\psi(t,x):=e^{i \omega t}u(x)$, $\omega \in \mathbb{R}$. Then equivalently we turn to study the following stationary equation:
\begin{equation}\label{elliptic partial}
-\Delta u+\omega u+b^2(x_1^2+x_2^2)u+\frac{\lambda_1}{|x|}u+ \lambda_2(|\cdot|^{-1}\ast |u|^2)u-
 \lambda_3|u|^p u=0, ~\ x\in \mathbb{R}^3.
\end{equation}
The equation \eqref{elliptic partial} is variational, whose action functional is defined by
\begin{eqnarray}\label{action partial}
S_{b,\omega}(u):= E_b(u)+\frac{\omega}{2}\| u\|_{L^2}^2,
\end{eqnarray}
where the corresponding energy $E_b(u)$ is defined by
\begin{align}\label{energyE}
E_b(u):= &\frac{1}{2}\|\nabla u\|_{L^2}^2+\frac{b^2}{2}\int_{\mathbb{R}^3} (x_1^2+x_2^2)|u|^{2}dx+\frac{\lambda_1}{2}\int_{\mathbb{R}^3} \frac{|u|^{2}}{|x|}dx  \nonumber\\&
+\frac{\lambda_2}{4}\int_{\mathbb{R}^3} (|x|^{-1}\ast|u|^{2})|u|^{2}dx-\frac{\lambda_3}{p+2}\|u\|_{L^{p+2}}^{p+2}.
\end{align}
Clearly $S_{b,\omega}(u)$ is well-defined and of class $C^1$ in $X$.
The set of all ground states for \eqref{elliptic partial} is defined by
\begin{equation}\label{G}
\mathcal{G_\omega}:=\{u\in \mathcal{F_\omega}:~~S_{b,\omega}(u)\leq S_\omega(v),\ \forall~v\in \mathcal{F_\omega}\},
\end{equation}
where
\[
\mathcal{F_\omega}:=\{v\in X\setminus \{0\}:~S'_{b,\omega}(v)=0\},
\]
is the set of all nontrivial solutions for \eqref{elliptic partial}.


We shall focus on the existence, stability and strong instability of ground state standing waves.
Note that since the action functional for \eqref{elliptic partial} includes six terms with different scaling rates, there
are several essential difficulties in the analysis. Recall that action functionals only including four
terms, which corresponds on some of the coefficients in \eqref{elliptic partial} vanishing, the standing waves involved
normalized solutions and stability/instability, have been studied by \cite{jeanCMP2017,jean16siam,jean-luo,jean-goutrans,jean-gouinstablility,
jean19siam,f20jdde,f20na,ohta03die,ohta18doublepower,ohta18inversepotential,ohtaharmonic,ohta18cpaa,
ohta95die,zhangcpde}. In these studies, we see particularly that there are some essential differences
between $b=0$ and $b\neq 0$, where the latter is more difficult to treat in analysis. Consequently, in
what follows we divide this paper into two parts.

In the first part of the paper, we consider the situation where the magnetic potential is
not involved, namely, $A=0$ or equivalently $b=0$ in \eqref{equation partial}. Then the corresponding stationary equation is
\begin{equation}\label{elliptic partial12}
-\Delta u+\omega u+\frac{\lambda_1}{|x|}u+ \lambda_2(|\cdot|^{-1}\ast |u|^2)u-
 \lambda_3|u|^p u=0,\quad  x\in \mathbb{R}^3,
\end{equation}
and the action functional is reduced to
\begin{eqnarray}\label{action}
S_{\omega}(u):= E(u)+\frac{\omega}{2}\| u\|_{L^2}^2,\quad \forall u\in H^1,
\end{eqnarray}
where the energy $E$ is defined by
\begin{eqnarray}\label{energy}
E(u):= \frac{1}{2}\|\nabla u\|_{L^2}^2+\frac{\lambda_1}{2}\int_{\mathbb{R}^3} \frac{|u|^{2}}{|x|}dx
+\frac{\lambda_2}{4}\int_{\mathbb{R}^3} (|x|^{-1}\ast|u|^{2})|u|^{2}dx-\frac{\lambda_3}{p+2}\|u\|_{L^{p+2}}^{p+2}.
\end{eqnarray}

Recall that in the mass-subcritical case $0<p<\frac{4}{3}$, the solution of \eqref{equation partial} exists globally, see \cite[Theorem 3.5]{f18dcds}. Then Li and Zhao \cite{licma} studied the existence of stable standing waves of \eqref{equation partial}. In the mass-critical and supercritical cases $\frac{4}{3}\leq p<4$,
the solution $\psi(t)$ of \eqref{equation partial} may blow up in finite time, see \cite[Remark 9.2.10]{ca2003}. Consequently we shall establish the strong instability of standing waves and also the existence of normalized solutions in the cases $\frac{4}{3}\leq p<4$. Before giving our results in this case, we remark that due to the appearance of Coulomb potential, the functional related with the equation \eqref{elliptic partial12} (also \eqref{elliptic partial}) does not keep invariant by translation, and also the technique of symmetric rearrangement can not be applied because of the nonlocal term. All these cause a difficulty in establishing the existence result. Hence, technically we restrict our analyses in the space of radially symmetric functions. However, as one will see that, most of our arguments in the proofs are general and could be applied to other problems whose functionals are invariant by translation, even in some direction, see the proofs of Proposition \ref{proposition ground state no} and Lemma \ref{lemma minimization problem}.

We first establish the following variational characterization of standing waves related to \eqref{elliptic partial12} on the Pohozaev manifold.
\begin{proposition}\label{proposition ground state no}
Let $\lambda_1\geq 0$, $\lambda_2,\lambda_3>0$, $\omega>0$ and $\frac{4}{3}\leq p<4$. Then, there exists a solution $u\in H^1_r\backslash \{0\}$ related to \eqref{elliptic partial12}. Moreover, $u$ solves the following minimization problem:
\begin{equation}\label{minimization ground no}
d(\omega):=\inf\{S_\omega(v):~v\in H^1_r\backslash \{0\},~~Q(v)= 0\},
\end{equation}
where $Q(v)$ is the Pohozaev functional of \eqref{elliptic partial12} defined by
\begin{align}\label{qu}
Q(v):=\|\nabla v\|_{L^2}^2+\frac{\lambda_1}{2}\int_{\mathbb{R}^3} \frac{|v|^{2}}{|x|}dx+\frac{\lambda_2}{4}\int_{\mathbb{R}^3} (|x|^{-1}\ast|v|^{2})|v|^{2}dx-\frac{3\lambda_3p}{2(p+2)}\|v\|_{L^{p+2}}^{p+2}.
\end{align}
\end{proposition}
\textbf{Remark 1.1.}
When $\lambda_1=0$, this result has been established in \cite[Lemma 3.7]{f20jdde}. Here, we extend \cite[Lemma 3.7]{f20jdde} to the case $\lambda_1>0$ and provide a simpler proof. In particular, the proof of
$S'_\omega(u)=0$ (see Lemma \ref{lemma daoshuweiling}) contains some very
general methods that may be useful for other related problems, see e.g. Lemma \ref{lemma daoshuweiling partial}. Moreover, our method also allows to simplify the proof for \cite[Theorem 3.2]{ruiz}. In order to solve the minimization problem \eqref{minimization ground no}, we consider an equivalent minimization problem \eqref{minimization dengjia no}, which can be easily solved by using the Brezis-Lieb lemma. Based on the minimization problem \eqref{minimization dengjia no}, we can easily obtain the sharp threshold of global
existence and blow-up, the strong instability of standing waves.
For this aim, we introduce the following invariant sets:
\[
\mathcal{A}_\omega:=\{v\in H_r^1:~~Q(v)>0~\mbox{and}~S_\omega(v)< S_\omega(u)\},
\]
\[
\mathcal{B}_\omega:=\{v\in H_r^1:~~Q(v)<0~\mbox{and}~S_\omega(v)<S_\omega(u)\},
\]
where $u$ is a minimizer of the minimization problem \eqref{minimization ground no}.
\begin{theorem}\label{theorem sharp}(Global versus blow-up dichotomy)
Let $\lambda_1\geq 0$, $\lambda_2,\lambda_3>0$, $\omega>0$ and $\frac{4}{3}\leq p<4$.
Then, $\mathcal{A}_\omega$ and $\mathcal{B}_\omega$ are invariant under the flow of \eqref{equation partial} with $b=0$. More precisely, if $\psi_0\in \mathcal{A}_\omega$ or $\psi_0\in \mathcal{B}_\omega$, then the corresponding  solution $\psi(t)\in \mathcal{A}_\omega$ or $\psi(t)\in \mathcal{B}_\omega$ for any $t\in [0,T^*)$, respectively. Moreover, we can obtain the following sharp threshold of global
existence and blow-up for \eqref{equation partial} with $b=0$.

\begin{enumerate}
\item  If $\psi_0\in \mathcal{A}_\omega$ and $\frac{4}{3}< p<4$, then the corresponding solution $\psi(t)$ of \eqref{equation partial} with $b=0$ exists globally in time.
\item  If $\psi_0 \in \mathcal{B}_\omega$ and $x\psi_0\in L^2$, then the  corresponding solution $\psi(t)$ of \eqref{equation partial} with $b=0$ blows up in finite time.
    \end{enumerate}
\end{theorem}
\textbf{Remark 1.2.} Since $\psi_0$ is radial, by using the method from the proof of \cite[Theorem 6.5.10]{ca2003} (see also \cite[Theorem 1.3]{dinh}), we can prove the corresponding solution $\psi(t)$ of \eqref{equation partial} with $b=0$ blows up in finite time without the assumption $x\psi_0\in L^2$.

In the mass-critical case $p=\frac{4}{3}$, the argument for global existence in Theorem 1.2 (1) does not work any more. However, we can obtain the following sharp threshold of global existence and blow-up in this case.
	\begin{theorem} \label{theorem sharp criteria critical}
Let $\lambda_1\geq0$, $\lambda_2,\lambda_3>0$, and $p=\frac{4}{3}$.
 Assume that $R$ is the unique ground state of the Schr\"odinger equation \begin{equation}\label{elliptic classical}
-\Delta w+ w+|w|^{\frac{4}{3}} w=0\quad \mbox{in }~ \mathbb{R}^3.
\end{equation}
  Then we have the following sharp criteria for global existence and blow-up of \eqref{equation partial} with $b=0$.
		\begin{enumerate}
			\item If the initial data $\psi_0$ satisfies $\|\psi_0\|_{L^2}\leq \lambda_3^{-\frac{3}{4}}\|R\|_{L^2}$, then the corresponding solution $\psi(t)$ exists globally in time.
			\item For any $\rho> \lambda_3^{-\frac{3}{4}}\|R\|_{L^2}$, there exists $\psi_0\in H^1$ such that $\|\psi_0\|_{L^2}=\rho$ and such that the corresponding solution $\psi(t)$ blows up in finite time.
	\end{enumerate}
	\end{theorem}
\textbf{Remark 1.3.}
When $\lambda_1=\lambda_2=0$, there exists an initial data $\psi_0$
such that $\|\psi_0\|_{L^2}= \lambda_3^{-\frac{3}{4}}\|R\|_{L^2}$ and such that the corresponding solution $\psi(t)$ blows up in finite time, see \cite[Remark 6.7.3]{ca2003}.
When $\lambda_2>0$, i.e., in the defocusing case, we prove that all solutions of \eqref{equation partial} with $\|\psi_0\|_{L^2}= \lambda_3^{-\frac{3}{4}}\|R\|_{L^2}$ exist globally.

\begin{theorem}\label{th instability no partial}
Let $\lambda_1\geq0$, $\lambda_2,\lambda_3>0$, $\omega>0$, $\frac{4}{3}\leq p<4$, and $u$ is a solution of the minimization problem \eqref{minimization ground no}.
Then, the standing wave $\psi(t,x)=e^{i\omega t}u(x)$ is
strongly unstable in the following sense: there exists $\{\psi_{0,n}\}\subset H^1$ such that
$\psi_{0,n}\rightarrow u$ in $H^1$ as $n \rightarrow \infty$ and the corresponding solution $\psi_n(t)$ of \eqref{equation partial} with initial data $\psi_{0,n}$
blows up in finite time for any $n\geq 1$.
\end{theorem}
\textbf{Remark 1.4.}
To prove this theorem, our argument relies heavily on the property of function
$f(\lambda):=S_\omega(v_\lambda)$ with $v_\lambda(x):=\lambda^{\frac{3}{2}}v(\lambda x)$. That is, $f(\lambda)$ has an unique critical point on $(0,\infty)$, which is the maximum point of $f(\lambda)$ on $(0,\infty)$ in the mass-supercritical case, see Lemma \ref{lemma jihexingzhi}. In fact, many known
physical models enjoy this property. For example, NLS or Davey-Stewartson system with combined focusing mass-supercritical and defocusing mass-subcritical nonlinearities (see \cite{ohta95die}), the mixed
dispersion NLS (see \cite{jean-goutrans,jean-gouinstablility}) and so on.
Moreover, we think that our method is robust and can be applied to many other models, only if the corresponding function $f(\lambda)=S_\omega(v_\lambda)$ has an unique critical point on $(0,\infty)$.
In particular, the idea of this proof can be also applied to prove an existence result of normalized
solutions in the mass-supercritical case, see the following Theorem \ref{theorem existence normalized}, which shares a common character that the scaling function $f(\lambda)=S_\omega(v_\lambda)$ has an unique critical point of maximum type on $(0,\infty)$. This indicates that there is a close connection between the study of the strong instability and the one of the existence of normalized solutions.

Next, we consider the existence and properties of solutions to \eqref{elliptic partial12} having prescribed $L^2$-norm. That is, for any given $c>0$, we study solutions of \eqref{elliptic partial12} satisfying the
$L^2$-norm constraint
\begin{equation}\label{mass constraint}
S_r(c):=\{u\in H_r^1:~\|u\|_{L^2}^2=c\},\quad c>0.
\end{equation}
Physically, such solutions are called normalized solutions of \eqref{elliptic partial12}, which formally
corresponds to critical points of the energy functional $E(u)$ restricted on $S_r(c)$.
In particular, in this situation, the frequency $\omega\in \mathbb{R}$ is an unknown part, which appears as the associated Lagrange multiplier.

In the mass-subcritical case, i.e. $0<p<\frac{4}{3}$, $E(u)$ is bounded from
below on $S_r(c)$. Thus, for every $c>0$, ground states can be found as global minimizers of $E|_{S_r(c)}$, see \cite{licma}. In the mass-supercritical case, i.e. $\frac{4}{3}<p<4$, on the contrary, $E|_{S_r(c)}$ is unbounded from below for any $c>0$.
 For this reason, it is unlikely to obtain a solution
to \eqref{elliptic partial12}-\eqref{mass constraint} by developing a global minimizing problem. Instead, we minimize the functional on a Pohozaev manifold, which is indeed a submanifold of $S_r(c)$. On this submanifold, $E(u)$ is bounded from below and coercive, and then we look for minimizers of $E(u)$ on it, to obtain the existence of a ground state. This idea has been carried out in many cases, see for example \cite{jean-luo,J}. Here precisely, we introduce the following minimizing problem
\begin{equation}\label{minimization problem normalzied}
\gamma(c):=\inf_{u\in V_r(c)}E(u),
\end{equation}
where the constraint $V_r(c)$ is defined by
\begin{equation}\label{submanifold}
V_r(c):=\{u\in S_r(c):~Q(u)=0\}.
\end{equation}
The identity $Q(u)=0$ is the Pohozaev identity related to \eqref{elliptic partial12}, namely any solution $u$ to \eqref{elliptic partial12} must satisfy that $Q(u)=0$, see Lemma \ref{lemma pohozaev identity}.
Then the set $V_r(c)$ is called normally a Pohozaev manifold. Note that $V_r(c)$ is indeed a natural constraint in the sense that any minimizer of \eqref{minimization problem normalzied}
 is a ground state to \eqref{elliptic partial12}, where $\omega\in \mathbb{R}$ appears as the corresponding Lagrange multiplier, see Lemma \ref{lm-lagrange}.
Our results are as follows:
\begin{theorem}\label{theorem existence normalized}
Let $\lambda_1\geq 0, \lambda_2,\lambda_3>0$, and $\frac{4}{3}\leq p<4$. Then we have
\begin{enumerate}
  \item [(1)] When $\frac{4}{3}< p<4$, there exists a $c_0>0$ such that for any $c\in (0,c_0)$, we have that $\gamma(c)>0$ and $\gamma(c)$ has a minimizer $u_c\in V_r(c)$. In particular, there exists a $\omega_c>0$ such that $u_c$ is a radial solution to \eqref{elliptic partial12} with $\omega=\omega_c$. Moreover, we have
       $$\|\nabla u_c\|_{L^2}\rightarrow +\infty,~~\omega_c\rightarrow +\infty,~~E(u_c)=\gamma(c)\rightarrow +\infty,\quad \mbox{as $c\rightarrow 0^+$};$$
  \item [(2)] When $p=\frac{4}{3}$, then for all $c\in \Big(\lambda_3^{-\frac{3}{2}}\|R\|_{L^2}^2, (\frac{9}{7})^{\frac{3}{2}}\lambda_3^{-\frac{3}{2}}\|R\|_{L^2}^2\Big)$, we have that $\gamma(c)>0$ and $\gamma(c)$ has a minimizer $u_c\in V_r(c)$; for all $c\in (0, \lambda_3^{-\frac{3}{2}}\|R\|_{L^2}^2]$, the functional $E$ has no any critical point on $S_r(c)$, where $R$ is the unique ground state of \eqref{elliptic classical};
  \item [(3)] For any $u_c\in V_r(c)$ with $E(u_c)=\gamma(c)$, the standing wave $\psi(t,x)=e^{i\omega_c t}u_c(x)$ is strongly unstable.
\end{enumerate}
\end{theorem}
\textbf{Remark 1.5.} We remark that when $\lambda_1=0$, the same existence result has been obtained by Bellazzini et al \cite[Theorem 1.1]{jean-luo} for $\frac{4}{3}<p<4$, and by Ye \cite[Theorem 1.2]{Ye} for $p=\frac{4}{3}$. In \cite{jean-luo} and \cite{Ye}, the existence result was established for suitable $c>0$ , by considering a related minimization problem as \eqref{minimization problem normalzied} in the space $H^1$, and using a delicate mountain pass argument. However, compared with their proofs, we provide here a simpler proof and avoid using the mountain pass argument. In particular, we point out that though our proof is restricted on the radial space $H_r^1$, as we clarified above, when $\lambda_1=0$, the functional $E$ is invariant by translation, then the arguments can be carried out in the space $H^1$, see Proposition \ref{proposition existence} and Remark 4.5.

Checking the proof of Theorem \ref{theorem existence normalized}, one will find that we only need the following two points: the special geometry of the scaling function $f(\lambda)=E(v_\lambda)$, namely the $f(\lambda)$ has an unique critical point of maximum point at $(0,\infty)$ (see Lemma \ref{lemma jihexingzhi}), and the strict decrease property of $\gamma(c)$ (see Lemma \ref{lemma smonotonicity}). Note that the first point is
important for establishing the existence of a critical point for the functional, namely a solution, see also the case in Theorem \ref{th instability no partial}. The second one is essential to ensure that the critical point we find is indeed on the $L^2$ constraint $S_r(c)$. We believe that our method is somehow quite general and can be applied to similar minimization problems in other models, for example \cite{jean-goutrans,gou-zhang,soave}. Notice that a function $u_c\in V_r(c)$ with $E(u_c)=\gamma(c)$ is indeed a ground state, corresponding to some $\omega_c\in \mathbb{R}$, at least for $c>0$ small, $\omega_c>0$, then the strong instability in this theorem is naturally a consequence of Theorem \ref{th instability no partial}.  \\

In a second part of the paper, we analyze the situation when the magnetic potential $A\neq0$ (namely $b\neq 0$ in \eqref{equation partial}) is considered. Since the parameter $b\neq 0$ does not enter essentially into the  play in our analysis, we shall not clarify the assumption of $b\neq 0$ if $b$ appears in the sequel. To begin with, we establish the following the global existence of the solution of \eqref{equation partial}, which is prerequisite for the orbital stability.
\begin{theorem}\label{theorem global existence}
Let $\psi_0 \in X$, and $\psi(t)$ be the unique solution of \eqref{equation partial}, with the initial data $\psi(0)=\psi_0$. Then $\psi(t) $ exists globally in time, under each of the following assumptions:
\begin{enumerate}
\item[(1)]  $\lambda_1,\lambda_2\in \mathbb{R}$, $\lambda_3<0$, $0< p<4$;
\item[(2)]  $\lambda_1,\lambda_2\in \mathbb{R}$, $\lambda_3>0$, $0< p<\frac{4}{3}$;
\item[(3)]  $\lambda_1,\lambda_2\in \mathbb{R}$, $\lambda_3>0$, $ p=\frac{4}{3}$, $\|\psi_0\|_{L^2}<\lambda_3^{-\frac{3}{4}}\|R\|_{L^2}$, where $R$ is the unique ground state of \eqref{elliptic classical};
\item[(4)]  $\lambda_1,\lambda_2\in \mathbb{R}$, $\lambda_3>0$, $\frac{4}{3}<p<4$. For every given $\kappa>0$, there exists $0<c_{\kappa}<1$ such that $\|\psi_0\|_{\tilde{X}}<\kappa$ and $\|\psi_0\|_{L^2}< c_{\kappa}$, where the norm $\|u\|_{\tilde{X}}$ is defined by
    \begin{eqnarray}\label{Xnorm00}
\|u\|_{\tilde{X}}^2:=\|\nabla u\|_{L^2}^2+b^2\int_{\mathbb{R}^3} (x_1^2+x_2^2)|u|^{2}dx,\ b\in \mathbb{R}.
    \end{eqnarray}
\end{enumerate}
\end{theorem}

\textbf{Remark 1.6.}
We remark that, the global existences in Theorem \ref{theorem global existence} $(1)(2)(3)$, namely for the mass-subcritical and critical case, seem to hold reasonably by the standard way.
 However, generally in the mass-supercritical case, global existence holds
  for small initial data, see for example \cite[Remark 6.2.3]{ca2003}, or
  also \cite[Cororally 2.6]{RC}, but for large initial data,
   blow-up in finite time may possibly occur.
   Our Theorem \ref{theorem global existence} $(4)$ actually extends to
    the case that for any given $\tilde{X}$-norm of the initial data, we only
     need its mass being small. In particular, having this global existence for the time-depending solutions, one could then consider further the possible stability of solutions in the mass-supercritical case.
 Also we point out that, Theorem \ref{theorem global existence} $(4)$ covers the case $\lambda_1=\lambda_2=0, \lambda_3>0$ and $\frac{4}{3}<p<4$. Such result has already been proved in \cite[Theorem 1.1]{jeanCMP2017}, by a local minima construction of the corresponding functional. Here our proof is completely different.

In what follows, we investigate the existence of ground state standing waves and their stability and strong instability to the equation \eqref{equation partial}.

Firstly, we consider the mass-subcritical and critical cases $0<p\leq\frac{4}{3}$. By Lemma \ref{lm3.1}, under suitable assumptions on $p$ and $c>0$, $E_b(u)$ is bounded from
below on
\begin{equation}\label{mass constraint1}
\bar{S}(c)=\{u\in X:~\|u\|_{L^2}^2=c\}.
\end{equation}
Therefore, we can study the existence and stability of standing waves for \eqref{equation partial} by considering the following global minimization problem:
\begin{equation}\label{minimization subcritical}
m(c):=\inf_{u\in \bar{S}(c)}E_b(u),
\end{equation}
where $E_b(u)$ is defined by \eqref{energyE}. Denote the set of all minimizers of $m(c)$ by
$$\mathcal{M}_c:=\{u\in \bar{S}(c):~ E_b(u)=m(c) \}.$$
It is standard that for any $u_c\in \mathcal{M}_c$, there exists a $\omega_c \in  \mathbb{R}$, such that $( u_c, \omega_c )$ solves the stationary equation \eqref{elliptic partial}.
Our results are as follows:
\begin{theorem}\label{Theorem globalmini}
Assume that $\lambda_1\leq0, \lambda_2\leq 0, \lambda_3>0$. Let $R$ be the unique ground state of \eqref{elliptic classical}. Then for any $c>0$ if $0<p<\frac{4}{3}$ or $0<c<\lambda_3^{-\frac{3}{2}}\|R\|_{L^2}^2$ if $p=\frac{4}{3}$, $\mathcal{M}_c\neq \emptyset$, and is orbitally stable in the sense that: for any given $\varepsilon>0$, there exists a $\delta>0$ such that for any initial data $\psi_0$ satisfying
$$\inf_{u\in \mathcal{M}_c}\|\psi_0-u\|_{X}<\delta,$$
the corresponding solution $\psi(t)$ of \eqref{equation partial} satisfies
$$\inf_{u\in \mathcal{M}_c}\|\psi(t)-u\|_{X}<\varepsilon,\ \forall t\geq0.$$
\end{theorem}
\textbf{Remark 1.7.} We can prove that if $\lambda_1>0, \lambda_2\leq 0, \lambda_3>0$, then for any $c>0$ if $0<p<\frac{4}{3}$ or $0<c<\lambda_3^{-\frac{3}{2}}\|R\|_{L^2}^2$ if $p=\frac{4}{3}$, $\mathcal{M}_c=\emptyset$, see Theorem \ref{th-nonexistence1}. In this process, since the energy functional does not keep invariant by translation, we need to compare with the limit minimization problem (i.e. $\lambda_1=0$):
\begin{eqnarray}\label{limitmini}
m^{\infty}(c):=\inf_{u\in \Bar{S}(c)}E^{\infty}_{b}(u),
\end{eqnarray}
where the limit energy $E^{\infty}_{b}$ is given by
\begin{eqnarray}\label{limitenergy}
E^{\infty}_{b}(u):=\frac{1}{2}\|\nabla u\|_{L^2}^2+\frac{b^2}{2}\int_{\mathbb{R}^3} (x_1^2+x_2^2)|u|^{2}dx+\frac{\lambda_2}{4}\int_{\mathbb{R}^3} (|x|^{-1}\ast|u|^{2})|u|^{2}dx -\frac{\lambda_3}{p+2}\|u\|_{L^{p+2}}^{p+2}.
\end{eqnarray}
To prove this theorem on both the existence and orbital stability, the key is to show that any minimizing sequence of $m(c)$ is pre-compact. For this aim, we first prove the pre-compact property of a minimizing sequence of $m^{\infty}(c)$. Note that when $\lambda_1=0$, the functional $E_b$ is invariant by a translation at the $x_3$ direction. By Lemma \ref{lemma Jeanjean} we prove the non-vanishing of a minimizing sequence, and then to show its compactness, we provide new arguments, which are basically the concentration compactness principle, but compared with its standard form, ours are based on the Brezis-Lieb lemma, which look more direct and simpler. See the details in the proofs of Proposition \ref{prop-limit1}. To treat the case $\lambda_1<0$, based on the result obtained in $\lambda_1=0$, we establish in Lemma \ref{lm-limit-nonlimit} the following two strict inequalities:
\begin{eqnarray*}
m(c)<m^{\infty}(c),
\end{eqnarray*}
and
\begin{eqnarray*}
m(c)<m(c-\mu)+m(\mu),\quad \forall \mu\in (0,c).
\end{eqnarray*}
The former helps to remove the vanishing of a minimizing sequence and the latter to remove the dichotomy, then we finally reach the aim. See Proposition \ref{prop-globalmini}.
In addition, we point out that our proof can not be directly applied on the case $\lambda_2>0$, which is indeed more complicated to treat. Recall that when $\lambda_1=0,\lambda_2>0$ in the equation \eqref{elliptic partial12} (without the partial confine), existence only was established under different conditions on $c>0$ and $p>0$, precisely, existence holds for $c>0$ sufficiently small if $0<p<1$, and if $1<p<\frac{4}{3}$, there exists a $c_0>0$ which separates sharply the range of $c$  for the existence and non-existence (see \cite{BS} and \cite{JL}). If a partial confine is involved additionally, then the situation will become more complicated. Hence we shall not pay more attention on this direction in this paper.

In the mass-critical and supercritical cases $\frac{4}{3}\leq p<4$, we observe that though the functional $E_b$ is unbounded from below on $\bar{S}(c)$, due to the partial harmonic potential, the functional $E_b$ apparently has a concave-convex construct on $\bar{S}(c)$. Similar situation is verified by Bellazzini et al \cite{jeanCMP2017} where the case $\lambda_1=\lambda_2=0$ in \eqref{elliptic partial} is studied. Hence, in order to construct a stable standing wave in this case, inspirited by the work in \cite{jeanCMP2017}, we consider the following local minimization problem:
for any given $r>0$, defining
\begin{eqnarray}\label{localmini}
\bar{m}_c^r:=\inf_{u\in \bar{S}(c)\cap B(r)}E_b(u),\ c>0,
\end{eqnarray}
where $B(r):=\{u\in X:~ \|u\|_{\tilde{X}}^2\leq r \}$, and $\|\cdot\|_{\tilde{X}}$ is given in \eqref{Xnorm00}.

It can be proved that for any given $r>r_0$ with some $r_0>0$, there exists a $c_r>0$ , such that $\bar{S}(c)\cap B(r)\neq \emptyset, \forall c<c_r$, see Lemma \ref{lmwelldefine}. Also it is obvious that $\bar{m}_c^r>-\infty$. Thus we can expect the existence of a minimizer of $\bar{m}_c^r$. Denote the set of all minimizers of $\bar{m}_c^r$ by
$$\mathcal{M}_c^r:=\{u\in \bar{S}(c)\cap B(r):~ E_b(u)=\bar{m}_c^r \}.$$
 Then we prove that

\begin{theorem}\label{th-localmini}
Assume that $\lambda_1\leq 0, \lambda_2\leq 0, \lambda_3>0$ and $\frac{4}{3}\leq p<4$, all being fixed, then there exists a $r_0>0$, such that for every given $r>r_0$, there exists a $c_r$ with $0<c_r<1$, we have for any $c\in (0,c_r)$ that,
\begin{enumerate}
  \item [(i)] $\emptyset \neq \mathcal{M}_c^r \subset \bar{S}(c)\cap B(\frac{rc}{2})$;
  \item [(ii)] The set $\mathcal{M}_c^r$ is orbitally stable in the sense that: for any given $\varepsilon>0$, there exists $\delta>0$ such that for any initial data $\psi_0$ satisfying
$$\inf_{u\in \mathcal{M}_c^r}\|\psi_0-u\|_{X}<\delta,$$
the corresponding solution $\psi(t)$ of \eqref{equation partial} satisfies
$$\inf_{u\in \mathcal{M}_c^r}\|\psi(t)-u\|_{X}<\varepsilon,\ \forall t>0.$$
\end{enumerate}

\end{theorem}

\textbf{Remark 1.8.} We note that the stable result obtained here and also the one in \cite{jeanCMP2017,jean16siam} are surprise to us at the first view, since normally without the partial harmonic potential, ground states in the mass-supercritical case are strongly unstable by blow-up, see Theorem \ref{th instability no partial}. In such a sense, it seems that it is the partial harmonic potential that causes the ground states to be stable, which is also observed in \cite[Remark 1.6]{jean16siam}. By the global existence result established in Theorem \ref{theorem global existence} (4), we see that it is reasonable to study the existence of stable standing waves at least in $  \bar{S}(c)\cap B(r)$ with $c$ small. To prove this theorem, our proof is mainly inspirited by \cite{jeanCMP2017}. However, due to the appearance of the Coulomb potential and Hartree-type nonlinearity, our analyses are more subtle. Precisely, we first show a local minima structure and then show that for $c>0$ basically small, any minimizing sequence of $\bar{m}_c^r$ is pre-compact. By this essential preliminary we then finally prove both the existence and stability. We remark that to prove that the pre-compactness of a minimizing sequence of $\bar{m}_c^r$, we use the same process as that in the proof of Theorem \ref{Theorem globalmini}, and correspondingly we need to consider and do some compare with the following limit problem:
\begin{eqnarray}\label{limitlocmini}
(\bar{m}_{c}^r)^{\infty}:=\inf_{u\in \bar{S}(c)\cap B(r)}E^{\infty}_{b}(u),
\end{eqnarray}
see Proposition \ref{prop-localmini1}.

In addition, we have proved in Lemma \ref{lm-limit-nonlimit02} that if $\lambda_1> 0, \lambda_2\leq 0, \lambda_3>0$ and $\frac{4}{3}\leq p<4$, then $\mathcal{M}_c^r=\emptyset$. We also remark that for the case $\lambda_2>0$, we could prove a local minima structure, see Lemma \ref{lmwelldefine1}. But to prove the compactness of an arbitrary minimizing sequence of $\bar{m}_c^r$ is complicated. The reason is the same as we clarify in Remark 1.6.

\textbf{Remark 1.9.}
Compared with Theorem \ref{Theorem globalmini} and Theorem \ref{th-localmini} in the common case: $\lambda_1\leq 0, \lambda_2 \leq 0, \lambda_3>0$ and $p=\frac{4}{3}$, one will find that for $c>0$ small enough, there exist a global minimizer of $E_b$ on $S_r(c)$ by Theorem \ref{Theorem globalmini}, and also a local minimizer by Theorem \ref{th-localmini}, then an interesting question arises whether this local minimizer is a global one or not. We believe that the answer is positive, but for the moment, we do not have a convinced proof for that.\\

Finally,  we consider the existence and the strong instability of ground state standing waves for \eqref{equation partial} in the mass-critical and supercritical cases.
For this purpose, we need to establish the variational characterization of ground states related to \eqref{elliptic partial12} on a certain manifold. Due to the appearance of a partial harmonic potential, it is obvious that $f(\lambda)=S_{b,\omega}(v_\lambda)\rightarrow +\infty$, as $\lambda \rightarrow 0^+$. Thus, there is no maximum point of  $f(\lambda)$ on $(0,\infty)$. It is hard to establish the variational characterization of ground states on the Pohozaev manifold $\{v\in X,~~Q_b(v)=0\}$,
where $Q_b(v)$ is the Pohozaev functional of \eqref{elliptic partial} defined by
\begin{align}\label{qu partial}
Q_b(v):=&\partial_\lambda S_{b,\omega}(v_{\lambda})|_{\lambda=1}=\|\nabla v\|_{L^2}^2-b^2\int_{\mathbb{R}^3} (x_1^2+x_2^2)|v|^{2}dx\nonumber\\&+\frac{\lambda_1}{2}\int_{\mathbb{R}^3} \frac{|v|^{2}}{|x|}dx+\frac{\lambda_2}{4}\int_{\mathbb{R}^3} (|x|^{-1}\ast|v|^{2})|v|^{2}dx-\frac{3\lambda_3p}{2(p+2)}\|v\|_{L^{p+2}}^{p+2}.
\end{align}
On the other hand, due to the appearance of Hartree nonlinearity, the function $h(\lambda)=S_{b,\omega}(\lambda v)$ maybe have two critical points. Therefore, the usual Nehari manifold does not happen to be a good choice, see also \cite{ruiz}.

To overcome this difficulty, we must find a new manifold, on which we manage to establish the variational characterization of ground states. Now let us show the choice of the new manifold. Suppose that $u$ is a critical point of $S_{b,\omega}$, setting $u^\lambda(x):=\lambda^{3}u(\lambda x), \lambda>0$, we have
\begin{align}\label{suscalings}
f_1(\lambda):=S_{b,\omega}(u^\lambda)=& \frac{\lambda^5}{2}\|\nabla u\|_{L^2}^2+\frac{\omega\lambda^3}{2}\int_{\mathbb{R}^3} |u|^{2}dx+\frac{\lambda b^2}{2}\int_{\mathbb{R}^3} (x_1^2+x_2^2)|u|^{2}dx+\frac{\lambda_1\lambda^4}{2}\int_{\mathbb{R}^3} \frac{|u|^{2}}{|x|}dx
\nonumber\\&+\frac{\lambda_2\lambda^7}{4}\int_{\mathbb{R}^3} (|x|^{-1}\ast|u|^{2})|u|^{2}dx -\frac{\lambda_3\lambda^{3p+3}}{p+2}\|u\|_{L^{p+2}}^{p+2}.
\end{align}
When $\lambda_1,\lambda_2\geq0$ and $\lambda_3>0$, one
 can checked easily that this new scaling function $f_1(\lambda)$ has an unique critical point $\lambda=1$ of maximum type on $(0,\infty)$. In particular, $f_1'(1)=0$ can be written as
\begin{align}\label{ku}
K_{b,\omega}(u):&=\partial_\lambda S_\omega(u^\lambda)|_{\lambda=1}= \frac{5}{2}\|\nabla u\|_{L^2}^2+\frac{3\omega}{2}\int_{\mathbb{R}^3} |u|^{2}dx+\frac{b^2}{2}\int_{\mathbb{R}^3} (x_1^2+x_2^2)|u|^{2}dx
\nonumber\\&+2\lambda_1\int_{\mathbb{R}^3} \frac{|u|^{2}}{|x|}dx+\frac{7\lambda_2}{4}\int_{\mathbb{R}^3} (|x|^{-1}\ast|u|^{2})|u|^{2}dx -\frac{\lambda_3(3p+3)}{p+2}\|u\|_{L^{p+2}}^{p+2}.
\end{align}
In fact, $K_{b,\omega}(u)=3I_{b,\omega}(u)-J_{b,\omega}(u)$, where $I_{b,\omega}(u)$ and $J_{b,\omega}(u)$ are defined by \eqref{pohozaev identity0} and \eqref{pohozaev identity1}, respectively.
Hence, any solution $u$ of \eqref{elliptic partial} satisfies $K_{b,\omega}(u)=0$.
 We can obtain the following variational characterization of ground states related to
 \eqref{elliptic partial}.
\begin{proposition}\label{proposition ground state partial}
Let $\lambda_1=0$, $\lambda_2\geq 0$, $\lambda_3>0$, $\omega>0$ and $\frac{4}{3}\leq p<4$. Then, $u$ is a ground
state related to \eqref{elliptic partial} if and only if $u$ solves the following minimization problem:
\begin{equation}\label{minimization ground state partial}
d_K(\omega):=\inf\{S_{b,\omega}(v):~v\in X\backslash \{0\},~~K_{b,\omega}(v)= 0\}.
\end{equation}
\end{proposition}
\textbf{Remark 1.10.} When $\lambda_1>0$,  by a similar argument as Theorem \ref{th-nonexistence1} or Lemma \ref{lm-limit-nonlimit02}, we can prove that $d_K(\omega)$ has no any minimizers. When $\lambda_1<0$, by making use of the equivalence of norms
\[
\|\nabla u\|_{L^2}^2+b^2\int_{\mathbb{R}^3}(x_1^2+x_2^2)|u|^2dx+\omega\int_{\mathbb{R}^3} |u|^{2}dx-|\lambda_1|\int_{\mathbb{R}^3} \frac{|u|^{2}}{|x|}dx\cong \|u\|_{X}^2,
\]
where $\omega>\omega_0$ and $\omega_0$ is the smallest eigenvalue of the Schr\"{o}dinger operator $-\Delta-|\lambda_1||x|^{-1}$, we can establish the variational characterization of ground states related to \eqref{elliptic partial}, see also \cite[Proposition 1.1]{ohta18inversepotential}. However, in this case, we can not establish the estimate \eqref{keyestimate partial} in Lemma \ref{lemma keyestimate partial}, which is the key to derive the strong instability. Therefore, we assume $\lambda_1=0$ in this section.

Moreover, we remark that since the embedding $X_r:=\{v\in X\ :\ v(x_1,x_2,x_3)=v(x_1,x_2,|x_3|)\}\hookrightarrow L^q$ for some $q\in (2,6)$ is not compact, a bounded sequence $\{u_n\}$ in $X_r$ satisfying
$\liminf_{n\rightarrow \infty}\|u_n\|_{L^{p+2}}\geq \delta$
for some $\delta>0$, could possibly converge weakly to zero. Then it is hard to establish the existence of a ground state by solving the following minimization problem
\[
d_{r,K}(\omega):=\inf\{S_{b,\omega}(v):~v\in X_r\backslash \{0\},~~K_{b,\omega}(v)= 0\}.
\]
Also due to the Hartree nonlinearity, the technique of symmetric arrangement could not applied. To solve the minimization problem \eqref{minimization ground state partial}, noticing that when $\lambda_1=0$, the action functional $S_{b,\omega}$ is invariant by translation at the $x_3$ direction, then we could make use of the idea of proof in Proposition \ref{proposition ground state no} to complete the proof.

\textbf{Remark 1.11.} When we restrict on $2\leq p<4$, the function $g(\lambda)=S_{b,\omega}(\lambda v)$ indeed has an unique critical point on $(0,\infty)$, by which one could establish a variational characterization of ground states on Nehari manifold by using our method.

Next, we consider the strong instability of standing waves for \eqref{equation partial} with $b\neq 0$ and $\lambda_1=0$. In this case, the function $f(\lambda)=S_{b,\omega}(v_\lambda)$ with $v_\lambda(x):=\lambda^{\frac{3}{2}}v(\lambda x)$ does not satisfy the properties in Lemma \ref{lemma jihexingzhi}. Recently, the strong instability for other kinds of NLS has been investigated in similar case, see \cite{f20na,gancmp,ohta03die,ohta18doublepower,ohta18inversepotential,
ohta18cpaa,ohtaharmonic,
ohta95die,zhangcpde}. Moreover, the action functionals in these papers only include four terms. Therefore, the Pohozaev functionals and functionals $\partial_\lambda^2S_{\omega}(v_\lambda)|_{\lambda=1}$ in these papers only include three terms. This implies that any two of three terms can be calculated by the remaining term. Based on this fact and the assumption $\partial_\lambda^2S_{\omega}(v_\lambda)|_{\lambda=1}\leq 1$, Fukuizumi and Ohta in \cite{ohta18doublepower,ohta18inversepotential,ohtaharmonic} studied the strong instability of standing waves. Since the action functional  of \eqref{equation partial} with $\lambda_1=0$ includes five terms, the methods of Fukuizumi and Ohta  cannot work for \eqref{equation partial}.

 Under some assumptions, Zhang in \cite{zhangcpde} studied the strong instability for NLS with a harmonic potential by establishing so-called cross-invariant manifolds, see also \cite{gancmp}. The arguments in \cite{gancmp,zhangcpde} rely heavily on the fact  that the function  $g(\lambda):=S_{\omega}(\lambda v)$ has an unique critical point on $(0,\infty)$.
When $2\leq p<4$, equation \eqref{equation partial} also enjoys this property.
Therefore, we will mainly apply the ideas in \cite{gancmp,zhangcpde} to study the strong instability. However, due to the complexity
of our functional, more difficulties have to be overcame here.
The key to the idea is to construct a type of cross-constrained minimization problem, and establish so-called cross-invariant manifolds, which require us to search for proper functionals. In \cite{gancmp,zhangcpde}, they used the action, Pohozaev and Nehari functionals, and every functional only includes four terms. But in the present paper, we use the action, Pohozaev and Pohozaev-Nehari functionals, and
every functional includes five terms with different scaling rates. In \cite{gancmp}, two nonlinearities are focusing and mass-supercritical. But for \eqref{equation partial}, one nonlinearity is defocusing and mass-subcritical, the other is focusing and mass-supercritical. For these reasons, compared to
the argument in \cite{gancmp,zhangcpde}, our discussion will be more complicated and skillful.

When $\frac{4}{3}\leq p<2$,
$g(\lambda):=S_{b,\omega}(\lambda v)$ maybe have two critical points on $(0,\infty)$.
Therefore, the above argument cannot work. Let $u$ be the ground state of \eqref{elliptic partial}, in order to study the strong instability of the standing wave $e^{i\omega t}u(x)$, in view of definition, we need to construct a sequence of initial data $\{\psi_{0,n}\}\subset X$ such that $\psi_{0,n}\rightarrow u$ in $X$ as $n\rightarrow \infty$, and such that the solution $\psi_n(t)$ with initial data $\{\psi_{0,n}\}$ blows up in finite time. Generally, a sequence of initial data $\{\psi_{0,n}\}\subset X$ can be chosen by $\psi_{0,n}=\lambda_n^cu(\lambda_n^dx)$ for some $\lambda_n>1$ and $\lambda_n\rightarrow 1$ as $n\rightarrow \infty$. In order to proving blow-up, by using Lemma \ref{lemma blowup}, we need $Q_b(\psi_{0,n})=Q_b(\lambda_n^cu(\lambda_n^dx))<0$.  But, for equation \eqref{equation partial} with $b\neq0$, it is easy to check that $Q_b(\lambda_n^cu(\lambda_n^dx))$ with $c\neq0$ and $d\in \mathbb{R}$ may be larger than zero. Therefore, motivated by the idea in \cite{ohta03die,ohta18doublepower,ohta18inversepotential,ohta18cpaa,ohtaharmonic}, we
overcome this difficulty by assuming that $\lambda=1$ is the local maximum point of $f(\lambda)$, i.e., $\partial_\lambda^2S_{b,\omega}(u_\lambda)|_{\lambda=1}\leq 0$, which implies that $Q_b(u_{\lambda})<0$ and $S_{b,\omega}(u_\lambda)<S_{b,\omega}(u)$ for all $\lambda>1$. Under this assumption, we can prove the strong instability of standing waves for \eqref{equation partial}.

In order to study the strong instability of standing waves for \eqref{equation partial}, we introduce a cross-manifold in $X$ as follows:
\begin{equation}\label{cross-manifold}
\mathcal{N}:=\{v\in X:~~K_{b,\omega}(v)<0~~\mbox{and}~~Q_b(v)=0\}.
\end{equation}
When $2\leq p<4$, or when $\frac{4}{3}\leq p<2$ additionally assume that $\partial_\lambda^2S_{b,\omega}(u_\lambda)|_{\lambda=1}\leq 0$ with $u_\lambda=\lambda^{\frac{3}{2}}u(\lambda x)$ and $u$ being a ground state of \eqref{elliptic partial}, we can prove that the cross-manifold $\mathcal{N}$ is not empty, see Lemma \ref{cross-manifold feikong}. Furthermore, we introduce the following cross-constrained minimization problem
\begin{equation}\label{minimization cross-manifold}
\alpha(\omega):=\inf\{S_{b,\omega}(v):~v\in \mathcal{N}\},
\end{equation}
and define the set:
\begin{equation}\label{manifold K}
\mathcal{K}_\omega:=\inf\{v\in X:~~S_{b,\omega}(v)<S_{b,\omega}(u),~~K_{b,\omega}(v)<0~~\mbox{and}~~Q_b(v)<0\}.
\end{equation}
We can prove that
\begin{lemma}\label{lemma keyestimate partial}
Let $\lambda_1=0$, $\lambda_2\geq 0$, $\lambda_3>0$, $\frac{4}{3}\leq p<4$. Assume that there exists $\omega>0$ such that $\alpha(\omega)\geq S_{b,\omega}(u)$, with $u$ being a ground state related to \eqref{elliptic partial}. Then when $2\leq p<4$, or when $\frac{4}{3}\leq p<2$ additionally assume that $\partial_\lambda^2S_{b,\omega}(u_\lambda)|_{\lambda=1}\leq 0$ with $u_\lambda=\lambda^{\frac{3}{2}}u(\lambda x)$, the set $\mathcal{K}_\omega$ is invariant under the flow of \eqref{equation partial}, and
\begin{equation}\label{keyestimate partial}
Q_b(\psi(t))\leq 2(S_{b,\omega}(\psi_0)-S_{b,\omega}(u)),
\end{equation}
for any $t\in [0,T^*)$.
\end{lemma}
Based on the key estimate \eqref{keyestimate partial}, we obtain finally the strong instability.
\begin{theorem}\label{Theorem instability partial}
Let $\lambda_1=0$, $\lambda_2\geq 0$, $\lambda_3>0$, $\frac{4}{3}\leq p<4$. Assume that there exists $\omega>0$ such that $\alpha(\omega)\geq S_{b,\omega}(u)$ with $u$ being a ground state related to \eqref{elliptic partial}. Then when $2\leq p<4$, or when $\frac{4}{3}\leq p<2$ additionally assume that $\partial_\lambda^2S_{b,\omega}(u_\lambda)|_{\lambda=1}\leq 0$, we have that the standing wave $\psi(t,x)=e^{i\omega t}u(x)$ is
strongly unstable in the following sense: there exists $\{\psi_{0,n}\}\subset X$ such that
$\psi_{0,n}\rightarrow u$ in $X$ as $n \rightarrow \infty$ and the corresponding solution $\psi_n$ of \eqref{equation partial} with initial data $\psi_{0,n}$
blows up in finite time for any $n\geq 1$.
\end{theorem}
\textbf{Remark 1.12.} Note that the strong instability of ground state standing waves of \eqref{equation partial} is established under the assumption that for some $\omega>0$, $\alpha(\omega)\geq S_{b,\omega}(u)$. However, the condition $\alpha(\omega)\geq S_{b,\omega}(u)$ is still vague, even in the simple case $\lambda_1=\lambda_2=0$, see \cite[Remark 5.1]{zhangcpde}. We need further to determine for which $\omega$,  $\alpha(\omega)\geq S_{b,\omega}(u)$ holds.
Moreover, it is also open that standing waves are orbital stable or not if $\alpha(\omega)< S_{b,\omega}(u)$.

\textbf{Remark 1.13.} When $b\neq 0$, $\lambda_1<0$ or $\lambda_2<0$, or even if more weakly that $b=0$, $\lambda_1<0$ or $\lambda_2<0$, some of our arguments in the proof of Theorem \ref{Theorem instability partial} do not work any more, see the proof of Lemma \ref{lemma keyestimate partial}. To the best of our knowledge, the strong instability of standing waves under these situations is still an open question. Anyway the physically interesting coefficients for \eqref{equation partial}, that is $\lambda_1,\lambda_2,\lambda_3\geq0$, are included in Theorems \ref{th instability no partial} and \ref{Theorem instability partial}. For the mathematical interest, the case of $\lambda_1<0$ or $\lambda_2<0$ will be the object of a future investigation.\\

This paper is organized as follows: in Section 2, we will collect
some lemmas such as the local well-posedness theory of \eqref{equation partial}, the Brezis-Lieb lemma, a compactness lemma. In section 3, we will prove Proposition \ref{proposition ground state no}, Theorems \ref{theorem sharp}, \ref{theorem sharp criteria critical} and \ref{th instability no partial}. In section 4, we will prove Theorem \ref{theorem existence normalized}.
In section 5, we will prove Theorems \ref{theorem global existence}, \ref{Theorem globalmini} and \ref{th-localmini}. In section 6, we will prove Proposition \ref{proposition ground state partial}, Lemma \ref{lemma keyestimate partial}, Theorem \ref{Theorem instability partial}.

\section{Preliminaries}
In this section, we recall some preliminary results that
will be used later. Firstly, let us recall the local theory for the Cauchy problem \eqref{equation partial}, see Theorem 9.1.5 in \cite{ca2003} or Theorem 3.1 in \cite{f18dcds}.
\begin{proposition}\cite[Theorem 9.1.5]{ca2003}\label{pro2.1}
Let $\psi_0 \in X$, $\lambda_1,\lambda_2,\lambda_3\in \mathbb{R}$, $0<p<4$. Then, there exists $T = T(\|\psi_0
\|_{X})$ such that \eqref{elliptic partial} admits a unique solution $\psi\in C([0,T],X)$. Let $[0,T^{\ast })$ be the maximal time interval on which the
solution $\psi$ is well-defined, if $T^{\ast }< \infty $, then $\|
 \psi(t)\| _{X}\rightarrow \infty $ as $t\uparrow T^{\ast } $. Moreover, for all $0\leq t<T^*$, the solution
$\psi(t)$ satisfies the following conservation of mass and energy
\[
\|\psi(t)\|_{L^2}=\|\psi_0\|_{L^2},
\]
\[
E_b(\psi(t))=E_b(\psi_0 ),
\]
 where $E_b(\psi(t))$ is defined by \eqref{energyE}.
\end{proposition}

In order to study the strong instability of standing waves,
 we need to prove the existence of blow-up solutions. Following the classical convexity method of Glassey,
 by some formal computations (which are made rigorous in \cite{ca2003}), we can obtain the following lemma.
\begin{lemma}\cite[Remark 9.2.10]{ca2003}\label{lemma blowup}
Let $0< p<4$ and
 $\psi_0 \in \Sigma:=  \{v\in H^1~and~|x|v \in L^2\}$. Then, the corresponding solution $\psi(t)\in \Sigma$ for all $t\in [0,T^*)$ and the function $J(t)$ belongs to $C^2[0,T^*)$, where
 $J(t)=\int_{\mathbb{R}^3} |x|^2 |\psi(t,x)|^2dx$. Furthermore, we have
\[
J^{\prime\prime}(t)=8Q_b(\psi(t)),
\]
for all $t\in [0,T^*)$,
where $Q_b(u)$ is defined by \eqref{qu partial}.
\end{lemma}

In this paper,
we need the so-called the Brezis-Lieb lemma, see \cite{blieb}.
\begin{lemma}\cite{blieb}\label{lemma Brezis-Lieb}
 Let $0<p<\infty$. Suppose that $f_n\rightarrow f$ almost everywhere
and $\{f_n\}$ is a bounded sequence in $L^p$, then
\[
\lim_{n\rightarrow \infty}(\|f_n\|_{L^{p}}^{p}-\|f_n-f\|_{L^{p}}^{p}-\|f\|_{L^{p}}^{p})=0.
\]
\end{lemma}
\begin{lemma}\label{lemma Brezis-Lieb lemma Hartree}\cite[Lemma 7.2]{Kikuchi2007}
 Let $f\in H^1$. Suppose that $f_n\rightharpoonup f$ in $H^1$, then
\[
\lim_{n\rightarrow \infty}(\int_{\mathbb{R}^3} (|x|^{-1}\ast|f_n|^{2})|f_n|^{2}dx-\int_{\mathbb{R}^3} (|x|^{-1}\ast|f_n-f|^{2})|f_n-f|^{2}dx-\int_{\mathbb{R}^3} (|x|^{-1}\ast|f|^{2})|f|^{2}dx)=0.
\]
\end{lemma}
The following compactness lemma is vital in our discussion,
\begin{lemma}\cite[Lemma 3.4]{jeanCMP2017}\label{lemma compactness lemma I}
Assume that a sequence $\{u_n\}$ satisfies that $\sup\limits_{n}\|u_n\|_{X}<\infty$, and
$$\|u_n\|_{L^{p+2}}\geq \delta,\  \forall n\in \mathbb{N},$$
for some constant $\delta>0$,  then there exist a sequence $\{z_n\}\subset \mathbb{R}$ and $\bar{u}\in X\backslash \{0\}$, such that
$$u_n(x_1,x_2, x_3-z_n) \rightharpoonup \bar{u}, \ \mbox{in } X. $$
\end{lemma}
Finally, we recall the following Pohozaev identity related to \eqref{elliptic partial}.
\begin{lemma}\label{lemma pohozaev identity}
If $u\in X$ is a solution of \eqref{elliptic partial}, then the following identities hold:
\begin{align}\label{pohozaev identity0}
I_{b,\omega}(u):=&\|\nabla u\|_{L^2}^2+\omega\|u\|_{L^2}^2+b^2\int_{\mathbb{R}^3}(x^2_1+x^2_2)|u|^2dx+ \lambda_1\int_{\mathbb{R}^3}\frac{|u|^2}{|x|}dx\nonumber\\& +\lambda_2\int_{\mathbb{R}^3} (|x|^{-1}\ast|u|^{2})|u|^{2}dx -\lambda_3\|u\|_{L^{p+2}}^{p+2}=0,
\end{align}
and
\begin{align}\label{pohozaev identity1}
 J_{b,\omega}(u):=&\frac{1}{2}\|\nabla u\|_{L^2}^2+\frac{3\omega}{2}\|u\|_{L^2}^2+ \frac{5b^2}{2}\int_{\mathbb{R}^3}(x^2_1+x^2_2)|u|^2dx+\lambda_1\int_{\mathbb{R}^3}\frac{|u|^2}{|x|}dx \nonumber\\&  +\frac{5\lambda_2}{4}\int_{\mathbb{R}^3} (|x|^{-1}\ast|u|^{2})|u|^{2}dx -\frac{3\lambda_3}{p+2}\|u\|_{L^{p+2}}^{p+2}=0.
\end{align}
Moreover, $Q_b(u)=\frac{3}{2}I_{b,\omega}(u)-J_{b,\omega}(u)=0$.
\end{lemma}
\begin{lemma}\cite[Lemma 2.1]{jeanCMP2017}\label{lemma Jeanjean}
Define
\begin{equation}\label{Lambda_0}
\Lambda_0:=\inf_{\int_{\mathbb{R}^3}|u|^2dx=1}
\int_{\mathbb{R}^3}\Big( |\nabla u|^2+b^2(x_1^2+x_2^2)|u|^{2}\Big)dx
\end{equation}
and
\begin{equation}\label{lambda_0}
\lambda_0:=\inf_{\int_{\mathbb{R}^2}|v|^2dx'=1}
\int_{\mathbb{R}^2}\Big(|\partial_{x_1}v|^2+|\partial_{x_2}v|^2+b^2(x_1^2+x_2^2)|v|^{2}\Big)dx'.
\end{equation}
Then $\Lambda_0=\lambda_0$.
\end{lemma}

\begin{lemma}\cite[Lemma 2.7]{RC}\label{lemma Bootstrap}
(Bootstrap argument). Let M(t) be a nonnegative
continuous function on [0,T] such that, for every $t\in [0,T]$,
\[
M(t)\leq \alpha+\beta M(t)^\theta ,
\]
where $\alpha, \beta>0$ and $\theta >1$ are constants such that
\[
\alpha<\left(1-\frac{1}{\theta }\right)\frac{1}{(\theta \beta)^{1/(\theta
-1)}},~~~M(0)\leq \frac{1}{(\theta \beta)^{1/(\theta -1)}}.
\]
Then, for every $t\in [0,T]$, we have
\[
M(t)\leq \frac{\theta }{\theta -1}\alpha.
\]
\end{lemma}

\section{Strong instability of standing waves without partial confine}

In this section, we shall prove the strong instability of standing waves for \eqref{e0} without magnetic potential (i.e. $b=0$), in the mass-critical and supercritical cases. To begin with, we prove a fundamental lemma:
\begin{lemma}\label{lemma jihexingzhi}
Let $\lambda_1\geq 0$, $\lambda_2> 0$, $\lambda_3>0$, $\omega>0$, $\frac{4}{3}\leq p<4$. Then for any given $u\in H^1\backslash \{0\}$ if $\frac{4}{3}<p<4$ or $u \in H^1\backslash \{0\}$ satisfying that $\sup_{\lambda>0}S_\omega(u_\lambda)<\infty$ if $p=\frac{4}{3}$, there exists a unique $\lambda_u>0$ such that $Q(u_{\lambda_u})=0$. Moreover, we have $S_\omega(u_{\lambda_u})=\max_{\lambda>0}S_\omega(u_{\lambda})$ and the function $\lambda\mapsto S_\omega(u_{\lambda})$ is
concave on $[\lambda_u,\infty)$. Moreover, if $Q(u)\leq 0$, then $\lambda_u\in(0,1]$. Here $u_{\lambda}(x):=\lambda^{\frac{3}{2}}u(\lambda x), \lambda>0$.
\end{lemma}
\textbf{Remark 3.1.} Note that when $p=\frac{4}{3}$, the additional
 condition $\sup_{\lambda>0}S_\omega(u_\lambda)<\infty$ is satisfied if
  and only if
   $\|\nabla u\|_{L^2}^2 < \frac{3\lambda_3}{5}\int_{\mathbb{R}^3} |u|^{\frac{10}{3}}dx$,
    which implies that $\|u\|_{L^2}^2>\lambda_3^{-\frac{3}{2}}\|R\|_{L^2}^2$, where $R$ is the unique ground state of \eqref{elliptic classical}.
\begin{proof}
For any $u\in H^1\setminus \{0\}$, we have
\begin{align*}\label{su}
\frac{d}{d\lambda}S_\omega(u_\lambda)= \lambda\|\nabla u\|_{L^2}^2+\frac{\lambda_1}{2}\int_{\mathbb{R}^3} \frac{|u|^{2}}{|x|}dx
+\frac{\lambda_2}{4}\int_{\mathbb{R}^3} (|x|^{-1}\ast|u|^{2})|u|^{2}dx -\frac{\lambda_3p^*\lambda^{p^*-1}}{p+2}\|u\|_{L^{p+2}}^{p+2}
=\frac{1}{\lambda}Q(u_\lambda),
\end{align*}
where $p^*=\frac{3p}{2}$.
When $p>\frac{4}{3}$, it easily follows that there exists an unique $\lambda_u$ such that
$Q(u_{\lambda_u})=0$ and also that
\begin{equation}\label{action daoshu}
\frac{d}{d\lambda}S_\omega(u_\lambda)>0~~\mbox{if}~\lambda\in(0,\lambda_u),~~\mbox{and}~~ \frac{d}{d\lambda}S_\omega(u_\lambda)<0~~\mbox{if}~\lambda\in(\lambda_u,\infty).
\end{equation}
from which we deduce that $S_\omega(u_\lambda)<S_\omega(u_{\lambda_u})$ for any $\lambda>0$ and $\lambda\neq \lambda_u$.

When $p=\frac{4}{3}$, note that the condition $\sup_{\lambda>0}S_\omega(u_\lambda)<\infty$ holds if and only if
\begin{eqnarray}\label{3.2}
\|\nabla u\|_{L^2}^2 < \frac{p^*\lambda_3}{p+2}\|u\|_{L^{p+2}}^{p+2}=\frac{3\lambda_3}{5}\|u\|_{L^{\frac{10}{3}}}^{\frac{10}{3}},
\end{eqnarray}
and thus there also exists an unique $\lambda_u$ such that
$Q(u_{\lambda_u})=0$ and \eqref{action daoshu} holds.
Now writing $\lambda=s\lambda_u$, we have
\begin{align*}\label{su}
\frac{d^2}{d\lambda^2}S_\omega(u_\lambda)=& \|\nabla u\|_{L^2}^2 -\lambda_3p^*(p^*-1)
\frac{s^{p^*-2}\lambda_u^{p^*-2}}{p+2}\|u\|_{L^{p+2}}^{p+2}\nonumber\\
=&\frac{1}{\lambda_u^2}[\lambda_u^2\|\nabla u\|_{L^2}^2 -\lambda_3p^*(p^*-1)
\frac{s^{p^*-2}\lambda_u^{p^*}}{p+2}\|u\|_{L^{p+2}}^{p+2}].
\end{align*}
Since
\[
Q(u_{\lambda_u})= \lambda_u^2\|\nabla u\|_{L^2}^2+\frac{\lambda_1\lambda_u}{2}\int_{\mathbb{R}^3} \frac{|u|^{2}}{|x|}dx
+\frac{\lambda_2\lambda_u}{4}\int_{\mathbb{R}^3} (|x|^{-1}\ast|u|^{2})|u|^{2}dx -\lambda_3\frac{p^*\lambda_u^{p^*}}{p+2}\|u\|_{L^{p+2}}^{p+2}=0,
\]
we infer that
$\frac{d^2}{d^2\lambda}S_\omega(u_\lambda)<0$ for any $\lambda\geq \lambda_u$. This proves the lemma.
\end{proof}

Next, we prove Proposition \ref{proposition ground state no}.
 We first solve the minimization problem \eqref{minimization ground no}. To do this, we consider the following minimization problem:
\begin{equation}\label{minimization dengjia no}
\widetilde{d}(\omega)=\inf\{\widetilde{S}_\omega(v):~v\in H_r^1\backslash \{0\},~~Q(v)\leq 0\},
\end{equation}
where
\begin{align}\label{action dengjia no}
&\widetilde{S}_\omega(v)=S_\omega(v)-\frac{2}{3p}Q(v) \nonumber\\&=\frac{3p-4}{6p}\|\nabla v\|^2_{L^2}+\frac{\omega}{2}\|v\|^2_{L^2}+\lambda_1\frac{3p-2}{6p}\int_{\mathbb{R}^3} \frac{|v|^{2}}{|x|}dx+\lambda_2\frac{3p-2}{12p}\int_{\mathbb{R}^3} (|x|^{-1}\ast|v|^{2})|v|^{2}dx.
\end{align}

Then we have
\begin{lemma}\label{lemma dengjiaxing}
Let $\lambda_1\geq 0$, $\lambda_2,\lambda_3>0$, $\omega>0$, and $\frac{4}{3}\leq p<4$. Then there holds that
\begin{equation}\label{minimization problemx no}
\widetilde{d}(\omega)=\inf \{\widetilde{S}_\omega(v):~v\in H_r^1\backslash \{0\},~Q(v)=0\}=d(\omega).
\end{equation}
\end{lemma}
\begin{proof}
Indeed, if $Q(v)<0$, since $Q(v_\lambda)>0$
for sufficiently small $\lambda>0$, then by continuity, there exists $\lambda_0\in (0,1)$ such that
$Q(v_{\lambda_0})=0$. Moreover, it follows that $\widetilde{S}_\omega(v_{\lambda_0})<\widetilde{S}_\omega(v)$.
This implies that \eqref{minimization problemx no} holds.
\end{proof}

\begin{lemma}\label{lemma minimization problem}
 Let $\lambda_1\geq 0$, $\lambda_2,\lambda_3>0$, $\omega>0$, and $\frac{4}{3}\leq p<4$. Then there exists $u\in H_r^1\backslash \{0\}$, $Q(u)=0$ and  $\widetilde{S}_\omega(u)=\widetilde{d}(\omega)$.
\end{lemma}
\begin{proof}
Let $\{v_n\}$ be a minimizing sequence for \eqref{minimization dengjia no}, i.e.
$\{v_n\}\subset H_r^1\backslash\{0\}$, $Q(v_n)\leq0$ and $\widetilde{S}_\omega(v_n)\rightarrow \widetilde{d}(\omega)$ as $n\rightarrow \infty$.
When $\frac{4}{3}< p<4$, we deduce from \eqref{action dengjia no} that for $n$ large enough,
\[
\widetilde{d}(\omega)+1\geq \widetilde{S}_\omega(v_n)=S_\omega(v_n)-\frac{2}{3p}Q(v_n) \geq\frac{3p-4}{6p}\|\nabla v_n\|^2_{L^2}+\frac{\omega}{2}\|v_n\|^2_{L^2},
\]
which implies that $\{v_n\}$ is bounded in $H_r^1$.

When $p=\frac{4}{3}$, we see from \eqref{action dengjia no} that there exists $C_1>0$
such that
 \begin{equation}\label{bound hartree}
\| v_n\|_{L^2}^2+\frac{\lambda_1}{2}\int_{\mathbb{R}^3} \frac{|v_n|^{2}}{|x|}dx+\frac{\lambda_2}{4}\int_{\mathbb{R}^3} (|x|^{-1}\ast|v_n|^{2})|v_n|^{2}dx\leq C_1.
\end{equation}
Let $V_n=|x|^{-1}\ast|v_n|^{2}$, then
\[
-\Delta V_n=4\pi |v_n(x)|^{2}.
\]
This implies that $V_n$ is a minimizer of the following variational problem:
\[
\inf_{\varphi\in D^{1,2}}\left\{\frac{1}{2}\int_{\mathbb{R}^3}|\nabla \varphi(x)|^2dx-4\pi \int_{\mathbb{R}^3}|v_n(x)|^2\varphi (x)dx\right\}.
\]
Taking $\varphi (x)=\frac{4\pi\|v_n\|_{L^3}^3}{\|\nabla v_n\|_{L^{2}}^2} |v_n(x)|$, it follows that
\begin{align}\label{inequality}
\|v_n\|_{L^{3}}\leq C\|\nabla v_n\|_{L^{2}}^{1/3}\left(\int_{\mathbb{R}^3} (|x|^{-1}\ast|v_n|^{2})|v_n|^{2}dx\right)^{1/6}.
\end{align}
 Thus, it follows from $Q(v_n)\leq 0$ that there exists $C>0$ such that
  \begin{align}\label{bound}
&\|\nabla v_n\|_{L^2}^2+\frac{\lambda_1}{2}\int_{\mathbb{R}^3} \frac{|v_n|^{2}}{|x|}dx+\frac{\lambda_2}{4}\int_{\mathbb{R}^3} (|x|^{-1}\ast|v_n|^{2})|v_n|^{2}dx\nonumber\\\leq &\frac{3\lambda_3}{5}\int_{\mathbb{R}^3} |v_n|^{\frac{10}{3}}dx\leq C\|v_n\|_{L^{3}}^{8/3}\|\nabla v_n\|_{L^{2}}^{2/3}\nonumber\\\leq &C\|\nabla v_n\|_{L^{2}}^{14/9}\left(\int_{\mathbb{R}^3} (|x|^{-1}\ast|v_n|^{2})|v_n|^{2}dx\right)^{4/9}.
\end{align}
This, together with \eqref{bound hartree}, implies
\[
\|\nabla v_n\|_{L^2}^{4/9}\leq C\left(\int_{\mathbb{R}^3} (|x|^{-1}\ast|v_n|^{2})|v_n|^{2}dx\right)^{4/9}\leq C_2.
\]
This implies that $\{v_n\}$ is bounded in $H_r^1$ in the mass-critical case.

Next, we claim that $\liminf_{n\rightarrow \infty}\|v_n\|_{L^{p+2}}^{p+2}>0$. Indeed, if $\|v_n\|_{L^{p+2}}^{p+2}\rightarrow 0$ as $n\rightarrow \infty$, we deduce from $Q(v_n)\leq0$ that $\|\nabla v_n\|_{L^{2}}\rightarrow 0$.
When $\frac{4}{3}<p<4$, it follows from $Q(v_n)\leq 0$ and H\"{o}lder's inequality that
\[
\|\nabla v_n\|_{L^{2}}^2\leq C\|v_n\|_{L^{p+2}}^{p+2}\leq C\|\nabla v_n\|_{L^{2}}^{\frac{3p}{2}}\|v_n\|_{L^{2}}^{p+2-\frac{3p}{2}}\leq C\|\nabla v_n\|_{L^{2}}^{\frac{3p}{2}},
\]
which is a contradiction with $\|\nabla v_n\|_{L^{2}}\rightarrow 0$.

When $p=\frac{4}{3}$, we deduce form $Q(v_n)\leq0$ and the H\"{o}lder inequality that
\begin{align*}
\|\nabla v_n\|_{L^{2}}^2&\leq C\|v_n\|_{L^{\frac{10}{3}}}^{10/3}\leq C_3\|v_n\|_{L^3}^3+C_4\|v_n\|_{L^{5}}^5\\&\leq C_5\|\nabla v_n\|_{L^{2}}\left(\int_{\mathbb{R}^3} (|x|^{-1}\ast|v_n|^{2})|v_n|^{2}dx\right)^{1/2}+C_4\|v_n\|_{L^{5}}^5,
\end{align*}
which, together with $\|\nabla v_n\|_{L^{2}}\rightarrow 0$ implies that
\[
\|\nabla v_n\|_{L^{2}}^2\leq C_4\|v_n\|_{L^{5}}^5\leq C_6\|v_n\|_{L^{2}}^{1/2}\|\nabla v_n\|_{L^{2}}^{9/2}\leq C_7\|\nabla v_n\|_{L^{2}}^{9/2},
\]
for some positive constants $C_4$, $C_5$, $C_6$, $C_7$.
However, this contradicts with $\|\nabla v_n\|_{L^{2}}\rightarrow 0$.
Thus, we obtain
\[
\liminf_{n\rightarrow \infty}\|v_n\|_{L^{p+2}}^{p+2}>0,
\]
for all $\frac{4}{3}\leq p<4$.
Therefore, there exist a subsequence, still denoted by $\{v_n\}$ and $u\in H_r^1\backslash \{0\}$ such that
\[
v_n\rightharpoonup u\neq 0~~\mbox{weakly~in} ~H_r^1.
\]
Moreover, we deduce from Lemmas \ref{lemma Brezis-Lieb} and \ref{lemma Brezis-Lieb lemma Hartree} that
\begin{equation}\label{bl1no}
Q(u_n)-Q(u_n-u)-Q(u)\rightarrow 0,
\end{equation}
and
\begin{equation}\label{bl2no}
\widetilde{S}_\omega(u_n)-\widetilde{S}_\omega(u_n-u)-\widetilde{S}_\omega(u)\rightarrow 0.
\end{equation}
Now, we claim that $Q(u)\leq 0$. If not, it follows from \eqref{bl1no} and
$Q(u_n)\leq 0$ that $Q(u_n-u)\leq0$ for sufficiently large $n$. Thus, by the definition of
$\widetilde{d}(\omega)$, it follows that
\[
\widetilde{S}_\omega(u_n-u)\geq \widetilde{d}(\omega),
\]
which, together with $\widetilde{S}_\omega(u_n)\rightarrow \widetilde{d}(\omega)$, implies that $\widetilde{S}_\omega(u)\leq 0$,
which is a contradiction with $\widetilde{S}_\omega(u)>0$. We thus obtain $Q(u)\leq 0$.

Furthermore, we deduce from the definition of
$\widetilde{d}(\omega)$ and the weak lower semicontinuity of norm that
\[
\widetilde{d}(\omega)\leq \widetilde{S}_\omega(u)\leq \liminf_{n\rightarrow \infty}\widetilde{S}_\omega(u_n)=\widetilde{d}(\omega).
\]
This yields that
\[
\widetilde{S}_\omega(u)=\widetilde{d}(\omega).
\]
Finally, it follows from Lemma \ref{lemma dengjiaxing} that
 $Q(u)=0$.
This completes the proof.
\end{proof}
By the fact $d(\omega)=\widetilde{d}(\omega)$ and this lemma, we can obtain the following Corollary.
\begin{corollary}\label{corollary no}
 Let $\lambda_1\geq 0$, $\lambda_2,\lambda_3>0$, $\omega>0$, and $\frac{4}{3}\leq p<4$. Then there exists $u\in H_r^1\backslash \{0\}$, $Q(u)=0$ and $S_\omega(u)=d(\omega)$.
\end{corollary}
\begin{lemma}\label{lemma daoshuweiling}
 Let $\lambda_1\geq 0$, $\lambda_2,\lambda_3>0$, $\omega>0$, and $\frac{4}{3}\leq p<4$. Assume that $u\in H_r^1\backslash \{0\}$ such that $Q(u)=0$ and $S_\omega(u)=d(\omega)$. Then $S'_\omega(u)=0$.
\end{lemma}
\begin{proof}
Applying the Lagrange multiplier rule, there exists $\lambda\in\mathbb{R}$ such that $S'_\omega(u)+\lambda Q'(u)=0$. We need only to show $\lambda=0$.
Note that the equation $S'_\omega(u)+ \lambda Q'(u)=0$ can be written as
\begin{align}\label{equation lagrane0}
&-\Delta u+\omega u+\frac{\lambda_1}{|x|}u+ \lambda_2(|\cdot|^{-1}\ast |u|^2)u-
 \lambda_3|u|^p u\nonumber\\&+\lambda[-2\Delta u+\frac{\lambda_1}{|x|}u+ \lambda_2(|\cdot|^{-1}\ast |u|^2)u-
\frac{3\lambda_3p}{2}|u|^p u]=0.
\end{align}
Multiplying \eqref{equation lagrane0} by $\bar{u}$, then integrating the resulting equations
with respect to $x$ on $\mathbb{R}^3$, we get
\begin{equation}\label{0lagrane1}
S_1(u)+\lambda Q_1(u)=0,
\end{equation}
where $S_1(u)$ and $Q_1(u)$ are defined by
\begin{equation*}\label{su1}
S_1(u)=\|\nabla u\|_{L^2}^2+\omega\| u\|_{L^2}^2+\lambda_1\int_{\mathbb{R}^3} \frac{|u|^{2}}{|x|}dx+\lambda_2\int_{\mathbb{R}^3} (|x|^{-1}\ast|u|^{2})|u|^{2}dx-\|u\|_{L^{p+2}}^{p+2},
\end{equation*}
and
\begin{equation*}\label{qu1}
Q_1(u)=2\|\nabla u\|_{L^2}^2+\lambda_1\int_{\mathbb{R}^3} \frac{|u|^{2}}{|x|}dx+\lambda_2\int_{\mathbb{R}^3} (|x|^{-1}\ast|u|^{2})|u|^{2}dx-\frac{3\lambda_3p}{2}\|u\|_{L^{p+2}}^{p+2}.
\end{equation*}
On the other hand, multiplying \eqref{equation lagrane0} by $x\cdot\nabla \bar{u}$, then integrating
the resulting terms with respect to $x$ on $\mathbb{R}^3$, one
obtains
\begin{equation}\label{lagrane20}
S_2(u)+\lambda Q_2(u)=0,
\end{equation}
where $S_2(u)$ and $Q_2(u)$ are defined by
\begin{equation*}\label{su2}
S_2(u)=\frac{1}{2}\|\nabla u\|_{L^2}^2+\frac{3\omega}{2}\| u\|_{L^2}^2+\lambda_1\int_{\mathbb{R}^3} \frac{|u|^{2}}{|x|}dx+\frac{5\lambda_2}{4}\int_{\mathbb{R}^3} (|x|^{-1}\ast|u|^{2})|u|^{2}dx-\frac{3\lambda_3}{p+2}\|u\|_{L^{p+2}}^{p+2},
\end{equation*}
and
\begin{equation*}\label{qu2}
Q_2(u)=\|\nabla u\|_{L^2}^2+\lambda_1\int_{\mathbb{R}^3} \frac{|u|^{2}}{|x|}dx+\frac{5\lambda_2}{4}\int_{\mathbb{R}^3} (|x|^{-1}\ast|u|^{2})|u|^{2}dx-\frac{9\lambda_3p}{2(p+2)}\|u\|_{L^{p+2}}^{p+2}.
\end{equation*}
Note that
\begin{equation*}\label{lagrane30}
\frac{3}{2}S_1(u)- S_2(u)=Q(u)=0,
\end{equation*}
which, together with \eqref{0lagrane1} and \eqref{lagrane20}, implies  that
\begin{equation}\label{lagrane40}
\lambda \left(\frac{3}{2}Q_1(u)-Q_2(u)\right)=0.
\end{equation}
After some simple calculations, we have
\begin{equation*}\label{lagrane50}
\frac{3}{2}Q_1(u)-Q_2(u)=-\frac{\lambda_1}{2}\int_{\mathbb{R}^3} \frac{|u|^{2}}{|x|}dx-\frac{\lambda_2}{4}\int_{\mathbb{R}^3} (|x|^{-1}\ast|u|^{2})|u|^{2}dx-\lambda_3\frac{3p(3p-4)}{4(p+2)}\|u\|_{L^{p+2}}^{p+2}<0.
\end{equation*}
Thus, we obtain from \eqref{lagrane40} that $\lambda=0$. This completes the proof.
\end{proof}

\begin{proof}[\textbf{Proof of Proposition \ref{proposition ground state no}}]
Proposition \ref{proposition ground state no} follows immediately from Lemma \ref{lemma daoshuweiling}.
\end{proof}

\begin{proof}[\textbf{Proof of Theorem \ref{theorem sharp}.}]
We first prove that $\mathcal{A}_\omega$ and $\mathcal{B}_\omega$ are two invariant manifolds of \eqref{equation partial} with $b=0$.
Let $\psi_0\in \mathcal{A}_\omega$, by Proposition 2.1, we see that there
exists a unique solution $\psi\in C([0,T^*),H_r^1)$ with initial data $\psi_0$. We deduce from the conservations of mass and energy that
\begin{equation}\label{asc no}
S_\omega(\psi(t))=S_\omega(\psi_0)<S_\omega(u),
\end{equation}
for any $t\in [0,T^*)$. In addition, by the continuity of the function $t\mapsto Q(\psi(t))$ and
Corollary \ref{corollary no}, if there exists
$t_0\in[0,T^*)$ such that  $Q(\psi(t_0))= 0$, then $S_\omega(\psi(t_0))\geq S_\omega(u)$, which contradicts with \eqref{asc no}. Therefore, we have $Q(\psi(t))>0 $ for any $t\in [0,T^*)$. Similarly, we can prove that $ \mathcal{B}_\omega$ is invariant under the flow of \eqref{equation partial} with $b=0$.

Now, we prove $(1)$.
 If $\psi_0\in \mathcal{A}_\omega$,
then $Q(\psi(t))>0 $ for any $t\in [0,T^*)$.
Thus,
we deduce from the conservation of energy and \eqref{action dengjia no} that for all $\psi_0\in \mathcal{A}_\omega$
\begin{align}\label{action111}
S_\omega(u)> S_\omega(\psi(t))>S_\omega(\psi(t))-\frac{2}{3p}Q(\psi(t)) \geq \frac{3p-4}{6p}
\|\nabla \psi(t)\|^2_{L^2}
.
\end{align}
This implies that the solution $\psi(t)$ of \eqref{equation partial} exists globally.

Next, we prove $(2)$.
We first claim that if $v\in H^1$ satisfies $Q(v)<0$, then $Q(v)\leq 2(S_\omega(v)-S_\omega(u))$.
Indeed, due to $Q(v)<0$, it follows from Lemma \ref{lemma jihexingzhi}
there exists $\lambda_v\in (0,1)$ such that $Q(v_{\lambda_v})=0$.
Moreover, when $p\geq\frac{4}{3}$, the function
\begin{align*}
f(\lambda)=&S_\omega(v_\lambda)-\frac{\lambda^{2}}{2}Q(v)\\=&
\frac{\omega}{2}\int_{\mathbb{R}^3} |v|^{2}dx+\frac{\lambda_1(2\lambda-\lambda^2)}{4}\int_{\mathbb{R}^3} \frac{|v|^{2}}{|x|}dx
+\frac{\lambda_2(2\lambda-\lambda^2)}{8}\int_{\mathbb{R}^3} (|x|^{-1}\ast|v|^{2})|v|^{2}dx \\&-\frac{\lambda_3(4\lambda^{\frac{3p}{2}}-3p\lambda^2)}{4(p+2)}\|v\|_{L^{p+2}}^{p+2},
\end{align*}
attains its maximum at $\lambda=1$. Thus, it follows from $Q(v_{\lambda_v})=0$ and Proposition \ref{proposition ground state no} that
\[
S_\omega(u)\leq S_\omega(v_{\lambda_v})=
S_\omega(v_{\lambda_v})-\frac{\lambda_v^2}{2}Q(v_{\lambda_v})
<S_\omega(v)-\frac{1}{2}Q(v).
\]
Thus, when $\psi_0\in \mathcal{B}_\omega$,
 $Q(\psi(t))<0 $ and then
 \[
 Q(\psi(t))\leq 2(S_\omega(\psi(t))-S_\omega(u))
 =2(S_\omega(\psi_0)-S_\omega(u))<0
  \]
  for any $t\in [0,T^*)$. This, together with $x\psi_0\in L^2$, implies that the corresponding solution $\psi(t)$ blows up in finite time.
\end{proof}

\begin{proof}[\textbf{Proof of Theorem \ref{theorem sharp criteria critical}.}]
When $\|\psi_0\|_{L^2}< \lambda_3^{-\frac{3}{4}}\|R\|_{L^2}$, the proof of global existence is standard, so we omit it.

When $\|\psi_0\|_{L^2}= \lambda_3^{-\frac{3}{4}}\|R\|_{L^2}$, we prove this theorem by contradiction. If the solution $\psi(t)$ of \eqref{equation partial} blows up in finite time, then there exists $T^*>0$ such that $\lim_{t\rightarrow T^*}\|\nabla \psi(t)\|_{L^2}=\infty$.
 Set
\[
\rho(t)=\|\nabla R\|_{L^2}/\|\nabla \psi(t)\|_{L^2}~~\mbox{and}~~v(t,x)=\rho^{\frac{3}{2}}(t)\psi(t,\rho(t) x).
\]
Let $\{t_n\}_{n=1}^\infty$ be an any time sequence such that $t_n\rightarrow T^*$, $\rho_n:=\rho(t_n)$ and $v_n(x):=v(t_n,x)$.
Then, the sequence $\{v_n\}$
satisfies
\begin{equation}\label{44}
\|v_n\|_{L^2}=\|\psi(t_n)\|_{L^2}=\|\psi_0\|_{L^2}=\lambda_3^{-\frac{3}{4}}\|R\|_{L^2},~~\|\nabla v_n\|_{L^2}=\rho_n\|\nabla \psi(t_n)\|_{L^2}=\|\nabla R\|_{L^2}.
\end{equation}
Next, we recall the following sharp Gagliardo-Nirenberg inequality (see \cite[Lemma 8.4.2]{ca2003})
\begin{equation}\label{sharp gn}
\frac{3}{10}\|u\|_{L^{\frac{10}{3}}}^{\frac{10}{3}}\leq \frac{\|u\|_{L^2}^\frac{4}{3}}{2\|R\|_{L^2}^\frac{4}{3}}\|\nabla u\|_{L^2}^2,\ \forall u\in H^1,
\end{equation}
where $R$ is the ground state of \eqref{elliptic classical}.
Applying this inequality and the conservation of energy, we deduce that
\begin{align*}\label{45}
0\leq\frac{1}{2}\int_{\mathbb{R}^3} |\nabla v_n(x)|^2dx&-\frac{3\lambda_3}{10}\int_{\mathbb{R}^3}
|v_n(x)|^{\frac{10}{3}}dx
  =\rho_n^2\left(\frac{1}{2}\int_{\mathbb{R}^3} |\nabla \psi(t_n,x)|^2dx-\frac{3\lambda_3}{10}\int_{\mathbb{R}^3}
|\psi(t_n,x)|^{\frac{10}{3}}dx\right)\nonumber\\
  =&\rho_n^2\left(E(\psi_0)-\lambda_1\int_{\mathbb{R}^3} \frac{|\psi(t_n)|^{2}}{|x|}dx-\lambda_2\int_{\mathbb{R}^3} (|x|^{-1}\ast|\psi(t_n)|^{2})|\psi(t_n)|^{2}dx\right)\rightarrow 0,
  \end{align*}
as $n\rightarrow\infty$.
This implies that
\[
 \lim_{n\rightarrow \infty}\int_{\mathbb{R}^3}
|v_n(x)|^{\frac{10}{3}}dx= \frac{5}{3\lambda_3}\|\nabla R\|_{L^2}^2.
 \]
 Thus, we deduce from \eqref{44} that there exist a subsequence, still denoted by $\{v_n\}$ and $u\in H^1\backslash \{0\}$ such that
\[
u_n:=\tau_{x_n}v_n\rightharpoonup u\neq 0~~\mbox{weakly in} ~H^1,
\]
for some $\{x_n\}\subset \mathbb{R}^3$. We consequently deduce from \eqref{inequality} that there exists $C_0>0$ such that
\begin{align}\label{dayuling}
\liminf_{n\rightarrow \infty}\int_{\mathbb{R}^3} (|x|^{-1}\ast|v_n|^{2})|v_n|^{2}dx=\liminf_{n\rightarrow \infty}\int_{\mathbb{R}^3} (|x|^{-1}\ast|u_n|^{2})|u_n|^{2}dx\geq C_0>0.
\end{align}
On the other hand, we deduce from \eqref{sharp gn} and $\|\psi(t)\|_{L^2}=\|\psi_0\|_{L^2}=\lambda_3^{-\frac{3}{4}}\|R\|_{L^2}$ that
\begin{align*}
\frac{1}{2}\int_{\mathbb{R}^3} |\nabla \psi(t,x)|^2 dx-\frac{3}{10}\int_{\mathbb{R}^3}
|\psi(t,x)|^{\frac{10}{3}}dx+\lambda_1\int_{\mathbb{R}^3} \frac{|\psi(t,x)|^{2}}{|x|}dx\geq 0,
\end{align*}
for all $t\in[0,T^*)$.
This implies that
\begin{align*}
\lambda_2\int_{\mathbb{R}^3} (|x|^{-1}\ast|\psi(t)|^{2})|\psi(t,x)|^{2}dx\leq E(\psi_0),
\end{align*}
for all $t\in[0,T^*)$. We consequently obtain that
\begin{align*}
\lambda_2\int_{\mathbb{R}^3} (|x|^{-1}\ast|v_n|^{2})|v_n(x)|^{2}dx=\rho_n\lambda_2\int_{\mathbb{R}^3} (|x|^{-1}\ast|\psi(t_n)|^{2})|\psi(t_n,x)|^{2}dx\leq \rho_nE(\psi_0)\rightarrow 0,
\end{align*}
as $n\rightarrow\infty$,
which is a contradiction with \eqref{dayuling}. Thus, the solution $\psi(t)$ of \eqref{equation partial} with $b=0$ exists globally.

When $\rho>\lambda_3^{-\frac{3}{4}}\|R\|_{L^2}$, we define the initial data $\psi_0(x)=\frac{\rho}{\|R\|_{L^2}}\lambda^{\frac{3}{2}} R(\lambda x)$, then
 $|x|\psi_0\in L^2$ and $\|\psi_0\|_{L^2}=\rho$,  $J(t)=\int_{\mathbb{R}^3} |xu(t,x)|^2dx$ is well-defined, and it follows from Lemma 2.2 that
\begin{equation}\label{j11}
J''(t)=16E(\psi_0)-4\lambda_1\int_{\mathbb{R}^3} \frac{|\psi(t,x)|^{2}}{|x|}dx-2\lambda_2\int_{\mathbb{R}^3} (|x|^{-1}\ast|\psi(t)|^{2})|\psi(t,x)|^{2}dx.
\end{equation}
 By the
definition of initial data $\psi_0(x)=\frac{\rho}{\|R\|_{L^2}}\lambda^{\frac{3}{2}} R(\lambda x)$ and the Pohozaev identity for equation \eqref{elliptic classical}, i.e., $\frac{1}{2}\|\nabla R\|_{L^2}^2=\frac{3}{10}\|R\|^{\frac{10}{3}}_{L^{\frac{10}{3}}}$, we deduce that
\begin{align*}
E(\psi_0)=& \frac{\rho^2\lambda^2}{2\|R\|_{L^2}^2}\|\nabla R\|_{L^2}^2+\frac{\lambda_1\rho^2\lambda}{2\|R\|_{L^2}^2}\int_{\mathbb{R}^3} \frac{|R(x)|^{2}}{|x|}dx
\\&+\frac{\lambda_2\rho^4\lambda}{4\|R\|_{L^2}^4}\int_{\mathbb{R}^3} (|x|^{-1}\ast|R|^{2})|R|^{2}dx
-\frac{3\lambda_3\rho^{\frac{10}{3}}\lambda^2}{10\|R\|_{L^2}^{\frac{10}{3}}}
\|R\|_{L^{\frac{10}{3}}}^{\frac{10}{3}}
\\=&
-\frac{\rho^2\lambda^2}{2}\frac{\lambda_3\rho^{\frac{4}{3}}-\|R\|_{L^2}^\frac{4}{3}}
{\|R\|_{L^2}^{\frac{10}{3}}}
\|\nabla R\|_{L^2}^2+\frac{\lambda_1\rho^2\lambda}{2\|R\|_{L^2}^2}\int_{\mathbb{R}^3} \frac{|R(x)|^{2}}{|x|}dx
\\&+\frac{\lambda_2\rho^4\lambda}{4\|R\|_{L^2}^4}\int_{\mathbb{R}^3} (|x|^{-1}\ast|R|^{2})|R|^{2}dx<0,
\end{align*}
for sufficiently large $\lambda>0$.
This implies $E(\psi_0) < 0$. It follows from \eqref{j11} that $J''(t)<16 E(\psi_0) < 0$.
By the standard concave argument, the solution
$\psi(t)$ of \eqref{equation partial} with $b=0$ and the initial data $\psi_0$ blows up in finite time.
\end{proof}

\begin{proof}[\textbf{Proof of Theorem \ref{th instability no partial}.}]
Let $u$ be the solution of minimization problem \eqref{minimization ground no},
and $u_\lambda(x):=\lambda^{\frac{3}{2}}u(\lambda x)$, then $u_{\lambda}\in \mathcal{B}_\omega$ for all $\lambda> 1$.
 Due to $u_{\lambda}\rightarrow u$ in $H^1$ as $\lambda\searrow 1$, for every $\varepsilon>0$, there exists $\lambda_0>1$ such that $\|u_{\lambda_0}-u\|_{H^1}<\varepsilon/2$.
 Let $\chi\in C^\infty[0,\infty)$ be a function satisfying
$0\leq \chi \leq 1$, $\chi(r)=1$ if $0\leq r \leq 1$, and $\chi(r)=0$ if $r \geq 2$. For $M>0$, we define a cutoff
function $\chi_M\in C_c^\infty(\mathbb{R}^3)$ by $\chi_M(x)=\chi(|x|/M)$. Then we see that $\chi_Mu_{\lambda_0}\rightarrow u_{\lambda_0}$ in $H^1$ as
$M\rightarrow \infty$. Moreover, we have $\chi_Mu_{\lambda_0}\in \Sigma$.
Therefore, there exists $M_0>0$ such
that
$\|\chi_{M_0}u_{\lambda_0}-u_{\lambda_0}\|_{H^1}<\varepsilon/2$. Thus, we obtain
$\|\chi_{M_0}u_{\lambda_0}-u\|_{H^1}<\varepsilon$. In addition, by the continuity of  $S_{\omega}(u)$ and $Q(u)$, we have $\chi_{M_0}u_{\lambda_0}\in \Sigma\cap \mathcal{B}_\omega$.

Let $\psi_{\lambda_0}(t)$ be the maximal solution of \eqref{equation partial} with $b=0$ and the initial data $\chi_{M_0}u_{\lambda_0}\in \Sigma\cap \mathcal{B}_\omega$. Then, applying Lemma \ref{lemma blowup} and Theorem \ref{theorem sharp}, it follows that $\psi_{\lambda_0}(t)\in \Sigma\cap \mathcal{B}_\omega$ and
 \[
\frac{d^2}{dt^2}\int_{\mathbb{R}^3}|x\psi_{\lambda_0}(t,x)|^2dx=8Q(\psi_{\lambda_0}(t))\leq 16(S_{\omega}(\chi_{M_0}u_{\lambda_0})-S_{\omega}(u))<0,
 \]
 for all $t\in[0,T^*)$.
 This implies that the solution $\psi_{\lambda_0}(t)$ of \eqref{equation partial} with the initial data $\chi_{M_0}u_{\lambda_0}$ blows up in finite time. Hence, the result follows, since for any $\varepsilon>0$, there exist $M_0>0$ and $\lambda_0$ such that $\|\chi_{M_0}u_{\lambda_0}-u\|_{H^1}<\varepsilon$.
 This completes the proof.
\end{proof}

\section{Normalized solutions without partial confine}
In this section, we study the existence and properties of normalized solutions in the mass-supercritical case without partial confine. Since the functional is unbounded from below, we consider the following minimization problem.
\begin{equation}\label{minimization problem normalzied1}
\gamma(c):=\inf_{u\in V_r(c)}E(u),
\end{equation}
where the constraint $V_r(c)$ is defined by
\begin{equation}\label{submanifold1}
V_r(c):=\{u\in S_r(c):~Q(u)=0\}.
\end{equation}
We note that the manifold of $V_r(c)$ is indeed a natural constraint for $E$ on $S_r(c)$, since we have
\begin{lemma}\label{lm-lagrange}
Let $\lambda_1\geq 0$, $\lambda_2,\lambda_3>0$, and $\frac{4}{3}\leq p<4$. Then each critical point of $E|_{V_r(c)}$ is exactly a critical point of $E|_{S_r(c)}$.
\end{lemma}

\begin{proof}
Let $u$ be a critical point of $E|_{V_r(c)}$, then by \cite[Corollary 4.1.2]{k-cc} we have the alternative: either $(i)$ $Q'(u)$ and $(\|u\|_{L^2}^2)' $ are linearly dependent, or $(ii)$ there exists $\omega_1, \omega_2\in \mathbb{R}$ such that
\begin{eqnarray}\label{lagrange110}
E'(u)+\omega_1Q'(u)+\omega_2u=0,\ \mbox{ in } H_r^{-1}.
\end{eqnarray}
Indeed, $(i)$ is impossible. If so, then we have for some $\omega_0\in \mathbb{R}$,
\begin{eqnarray*}
Q'(u)+\omega_0(\|u\|_{L^2}^2)'=0,\ \mbox{ in } H_r^{-1},
\end{eqnarray*}
or equivalently
\begin{eqnarray}\label{lagrange111}
-2\Delta u +\dfrac{\lambda_1}{|x|}u+\lambda_2(|x|^{-1}\ast|u^2|)u-\dfrac{3p\lambda_3}{2}|u|^pu+2\omega_0u=0,\ \mbox{ in } H_r^{-1}.
\end{eqnarray}
Multiplying \eqref{lagrange111} by $\bar{u}$ and integrating by part, we have
\begin{eqnarray}\label{lagrange112}
2\|\nabla u\|_{L^2}^2+\lambda_1\int_{\mathbb{R}^3} \frac{|u|^{2}}{|x|}dx
+\lambda_2\int_{\mathbb{R}^3} (|x|^{-1}\ast|u|^{2})|u|^{2}dx -\frac{3p\lambda_3}{2}\|u\|_{L^{p+2}}^{p+2}+2\omega_0\|u\|_{L^2}^2=0.
\end{eqnarray}
Again, multiplying \eqref{lagrange111} by $x\cdot \nabla \bar{u}$ and integrating by part, we have
\begin{eqnarray}\label{lagrange113}
\|\nabla u\|_{L^2}^2+\lambda_1\int_{\mathbb{R}^3} \frac{|u|^{2}}{|x|}dx
+\dfrac{5\lambda_2}{4}\int_{\mathbb{R}^3} (|x|^{-1}\ast|u|^{2})|u|^{2}dx -\frac{9p\lambda_3}{2(p+2)}\|u\|_{L^{p+2}}^{p+2}+3\omega_0\|u\|_{L^2}^2=0.
\end{eqnarray}
Eliminate the $\omega_0\|u\|_{L^2}^2$ term from \eqref{lagrange112} and \eqref{lagrange113}, then we get
\begin{eqnarray}\label{lagrange114}
2\|\nabla u\|_{L^2}^2+\frac{\lambda_1}{2}\int_{\mathbb{R}^3} \frac{|u|^{2}}{|x|}dx
+\dfrac{\lambda_2}{4}\int_{\mathbb{R}^3} (|x|^{-1}\ast|u|^{2})|u|^{2}dx -\frac{9p^2\lambda_3}{4(p+2)}\|u\|_{L^{p+2}}^{p+2}=0.
\end{eqnarray}
Note that $u\in V_r(c)$ and then $Q(u)=0$. By this identity and \eqref{lagrange114}, we eliminate the $\|\nabla u\|_{L^2}^2$ term, to obtain
$$\frac{\lambda_1}{2}\int_{\mathbb{R}^3} \frac{|u|^{2}}{|x|}dx
+\dfrac{\lambda_2}{4}\int_{\mathbb{R}^3} (|x|^{-1}\ast|u|^{2})|u|^{2}dx + \frac{3p(3p-4)\lambda_3}{4(p+2)}\|u\|_{L^{p+2}}^{p+2}=0,$$
which is obviously a contradiction, since $u\neq 0$.

Hence to end the proof, we only need to show that $\omega_1=0$ in \eqref{lagrange110}. Following the above idea, multiplying \eqref{lagrange110} by $\bar{u}$ and $x\cdot \nabla \bar{u}$ respectively, and then integrating by part, we have
\begin{eqnarray}\label{lagrange115}
(1+2\omega_1)\|\nabla u\|_{L^2}^2+(1+\omega_1)\lambda_1\int_{\mathbb{R}^3} \frac{|u|^{2}}{|x|}dx
+(1+\omega_1)\lambda_2\int_{\mathbb{R}^3} (|x|^{-1}\ast|u|^{2})|u|^{2}dx \nonumber\\ -(1-\frac{3p\omega_1}{2})\lambda_3\|u\|_{L^{p+2}}^{p+2}+\omega_2\|u\|_{L^2}^2=0,
\end{eqnarray}
and
\begin{eqnarray}\label{lagrange116}
(1+2\omega_1)\|\nabla u\|_{L^2}^2+2(1+\omega_1)\lambda_1\int_{\mathbb{R}^3} \frac{|u|^{2}}{|x|}dx
+(1+\omega_1)\dfrac{5\lambda_2}{2}\int_{\mathbb{R}^3} (|x|^{-1}\ast|u|^{2})|u|^{2}dx \nonumber\\ -(1-\frac{3p\omega_1}{2})\dfrac{6\lambda_3}{p+2}\|u\|_{L^{p+2}}^{p+2}+3\omega_2\|u\|_{L^2}^2=0.
\end{eqnarray}
Eliminating the $\omega_1\|u\|_{L^2}^2$ term from \eqref{lagrange115} and \eqref{lagrange116}, we have
\begin{eqnarray}\label{lagrange117}
(2+4\omega_1)\|\nabla u\|_{L^2}^2+(1+\omega_1)\lambda_1\int_{\mathbb{R}^3} \frac{|u|^{2}}{|x|}dx
+\dfrac{(1+\omega_1)\lambda_2}{2}\int_{\mathbb{R}^3} (|x|^{-1}\ast|u|^{2})|u|^{2}dx \nonumber\\ -(1-\frac{3p\omega_1}{2})\dfrac{3p\lambda_3}{p+2}\|u\|_{L^{p+2}}^{p+2}=0.
\end{eqnarray}
Taking into account of $Q(u)=0$, we then deduce from \eqref{lagrange117} that
$$\omega_1\cdot\Big(4\|\nabla u\|_{L^2}^2+\lambda_1\int_{\mathbb{R}^3} \frac{|u|^{2}}{|x|}dx
+\dfrac{\lambda_2}{2}\int_{\mathbb{R}^3} (|x|^{-1}\ast|u|^{2})|u|^{2}dx +\dfrac{9p^2\lambda_3}{p+2}\|u\|_{L^{p+2}}^{p+2}\Big)=0 ,$$
which implies that $\omega_1=0$. Then the lemma is proved.
\end{proof}
\begin{lemma}\label{lm-nonexistence11}
Let $\lambda_1\geq 0$, $\lambda_2,\lambda_3>0$ and $p=\frac{4}{3}$. Then for all $c\in (0, \lambda_3^{-\frac{3}{2}}\|R\|_{L^2}^2]$, the functional $E$ has no any critical point on $S_r(c)$, where $R$ is the unique ground state of \eqref{elliptic classical}.
\end{lemma}
\begin{proof}
For any $c\in (0, \lambda_3^{-\frac{3}{2}}\|R\|_{L^2}^2]$, we note that if $u_0$ is a critical point of $E$ on $S_r(c)$, then necessarily $Q(u_0)=0$ and $u_0\neq 0$. Thus by \eqref{sharp gn}, we have
\begin{eqnarray}\label{necess}
\|\nabla u_0\|_{L^2}^2<\frac{3\lambda_3}{5}\int_{\mathbb{R}^3}|u_0|^{\frac{10}{3}}dx\leq \frac{\lambda_3\|u_0\|_{L^2}^{\frac{4}{3}}}{\|R\|_{L^2}^{\frac{4}{3}}}\|\nabla u_0\|_{L^2}^2.
\end{eqnarray}
This implies
$$\Big(\lambda_3c^{\frac{2}{3}}-\|R\|_{L^2}^{\frac{4}{3}}\Big)\|\nabla u_0\|_{L^2}^2>0,$$
which is a contradiction since $c\leq \lambda_3^{-\frac{3}{2}}\|R\|_{L^2}^2$.
\end{proof}
\textbf{Remark 4.1.} Due to this lemma, in what follows, to obtain the existence, it is necessary to restrict our analyses on $c>\lambda_3^{-\frac{3}{2}}\|R\|_{L^2}^2$ when $p=\frac{4}{3}$.

To show that $\gamma(c)$ is attained for some $u\in V_r(c)$, we consider first the following equivalent minimization problem:
\begin{equation}\label{minimization dengjia normal}
\widetilde{\gamma}(c):=\inf\{\widetilde{E}(v):~v\in S_r(c),~Q(v)\leq 0\},
\end{equation}
where
\begin{align}\label{enerfy dengjia}
\widetilde{E}(v):=&
E(v)-\frac{2}{3p}Q(v) \nonumber\\=&\frac{3p-4}{6p}\|\nabla v\|^2_{L^2}+\lambda_1\frac{3p-2}{6p}\int_{\mathbb{R}^3} \frac{|v|^{2}}{|x|}dx+\lambda_2\frac{3p-2}{12p}\int_{\mathbb{R}^3} (|x|^{-1}\ast|v|^{2})|v|^{2}dx.
\end{align}
By a similar argument as Lemma \ref{lemma dengjiaxing}, we can prove that
\begin{lemma}\label{lm-gammac}
Let $\lambda_1\geq 0$, $\lambda_2,\lambda_3>0$, and $\frac{4}{3}\leq p<4$. Then for any $c>0$, there holds that
\begin{equation}\label{minimization problemx0}
\widetilde{\gamma}(c)=\inf \{\widetilde{E}(v):~v\in S_r(c),~Q(v)=0\}=\gamma(c).
\end{equation}
\end{lemma}


\begin{lemma}\label{lemma dandiaoxing}
Let $\lambda_1\geq 0$, $\lambda_2,\lambda_3>0$.
Then when $\frac{4}{3}< p<4$, the function $c\mapsto \gamma(c)$ is non-increasing on $(0,+\infty)$, and when $p=\frac{4}{3}$, $c\mapsto \gamma(c)$ is non-increasing on $(\lambda_3^{-\frac{3}{2}}\|R\|_{L^2}^2,+\infty)$.
\end{lemma}
\begin{proof}
To prove the non-increasing property of $\gamma(c)$, it is essential to show that
\begin{eqnarray}\label{1111}
\gamma(c)=\inf_{u\in S_r(c)}\max_{\lambda>0}E(u_{\lambda}),\ \mbox{ if }\ \frac{4}{3}<p<4,
\end{eqnarray}
and
\begin{eqnarray}\label{1112}
\gamma(c)=\inf_{u\in K_r(c)}\max_{\lambda>0}E(u_{\lambda}),\ \mbox{ if }\ p=\frac{4}{3},
\end{eqnarray}
where $K_r(c):=\{u\in S_r(c) : \|\nabla u\|_{L^2}^2 <\frac{3\lambda_3}{5}\int_{\mathbb{R}^3} |u|^{\frac{10}{3}}dx\}$. For this aim, the key is the point that for any given $u\in S_r(c)$ if $\frac{4}{3}<p<4$ or $u\in K_r(c)$ if $p=\frac{4}{3}$, there exists a unique $\lambda_u>0$ such that $Q(u_{\lambda_u})=0$, and that $E(u_{\lambda_u})=\max_{\lambda>0}E(u_{\lambda})$, which is indeed already obtained by Lemma \ref{lemma jihexingzhi} and Remark 3.1. Since others can be argued as standard as that in \cite[Lemma 5.3]{jean-luo}, hence here we omit the details.
\end{proof}

\begin{proposition}\label{proposition existence}
Let $\lambda_1\geq 0$, $\lambda_2,\lambda_3>0$. Then for all $c>0$ if $\frac{4}{3}< p<4$, or for all $c>\lambda_3^{-\frac{3}{2}}\|R\|_{L^2}^2$ if $p=\frac{4}{3}$, we have $\gamma(c)>0$ and moreover, there exists a $u\in H_r^1\backslash \{0\}$, such that $\widetilde{E}(u)=\gamma(c)$ and $u\in V_r(\|u\|_{L^2}^2)$ with $0<\|u\|_{L^2}^2\leq c$.
\end{proposition}
\begin{proof}
We first show that $\gamma(c)>0$.  For any $v\in S_r(c)$ with $Q(v)\leq 0$, when $\frac{4}{3}< p<4$, we have
\[
\|\nabla v\|_{L^2}^2<\frac{3\lambda_3p}{2(p+2)}\|v\|_{L^{p+2}}^{p+2}\leq C\|\nabla v\|_{L^2}^{\frac{3p}{2}}
\| v \|^{\frac{4-p}{2}}_{L^{2}},
\]
which implies that
\begin{eqnarray}\label{4.144}
c^{\frac{p-4}{2}} \leq C\|\nabla v\|_{L^2}^{\frac{3p}{2}-2}.
\end{eqnarray}
When $p=\frac{4}{3}$, from \eqref{bound} and the Young inequality with $\varepsilon$, we have
\begin{align}\label{4.145}
\|\nabla v\|_{L^2}^2 + \frac{\lambda_2}{4}\int_{\mathbb{R}^3} (|x|^{-1}\ast|v|^{2})|v|^{2}dx\leq& C \|\nabla v\|_{L^2}^{\frac{14}{9}}\Big(\int_{\mathbb{R}^3} (|x|^{-1}\ast|v|^{2})|v|^{2}dx\Big)^{\frac{4}{9}}\nonumber\\
\leq & \varepsilon \|\nabla v\|_{L^2}^2 +C(\varepsilon)\Big(\int_{\mathbb{R}^3} (|x|^{-1}\ast|v|^{2})|v|^{2}dx\Big)^2.
\end{align}
Letting $\varepsilon=1$, then \eqref{4.145} implies that
\begin{eqnarray}\label{4.146}
\int_{\mathbb{R}^3} (|x|^{-1}\ast|v|^{2})|v|^{2}dx \geq \widetilde{C},
\end{eqnarray}
for some  $\widetilde{C}>0$ independent of $v$. Therefore, in both the cases, by \eqref{4.144}, \eqref{4.146} and the definition of $\widetilde{\gamma}(c)$, we see that $\widetilde{\gamma}(c)>0$. Thus by Lemma \ref{lm-gammac}, $\gamma(c)>0$ follows.

Now let $\{v_n\}$ be an arbitrary minimizing sequence for \eqref{minimization dengjia normal}, i.e.,
 $\{v_n\}\subset S_r(c)$, $Q(v_n)\leq 0$ and $\widetilde{E}(v_n)\rightarrow \widetilde{\gamma}(c)=\gamma(c)$ as $n\rightarrow \infty$. We claim that $\{v_n\}$ is bounded in $H_r^1$. Indeed, since $v_n\in S_r(c)$, it is enough to verify the boundedness of $\|\nabla v_n\|_{L^2}$. When $\frac{4}{3}<p<4$, by \eqref{enerfy dengjia}, $\|\nabla v_n\|_{L^2}$ is surely bounded. When $p=\frac{4}{3}$, by \eqref{enerfy dengjia}, $\int_{\mathbb{R}^3} (|x|^{-1}\ast|v_n|^{2})|v_n|^{2}dx$ is bounded, then in \eqref{4.145} by letting $\varepsilon =\frac{1}{2}$ we see that $\|\nabla v_n\|_{L^2}$ is bounded. Thus the claim is verified. In addition, by \eqref{4.144} and \eqref{4.146}, it follows from $Q(v_n)\leq 0$ that there exists $C_0>0$ such that
\[
\frac{3\lambda_3p}{2(p+2)}\liminf_{n\rightarrow \infty}\|v_n\|^{p+2}_{L^{p+2}}\geq C_0>0.
\]
Hence, there exist a subsequence, still denoted by $\{v_n\}$ and $u\in H_r^1\backslash \{0\}$ such that
\[
v_n\rightharpoonup u\neq 0~~\mbox{weakly~in} ~H_r^1.
\]

Moreover, we deduce from Lemmas \ref{lemma Brezis-Lieb} and \ref{lemma Brezis-Lieb lemma Hartree} that
\begin{equation}\label{bl10}
Q(v_n)-Q(v_n-u)-Q(u)\rightarrow 0,
\end{equation}
\begin{equation}\label{bl20}
\widetilde{E}(v_n)-\widetilde{E}(v_n-u)-\widetilde{E}(u)\rightarrow 0,
\end{equation}
\begin{equation}\label{bl30}
\|v_n\|_{L^{2}}^{2}-\|v_n-u\|_{L^{2}}^2-\|u\|_{L^{2}}^{2}\rightarrow 0.
\end{equation}
By \eqref{bl30}, $0<\|u\|_{L^{2}}^{2}\leq c$. We claim that $Q(u)\leq 0$, which indeed can be proved by excluding the other possibilities:

(1) If $Q(u)> 0$ and $\|u\|_{L^{2}}^{2}<c$, it follows from \eqref{bl10} and
$Q(v_n)\leq 0$ that $Q(v_n-u)\leq0$ for sufficiently large $n$.
Set $c_1=c-\|u\|_{L^{2}}^{2}$ and $w_n=\sqrt{c_1}\|v_n-u\|_{L^{2}}^{-1}(v_n-u)$, then we have
\[
\|v_n-u\|_{L^{2}}\rightarrow \sqrt{c_1},~~w_n\in S(c_1),~~\mbox{and}~~Q(w_n)\leq 0.
\]
Thus, by the definition of
$\widetilde{\gamma}(c_1)$, it follows that
\[
\widetilde{E}(w_n)\geq \widetilde{\gamma}(c_1)~~\mbox{and}~~\widetilde{E}(v_n-u)\geq \widetilde{\gamma}(c_1).
\]
Note from $Q(w_n)\leq 0$ and \eqref{necess}, we see in particular that $c_1>\lambda_3^{-\frac{3}{2}}\|R\|_{L^2}^2$ when $p=\frac{4}{3}$. Then applying Lemma \ref{lemma dandiaoxing}, $\widetilde{\gamma}(c_1)=\gamma(c_1)\geq \gamma(c)$, and by \eqref{bl20} we obtain
$$\widetilde{E}(u)=\frac{3p-4}{6p}\|\nabla u\|^2_{L^2}+\lambda_1\frac{3p-2}{6p}\int_{\mathbb{R}^3} \frac{|u|^{2}}{|x|}dx+\lambda_2\frac{3p-2}{12p}\int_{\mathbb{R}^3} (|x|^{-1}\ast|u|^{2})|u|^{2}dx\leq0,$$
which is a contradiction since $u\neq 0$.

(2) If $Q(u)> 0$ and $\|u\|_{L^{2}}^{2}=c$, then $v_n\rightarrow u$ in $L^2$ as $n\rightarrow \infty$. This implies that
$v_n\rightarrow u$ in $L^{q}$ as $n\rightarrow \infty$, for any $q\in [2, 6)$ . On the other hand, we deduce from $Q(u)> 0$ that $Q(v_n-u)\leq0$ for sufficiently large $n$. Thus, we can obtain $v_n\rightarrow u$ in $H_r^1$ as $n\rightarrow \infty$.
This yields $Q(v_n-u)\rightarrow 0$ as $n\rightarrow \infty$. Thus, it follows from \eqref{bl10} and $Q(u)> 0$ that $Q(v_n)> 0$ for sufficiently large $n$, which is a contradiction with $Q(v_n)\leq 0$.

Therefore, it follows that $Q(u)\leq 0$ and $0<\|u\|_{L^{2}}^{2}\leq c$, particularly, by \eqref{necess}, $\lambda_3^{-\frac{3}{2}}\|R\|_{L^2}^2< \|u\|_{L^{2}}^{2}\leq c$ when $p=\frac{4}{3}$. Thus we deduce from the definition of
$\widetilde{\gamma}(c)$, the weak lower semicontinuity of the functional $\widetilde{E}(u)$  and also Lemma \ref{lemma dandiaoxing} that
$$\widetilde{\gamma}(c)\leq \widetilde{\gamma}(\|u\|_{L^2}^2)\leq \widetilde{E}(u)\leq \lim_{n\to \infty}\widetilde{E}(v_n)=\widetilde{\gamma}(c),$$
which shows that $\widetilde{E}(u)=\widetilde{\gamma}(c)$. By this, we could prove further that $Q(u)=0$. Indeed, if $Q(u)<0$, then there exists a $\lambda_0\in (0,1)$ such that $Q(u_{\lambda_0})=0$, thus
$$\widetilde{\gamma}(c)\leq \widetilde{\gamma}(\|u\|_{L^2}^2)\leq \widetilde{E}(u_{\lambda_0})<\widetilde{E}(u)=\widetilde{\gamma}(c),$$
which is a contradiction. Therefore, by Lemma \ref{lm-gammac} we have already proved that $\widetilde{E}(u)=\gamma(c)$ and $Q(u)=0$. Then the proof is completed.
\end{proof}

\textbf{Remark 4.5.} As one may observe that since we work in the radial space $H_r^1$ where we have the advantage of compact embedding, the proof of Proposition \ref{proposition existence} can be more simpler. However, we here provide a new proof which can be applied in the space without radial property, provided that the functional keeps the translation invariant, even only in the $x_3$ direction, which is indeed the situation in Lemma \ref{lemma minimization problem}.\\

To prove finally  that Theorem \ref{theorem existence normalized}, we need the following two Lemmas.
\begin{lemma}\label{lm-lambda}
Assume that $\lambda_1\geq 0$, $\lambda_2,\lambda_3>0$. Let $u\in H^1$ be a weak solution of the equation
\begin{equation}\label{elliptic partial123}
-\Delta u+\omega u+\frac{\lambda_1}{|x|}u+ \lambda_2(|\cdot|^{-1}\ast |u|^2)u-
 \lambda_3|u|^p u=0, \ x\in \mathbb{R}^3.
\end{equation}
If $\omega \leq 0$, then the only solution of \eqref{elliptic partial123} fulfilling one of the following conditions is the null function: (i) $\frac{4}{3}<p<4$ and $\|u\|_{L^2}^2\leq c_0$ for some $c_0>0$ independent of $\omega$; (ii) $p=\frac{4}{3}$  and  $\|u\|_{L^2}^2<(\frac{9}{7})^{\frac{3}{2}}\lambda_3^{-\frac{3}{2}}\|R\|_{L^2}^2$.
\end{lemma}
\begin{proof}

By the Hardy-Littlewood-Sobolev and Sobolev inequalities, we have
\begin{equation}\label{Hartree estimate}
\int_{\mathbb{R}^3} (|x|^{-1}\ast|u|^{2})|u|^{2}dx\leq C\|u\|_{L^{\frac{12}{5}}}^4\leq C\|u\|_{L^2}^3\|\nabla u\|_{L^2}, \ \forall u\in H^1.
\end{equation}
Also for any $u\in H^1$, we deduce that
\begin{align}\label{Kulun estimate}
\int_{\mathbb{R}^3} \frac{|u|^2}{|x|}dx&=\int_{B_1(0)} \frac{|u|^2}{|x|}dx+\int_{B^c_1(0)} \frac{|u|^2}{|x|}dx \nonumber\\
&\leq C_1 \||x|^{-1}\chi_{B_1(0)}\|_{L^{\frac{3}{1+3\epsilon}}} \|u\|_{L^{\frac{6}{2-3\epsilon}}}^2
+C_2 \||x|^{-1}\chi_{B_1^c(0)}\|_{L^{\frac{3}{1-3\epsilon}}} \|u\|_{L^{\frac{6}{2+3\epsilon}}}^2\nonumber\\
&\leq C \Big (\|\nabla u\|_{L^2}^{1+3\epsilon}\|u\|_{L^2}^{1-3\epsilon} + \|\nabla u\|_{L^2}^{1-3\epsilon}\|u\|_{L^2}^{1+3\epsilon} \Big ),
\end{align}
with any $\epsilon\in (0,\frac{1}{3})$.
Since $u\in H^1$ be a weak solution of \eqref{elliptic partial123}, we deduce from \eqref{pohozaev identity0}, \eqref{pohozaev identity1}, \eqref{Hartree estimate} and \eqref{Kulun estimate} that
\begin{align}\label{5.49}
\omega\|u\|_{L^2}^2=&\lambda_3\|u\|_{L^{p+2}}^{p+2}-\|\nabla u\|_{L^2}^2-\lambda_1\int_{\mathbb{R}^3} \frac{|u|^{2}}{|x|}dx
-\lambda_2\int_{\mathbb{R}^3} (|x|^{-1}\ast|u|^{2})|u|^{2}dx\nonumber\\
=&\frac{4-p}{3p}\|\nabla u\|_{L^2}^2-\frac{2(p-1)\lambda_1}{3p}\int_{\mathbb{R}^3} \frac{|u|^{2}}{|x|}dx-\frac{(5p-2)\lambda_2}{6p}\int_{\mathbb{R}^3} (|x|^{-1}\ast|u|^{2})|u|^{2}dx\nonumber\\
\geq&\frac{4-p}{3p}\|\nabla u\|_{L^2}^2 - \widetilde{C}_1\Big(\|\nabla u\|_{L^2}^{\frac{3}{2}}\|u\|_{L^2}^{\frac{1}{2}} + \|\nabla u\|_{L^2}^{\frac{1}{2}}\|u\|_{L^2}^{\frac{3}{2}}\Big) - \widetilde{C}_2 \|u\|_{L^2}^3\|\nabla u\|_{L^2},
\end{align}
where $\widetilde{C}_i>0,i=1,2$, are independent of $u$. When $\frac{4}{3}< p<4$, then from \eqref{4.144} and \eqref{5.49}, we see that $\omega>0$ holds necessarily if $\|u\|_{L^2}$ is small enough.

When $p=\frac{4}{3}$, we deduce from \eqref{pohozaev identity0}, \eqref{pohozaev identity1} and \eqref{sharp gn} that
\begin{align}\label{15.49}
\omega\|u\|_{L^2}^2=3\|\nabla u\|_{L^2}^2-\frac{7\lambda_3}{5}\|u\|_{L^{p+2}}^{p+2}+\frac{\lambda_1}{4}\int_{\mathbb{R}^3} \frac{|u|^{2}}{|x|}dx
\geq\left(\frac{9\|R\|_{L^2}^\frac{4}{3}-7\lambda_3\|u\|_{L^2}^\frac{4}{3}}
{9\|R\|_{L^2}^\frac{4}{3}}\right)\|\nabla u\|_{L^2}^2,
\end{align}
which implies that $\omega>0$ if $\|u\|_{L^2}^2<(\frac{9}{7})^{\frac{3}{2}}\lambda_3^{-\frac{3}{2}}\|R\|_{L^2}^2$.
\end{proof}

\begin{lemma}\label{lemma smonotonicity}
Assume that $\lambda_1\geq 0$, $\lambda_2,\lambda_3>0$. Then when $\frac{4}{3}< p<4$, there exists a $c_1>0$  such that the mapping $c\mapsto \gamma(c)$ is strictly decreasing on the interval $(0,c_1)$, and when $p=\frac{4}{3}$, $c\mapsto \gamma(c)$ is strictly decreasing on $(\lambda_3^{-\frac{3}{2}}\|R\|_{L^2}^2,
(\frac{9}{7})^{\frac{3}{2}}\lambda_3^{-\frac{3}{2}}\|R\|_{L^2}^2)$.
\end{lemma}
\begin{proof} Indeed, by the fact that $c\mapsto \gamma(c)$ is non-increasing obtained in Lemma \ref{lemma dandiaoxing}, if we assume by contradiction that $$\gamma(c)\equiv \gamma(c_3),~~~ \forall c\in (c_2, c_3),$$
 for some $c_2,c_3>0$. Then by Proposition \ref{proposition existence}, for $c_3>0$, there exists a $u_0\in V_r(\|u_0\|_{L^2}^2)$ such that
 $$E(u_0)=\gamma(c_3)\ \mbox{ and }\ 0<\|u_0\|_{L^2}^2<c_3.$$
Hence $u_0$ is a local minimizer of $E$ on the manifold $\mathcal{M}:=\{u\in H_r^1\backslash\{0\}:~~ Q(u)=0\}$. Thus there exists a Lagrange multiplier $\lambda_0\in \mathbb{R}$ such that
$$E'(u_0)+\lambda_0Q'(u_0)=0,\ \mbox{ in }\ H_r^{-1}.$$
By the same argument in the proof of Lemma \ref{lm-lagrange}, we deduce that $\lambda_0=0$. Then $E'(u_0)u_0=0$, namely
\begin{eqnarray}\label{5.28}
\|\nabla u_0\|_{L^2}^2+\lambda_1\int_{\mathbb{R}^3} \frac{|u_0|^{2}}{|x|}dx
+\lambda_2\int_{\mathbb{R}^3} (|x|^{-1}\ast|u_0|^{2})|u_0|^{2}dx = \lambda_3\|u_0\|_{L^{p+2}}^{p+2}.
\end{eqnarray}
On the other hand, since the fact that $u_0$ is indeed a minimizer of
$$\gamma(\|u_0\|_{L^2}^2)=\inf\{E(u):~~ u\in S_r(\|u_0\|_{L^2}^2), ~~Q(u)=0\},$$
by Lemma \ref{lm-lagrange}, we know that $u_0$ is a critical point of $E$ on $S_r(\|u_0\|_{L^2}^2)$. Then there exists a $\omega_0\in \mathbb{R}$ such that
$$E'(u_0)+\omega_0u_0=0,\ \mbox{ in }\ H_r^{-1},$$
which implies
\begin{eqnarray}\label{5.29}
\|\nabla u_0\|_{L^2}^2 + \omega_0\|u_0\|_{L^2}^2 +\lambda_1\int_{\mathbb{R}^3} \frac{|u_0|^{2}}{|x|}dx
+\lambda_2\int_{\mathbb{R}^3} (|x|^{-1}\ast|u_0|^{2})|u_0|^{2}dx = \lambda_3\|u_0\|_{L^{p+2}}^{p+2}.
\end{eqnarray}
Hence by \eqref{5.28}-\eqref{5.29}, we conclude that $\omega_0\|u_0\|_{L^2}^2=0$. Then $\omega_0=0$ since $u_0\neq0$. However, by Lemma \ref{lm-lambda}, when $\frac{4}{3}<p<4$ and $c_3>0$ is small enough, or when $p=\frac{4}{3}$ and $(c_2,c_3)\subset (\lambda_3^{-\frac{3}{2}}\|R\|_{L^2}^2,
(\frac{9}{7})^{\frac{3}{2}}\lambda_3^{-\frac{3}{2}}\|R\|_{L^2}^2)$, there holds necessarily $\omega_0>0$, which is a contradiction. At this point, we have proved the strict decreasing property of $\gamma(c)$.
\end{proof}

Now we are ready to prove Theorem  \ref{theorem existence normalized}.

\begin{proof}[\textbf{Proof of Theorem \ref{theorem existence normalized}}]
Firstly we prove $(1)$. Indeed, by Proposition \ref{proposition existence} and Lemma \ref{lemma smonotonicity}, we conclude immediately that when $\frac{4}{3}<p<4$ and $c>0$ is small enough, $\gamma(c)>0$ and $\gamma(c)$ admits at least one minimizer $u_c\in V_r(c)$. In particular, by Lemma \ref{lm-lagrange},  $u_c$ is a critical point of $E$ on $S_r(c)$. Thus standardly, there exists a $\omega_c\in \mathbb{R}$ such that $(u_c,\omega_c)\in S_r(c)\times \mathbb{R}$ solves weakly the equation \eqref{elliptic partial12}. To show the behavior of $u_c$ as $c\to 0^+$. We note that by $Q(u_c)=0$, we have
\begin{eqnarray}\label{5.455}
\|\nabla u_c\|_{L^2}^2\leq \frac{3p\lambda_3}{2(p+2)}\|u_c\|_{L^{p+2}}^{p+2} \leq C\|u_c\|_{L^2}^{\frac{4-p}{2}}\|\nabla u_c\|_{L^2}^{\frac{3p}{2}}.
\end{eqnarray}
Since $\frac{4}{3}<p<4$, then by \eqref{5.455},
\begin{eqnarray}\label{5.466}
\|\nabla u_c\|_{L^2}\to +\infty,\ \mbox{ as }\ c\to 0^+.
\end{eqnarray}
and thus $E(u_c)\to +\infty$ as $c\to 0^+$, by the following calculation
\begin{align}\label{5.477}
E(u_c)=\frac{3p-4}{6p}\|\nabla u_c\|^2_{L^2}+\lambda_1\frac{3p-2}{6p}\int_{\mathbb{R}^3} \frac{|u_c|^{2}}{|x|}dx+\lambda_2\frac{3p-2}{12p}\int_{\mathbb{R}^3} (|x|^{-1}\ast|u_c|^{2})|u_c|^{2}dx.
\end{align}

Finally, we prove that $\omega_c\to +\infty$ as $c\to 0^+$.
Note that by \eqref{5.49}, we have
\begin{align}\label{5.499}
\omega_c\|u_c\|_{L^2}^2\geq\frac{4-p}{3p}\|\nabla u_c\|_{L^2}^2 - \widetilde{C}_1\Big(\|\nabla u_c\|_{L^2}^{\frac{3}{2}}\|u_c\|_{L^2}^{\frac{1}{2}} + \|\nabla u_c\|_{L^2}^{\frac{1}{2}}\|u_c\|_{L^2}^{\frac{3}{2}}\Big) - \widetilde{C}_2 \|u_c\|_{L^2}^3\|\nabla u_c\|_{L^2},
\end{align}
where $\widetilde{C}_i>0,i=1,2$, are independent of $u_c$. From \eqref{5.466} and \eqref{5.499}, we see that $\omega_c\to +\infty$ as $c\to 0^+$. Then Point $(1)$ is proved.

Point $(2)$, follows directly from Lemma \ref{lm-nonexistence11}, Proposition \ref{proposition existence} and Lemma \ref{lemma smonotonicity}.

Finally, we note that for any function $u_c\in V_r(c)$ with $E(u_c)=\gamma(c)$, by Lemma \ref{lm-lagrange},  it is a critical point of $E$ on $S_r(c)$ and then a radial solution of \eqref{elliptic partial12} with $\omega=\omega_c\in \mathbb{R}$. Then by following the proof of Theorem \ref{th instability no partial}, we conclude that the standing wave $\psi(t,x)=e^{i\omega_c t}u_c(x)$ is strongly unstable. See also similar proofs in \cite[Theorem 1.5]{jean16siam} or \cite[Theorem 1.6]{jean-luo}. Particularly, we note that if $\omega_c>0$ (which is the case as $c>0$ small), then Point $(3)$ is indeed a consequence of Theorem \ref{th instability no partial}. At this point, we complete the proof.
\end{proof}

\section{Normalized solutions with partial confine}

In this section we prove the existence of stable ground state by considering two types of minimization problems. To obtain the stability, we need first to show the global existence result of Theorem \ref{theorem global existence}.

\begin{proof}[\textbf{Proof of Theorem \ref{theorem global existence}}.]
(1) and (2) have been proved in \cite{f18dcds}. Here we only prove (3) and (4), in both cases, we always assume that $\lambda_1,\lambda_2\in \mathbb{R}$ and $\lambda_3>0$.

Recall that, by the Hardy inequality and the Young inequality with $\varepsilon$, we have
\begin{align}\label{i1}
\frac{|\lambda_1|}{2}\int_{\mathbb{R}^3}\frac{|\psi(t,x)|^2}{|x|}dx\leq C\|\psi(t)\|_{L^2}\|\psi(t)\|_{\tilde{X}}\leq \varepsilon_1\|\psi(t)\|_{\tilde{X}}^2+C(\varepsilon_1, \lambda_1)\|\psi(t)\|_{L^2}^2,
\end{align}
and
\begin{align}\label{i2}
\frac{|\lambda_2|}{4}\int_{\mathbb{R}^3}(|x|^{-1}*|\psi(t)|^{2})|\psi(t,x)|^{2}dx\leq C\|\psi(t)\|_{L^2}^3\|\psi(t)\|_{\tilde{X}}\leq\varepsilon_2\|\psi(t)\|_{\tilde{X}}^2+C(\varepsilon_2, \lambda_2)\|\psi(t)\|_{L^2}^6,
\end{align}
where $\varepsilon_i>0$ is arbitrary and $C(\varepsilon_i, \lambda_i)>0, i=1,2$.

When $p=\frac{4}{3}$, using the conservation of mass and energy, \eqref{i1}, \eqref{i2} and \eqref{sharp gn},
we have
\begin{align*}
\frac{1}{2}\|\psi(t)\|_{\tilde{X}}^2\leq& E_b(\psi_0)+C(\varepsilon_1, \varepsilon_2, \lambda_1, \lambda_2)(\|\psi_0\|_{L^2}^2+\|\psi_0\|_{L^2}^6)+(\varepsilon_1+\varepsilon_2 +\frac{\lambda_3\|\psi_0\|_{L^2}^{\frac{4}{3}}}{2\|R\|_{L^2}^{\frac{4}{3}}})\|\psi(t)\|_{\tilde{X}}^2.
\end{align*}
When $\|\psi_0\|_{L^2}<\lambda_3^{-\frac{3}{4}}\|R\|_{L^2}$, by
taking $\varepsilon_1$ and $\varepsilon_2$ small enough such that
\[
\frac{1}{2}
-\lambda_3\frac{\|\psi_0\|_{L^2}^{\frac{4}{3}}}{2\|R\|_{L^2}^{\frac{4}{3}}}
-\varepsilon_1-\varepsilon_2>0,
\]
we then obtain the global existence.

When $\frac{4}{3}<p<4$. Also by the conservation of mass and energy, \eqref{i1}, \eqref{i2} and the Gagiliardo-Nirenberg inequality, we have
\begin{align*}
\frac{1}{2}\|\psi(t)\|_{\tilde{X}}^2\leq& E_b(\psi_0)+C(\varepsilon_1, \varepsilon_2, \lambda_1, \lambda_2)(\|\psi_0\|_{L^2}^2+\|\psi_0\|_{L^2}^6)+(\varepsilon_1+\varepsilon_2)\|\psi(t)\|_{\tilde{X}}^2
+C\|\psi_0\|_{L^2}^{\frac{4-p}{2}
}\|\psi(t)\|_{\tilde{X}}^{\frac{3p}{2}}.
\end{align*}
Taking $\varepsilon_1=\varepsilon_2=\frac{1}{8}$, this implies that
\begin{align*}
\frac{1}{4}\|\psi(t)\|_{\tilde{X}}^2\leq C(E_b(\psi_0),\|\psi_0\|_{L^2})
+C\|\psi_0\|_{L^2}^{\frac{4-p}{2}
}\|\psi(t)\|_{\tilde{X}}^{\frac{3p}{2}}.
\end{align*}
Therefore, for every given $\kappa>0$, with $\|\psi_0\|_{\tilde{X}}<\kappa$ and $\|\psi_0\|_{L^2}<1$, denoting $C(E_b(\psi_0),\|\psi_0\|_{L^2})=C(\kappa)$, then applying Lemma \ref{lemma Bootstrap}, there exists $0<c_{\kappa}<1$ such that $\|\psi_0\|_{L^2}< c_{\kappa}$ and $\|\psi(t)\|_{\tilde{X}}\leq C({\kappa})$ for all $t\in [0,T^*)$. Thus the proof ends.
\end{proof}

\subsection{Solutions as global minimizers}

In this subsection, we prove Theorem \ref{Theorem globalmini}.
To begin with, we first show that
\begin{lemma}\label{lm3.1}
Assume that $\lambda_i\in \mathbb{R}, i=1,2,3$. Let $R$ be the unique ground state of \eqref{elliptic classical}. Then for any $c>0$ if $0<p<\frac{4}{3}$ or $0<c<|\lambda_3|^{-\frac{3}{2}}\|R\|_{L^2}^2$ if $p=\frac{4}{3}$, the functional $E_b(u)$ is bounded from below on $\bar{S}(c)$, and $m(c)>-\infty$.
\end{lemma}
\begin{proof}
Applying the Gagliardo-Nirenberg inequality,  \eqref{Hartree estimate}, \eqref{Kulun estimate},
we can conclude that, there exist constants $\widetilde{C}_i>0, i=1,2,3$, depending on only $c,p$, such that
\begin{eqnarray}\label{energy estimate}
E_b(u)\geq \dfrac{\min\{1,b^2\}}{2}(\|u\|_X^2-c)-\dfrac{|\lambda_1|\widetilde{C}_1}{2}\Big (\|u\|_{X}^{\frac{3}{2}} + \|u\|_{X}^{\frac{1}{2}} \Big ) - \dfrac{|\lambda_2|\widetilde{C}_2}{4}\|u\|_X -  \dfrac{|\lambda_3|\widetilde{C}_3}{p+2}\|u\|_X^{\frac{3p}{2}},
\end{eqnarray}
for all $u\in \Bar{S}(c)$.
If we assume that $\lambda_i \in \mathbb{R}, i=1,2,3,$ and $0<p<\frac{4}{3}$, then from \eqref{energy estimate} we see immediately  that for any $c>0$, $E_b(u)$ is bounded from below on $\Bar{S}(c)$, and then that $m(c)>-\infty$.

When $p=\frac{4}{3}$ and $0<c<|\lambda_3|^{-\frac{3}{2}}\|R\|_{L^2}^2$, this result can be easily obtained by using the sharp inequality \eqref{sharp gn} and \eqref{energy estimate}.
\end{proof}

As we declared in Remark 1.6, due to the appearance of the Coulomb potential, the functional does not keep invariant by translation, which brings new difficulties on the compactness of a minimizing sequence for $m(c)$. To overcome this obstacle, we consider the following limit minimization problem:
\begin{eqnarray}\label{limitmini1}
m^{\infty}(c):=\inf_{u\in \Bar{S}(c)}E^{\infty}_{b}(u),
\end{eqnarray}
where the limit energy $E^{\infty}_{b}$ is given by \eqref{limitenergy}. Denote by $\mathcal{M}_c^{\infty}$ the set of all minimizers of $m^{\infty}(c)$. We first establish an existence result for $m^{\infty}(c)$.
\begin{proposition}\label{prop-limit1}
Assume that $\lambda_2\leq 0, \lambda_3>0$. Then for any $c>0$ if $0<p<\frac{4}{3}$ or $0<c<\lambda_3^{-\frac{3}{2}}\|R\|_{L^2}^2$ if $p=\frac{4}{3}$, any minimizing sequence of $m^{\infty}(c)$ is pre-contract, up to a translation, in $X$. In particular, $\mathcal{M}_c^{\infty}\neq \emptyset$.
\end{proposition}
\begin{proof}
Let $\{u_n\}$ be an arbitrary minimizing sequence of $m^{\infty}(c)$, namely
\begin{equation}\label{o3}
E^{\infty}_b(u_n)\rightarrow m^{\infty}(c)~~\mbox{ and }~~\|u_n\|_{L^2}^2=c.
\end{equation}
Then by \eqref{energy estimate}, clearly $\{u_n\}$ is bounded in $X$. Further, we claim that
$$\liminf_{n\to \infty}\|u_n\|_{L^{p+2}}\geq \delta, $$ for some constant $\delta>0$.

Indeed, if there exists a subsequence, still denoted by $\{u_n\}$, such that $\|u_n\|_{L^{p+2}}\rightarrow 0$ as $n\rightarrow \infty$, then by the interpolation, $\|u_n\|_{L^q}\rightarrow 0$ as $n\rightarrow \infty$ for all $q\in (2,6)$. Thus by  \eqref{Hartree estimate},
\[
\int_{\mathbb{R}^3} (|x|^{-1}\ast|u_n|^{2})|u_n|^{2}dx\rightarrow 0  ~~\mbox{as}~~n\rightarrow \infty.
\]
We consequently obtain that
\begin{align}\label{xiajie estimate}
m^{\infty}(c)= E^{\infty}_b(u_n) + o_n(1) =
\frac{1}{2}\int_{\mathbb{R}^3}(|\nabla u_n|^2+b^2(x_1^2+x_2^2)|u_n|^2)dx
+o_n(1)\geq \frac{c\Lambda_0}{2}+o_n(1),
\end{align}
where $\Lambda_0$ is given in \eqref{Lambda_0}. Noting the definition of $\lambda_0$ in \eqref{lambda_0}, since the Sobolev space
$H:=\{v\in H^1(\mathbb{R}^2):~\int_{\mathbb{R}^2}(x_1^2+x_2^2)|v|^2dx<\infty\}$ is compactly embedded in $L^2(\mathbb{R}^2)$, it is standard to show
that $\lambda_0$ is achieved by some $w\in H^1(\mathbb{R}^2)$ with $\int_{\mathbb{R}^2}|w|^2dx=1$. Let $\varphi\in H^1(\mathbb{R})$ satisfy $\int_{\mathbb{R}}|\varphi|^2dx=c$
and for any $\lambda>0$, set
\[
u^\lambda(x)=w(x')\varphi^\lambda(x_3),~~ \mbox{where} ~~ \varphi^\lambda(x_3)=\lambda^{\frac{1}{2}}\varphi(\lambda x_3).
\]
Then $u^\lambda \in \Bar{S}(c), \forall \lambda>0$, and by Lemma \ref{lemma Jeanjean} we have
\begin{align}\label{shangjie estimate}
E_b(u^\lambda)=& \frac{c\Lambda^0}{2}+\frac{\lambda^2}{2}\int_{\mathbb{R}} |\partial_{x_3}\varphi|^{2}dx_3+\frac{\lambda_1\lambda}{2}\int_{\mathbb{R}^3} \frac{|w(x')|^{2}|\varphi(x_3)|^{2}}{\sqrt{\lambda^2x_1^2+\lambda^2x_2^2+x_3^2}}dx
\nonumber\\&+\frac{\lambda_2\lambda}{4}\int_{\mathbb{R}^3}\int_{\mathbb{R}^3} \frac{|w(x')\varphi(x_3)|^{2}|w(y')\varphi(y_3)|^{2}}{\sqrt{\lambda^2|x_1-y_1|^2
+\lambda^2|x_2-y_2|^2+|x_3-y_3|^2}}dxdy \nonumber\\&-\frac{\lambda_3\lambda^{\frac{p}{2}}}{p+2}\int_{\mathbb{R}^2} |w(x')|^{p+2}dx'\int_{\mathbb{R}} |\varphi(x_3)|^{p+2}dx_3.
\end{align}
Note that $\lambda_3>0$ and $0<\frac{p}{2}<1$ if $0<p\leq\frac{4}{3}$, then we observe from \eqref{shangjie estimate} that
$E^{\infty}_b(u^{\lambda})<\frac{c\Lambda^0}{2}$, for $\lambda>0$ small enough. This implies that
$$m^{\infty}(c)<\frac{c\Lambda^0}{2},$$
which contradicts with \eqref{xiajie estimate}. Hence by Lemma \ref{lemma compactness lemma I}, there exist a sequence $\{z_n\}\subset \mathbb{R}$ and $u\in X\backslash \{0\}$, such that
$$u_n(x_1,x_2, x_3-z_n) \rightharpoonup u, \ \mbox{in } X. $$

Denote $v_n:=u_n(x_1,x_2, x_3-z_n)$, we now prove the pre-compactness of $\{v_n\}_{n\geq 1}$. Indeed, clearly $0<\|u\|_{L^2}^2\leq \lim\limits_{n\to \infty}\|v_n\|_{L^2}^2=c$. If $\|u\|_{L^2}^2<c$, then we denote  $a:=\frac{\sqrt{c}}{\|u\|_{L^2}}$, thus $a\geq 1, au\in \Bar{S}(c)$ and
\begin{align}\label{energy scaling}
E^{\infty}_b(au)=& \frac{a^2}{2}\|\nabla u\|_{L^2}^2+\frac{a^2b^2}{2}\int_{\mathbb{R}^3} (x_1^2+x_2^2)|u|^{2}dx
\nonumber\\&+\frac{\lambda_2a^4}{4}\int_{\mathbb{R}^3} (|x|^{-1}\ast|u|^{2})|u|^{2}dx -\frac{\lambda_3a^{p+2}}{p+2}\|u\|_{L^{p+2}}^{p+2},
\end{align}
which implies that
\begin{align}\label{energy scaling1}
E^{\infty}_b(u)= \frac{E^{\infty}_b(au)}{a^2}-\frac{\lambda_2(a^2-1)}{4}\int_{\mathbb{R}^3} (|x|^{-1}\ast|u|^{2})|u|^{2}dx +\frac{\lambda_3(a^p-1)}{p+2}\|u\|_{L^{p+2}}^{p+2}.
\end{align}
Similarly, let $a_n:=\frac{\sqrt{c}}{\|v_n-u\|_{L^2}}$, then $a_n\geq 1, a_n(v_n-u)\in \Bar{S}(c)$ and
\begin{align}\label{energy scaling2}
E^{\infty}_b(v_n-u)=& \frac{E^{\infty}_b(a_n(v_n-u))}{a_n^2}-\frac{\lambda_2(a_n^2-1)}{4}\int_{\mathbb{R}^3} (|x|^{-1}\ast|v_n-u|^{2})|v_n-u|^{2}dx \nonumber\\&+\frac{\lambda_3(a_n^p-1)}{p+2}\|v_n-u\|_{L^{p+2}}^{p+2}.
\end{align}
Therefore, we have
\begin{align}\label{energy scaling3}
&E^{\infty}_b(v_n)=E^{\infty}_b(v_n-u)+E^{\infty}_b(u)+o_n(1)\nonumber\\\geq & m^{\infty}(c)\left(\frac{1}{a^2}+\frac{1}{a_n^2}\right)+\frac{\lambda_3(a^p-1)}{p+2}\|u\|_{L^{p+2}}^{p+2}
+\frac{\lambda_3(a_n^p-1)}{p+2}\|v_n-u\|_{L^{p+2}}^{p+2}
\nonumber\\&-\frac{\lambda_2(a^2-1)}{4}\int_{\mathbb{R}^3} (|x|^{-1}\ast|u|^{2})|u|^{2}dx-\frac{\lambda_2(a_n^2-1)}{4}\int_{\mathbb{R}^3} (|x|^{-1}\ast|v_n-u|^{2})|v_n-u|^{2}dx.
\end{align}
Note that since $E^{\infty}_b$ keeps invariant by translation at the $x_3$ direction, then $E^{\infty}_b(v_n)=m^\infty(c)+o_n(1), \frac{1}{a^2}+\frac{1}{a_n^2}=1+o_n(1)$, and  $a\geq 1, a_n\geq 1$. Since $\lambda_2\leq0$, $\lambda_3>0$, we deduce from \eqref{energy scaling3} that $a=1$. Namely $v_n\rightarrow u$ in $L^2$. In particular, by the interpolation, $v_n\rightarrow u$ in $L^q$, $q\in [2,6)$. Therefore, by
\[
\int_{\mathbb{R}^3} (|x|^{-1}\ast|v_n|^{2})|v_n|^{2}dx\rightarrow
\int_{\mathbb{R}^3} (|x|^{-1}\ast|u|^{2})|u|^{2}dx,
\]
and further by the weak lower semi-continuity of the $X$ norm, we see that
\[
E^{\infty}_b(u)\leq \lim_{k\rightarrow \infty}E^{\infty}_b(v_{n})=m^{\infty}(c).
\]
Thus by the definition of $m^{\infty}(c)$, we have $E^{\infty}_b(u)=m^{\infty}(c)$. Hence, $E^{\infty}_b(v_n)\rightarrow E^{\infty}_b(u)$, and $\|v_{n}\|_{X}\rightarrow \|u\|_{X}$, which implies that $u_n(x_1,x_2, x_3-z_n)\rightarrow u$ strongly in $X$. In particular, $u\in \bar{S}(c)$ is a minimizer of $m^{\infty}(c)$. Thus $\mathcal{M}^{\infty}_c\neq \emptyset$.
\end{proof}

Note that $m(c)=m^{\infty}(c)$ if $\lambda_1=0$. When $\lambda_1\neq 0$, we first prove the following non-existence result.
\begin{theorem}\label{th-nonexistence1}
Assume that $\lambda_1>0, \lambda_2\leq 0, \lambda_3>0$. Then for any $c>0$ if $0<p<\frac{4}{3}$ or $0<c<\lambda_3^{-\frac{3}{2}}\|R\|_{L^2}^2$ if $p=\frac{4}{3}$, we have $\mathcal{M}_c=\emptyset$, namely $m(c)$ has no any minimizers.
\end{theorem}
\begin{proof}
If $\lambda_1>0$, we claim that
$$m^{\infty}(c)=m(c).$$
Indeed, since $\lambda_1>0$, for all $u\in \bar{S}(c)$, we have $E_b(u)\geq E^{\infty}_b(u)$, which implies that $m(c)\geq m^{\infty}(c)$. On the other hand, for all $u\in \bar{S}(c)$,
$$m(c)\leq E_b(u(\cdot -n\vec{e}_3))=E^{\infty}_b(u)+\frac{\lambda_1}{2}\int_{\mathbb{R}^3}\frac{|u|^2}{\sqrt{x_1^2+x_2^2+(x_3+n)^2}}dx,\ \vec{e}_3:=(0,0,1),$$
letting $n\to \infty$, we then have
$$m(c)\leq E^{\infty}_b(u)\Rightarrow m(c)\leq m^{\infty}(c).$$
Then $m(c)= m^{\infty}(c)$ follows.

Now we argue by contradiction to assume that $\mathcal{M}_c\neq \emptyset$. Namely there exists a $u_0\in \bar{S}(c)$, such that $E_b(u_0)=m(c)$. Then
\begin{align}\label{5.13}
m(c)=E_b(u_0)=& E^{\infty}_b(u_0)+ \frac{\lambda_1}{2}\int_{\mathbb{R}^3}\frac{|u_0|^2}{|x|}dx \nonumber\\
\geq & m^{\infty}(c) + \frac{\lambda_1}{2}\int_{\mathbb{R}^3}\frac{|u_0|^2}{|x|}dx = m(c) + \frac{\lambda_1}{2}\int_{\mathbb{R}^3}\frac{|u_0|^2}{|x|}dx,
\end{align}
since $\lambda_1>0$, thus $\int_{\mathbb{R}^3}\frac{|u_0|^2}{|x|}dx=0$, i.e. $u_0=0$ a.e. in $\mathbb{R}^3$. This contradicts with the fact that $u_0\in \bar{S}(c)$. At this point, the proof ends.
\end{proof}

As for $\lambda_1<0$, we have
\begin{lemma}\label{lm-limit-nonlimit}
Assume that $\lambda_1<0$, $\lambda_2\leq 0$, $\lambda_3>0$. Then for any $c>0$ if $0<p<\frac{4}{3}$ or $0<c<\lambda_3^{-\frac{3}{2}}\|R\|_{L^2}^2$ if $p=\frac{4}{3}$, we have
\begin{eqnarray}\label{3.133}
m(c)<m^{\infty}(c),
\end{eqnarray}
and
\begin{eqnarray}\label{3.134}
m(c)<m(c-\mu)+m(\mu),\quad \forall \mu\in (0,c).
\end{eqnarray}
\end{lemma}

\begin{proof}
By Proposition \ref{prop-limit1}, $m^{\infty}(c)$ is attained at some $u_0\in \bar{S}(c)$. Since $\lambda_1<0$ and $u_0\neq 0$, then
\begin{align*}
m(c)\leq E_b(u_0)=E^{\infty}_b(u_0)+ \frac{\lambda_1}{2}\int_{\mathbb{R}^3}\frac{|u_0|^2}{|x|}dx= m^{\infty}(c) + \frac{\lambda_1}{2}\int_{\mathbb{R}^3}\frac{|u_0|^2}{|x|}dx < m^{\infty}(c).
\end{align*}
Thus $\eqref{3.133}$ is verified.

To show \eqref{3.134}, it is enough to prove
\begin{eqnarray}\label{3.135}
m(tc)<tm(c),\quad \forall t>1.
\end{eqnarray}
Indeed, by \eqref{3.135}, we have for all $\mu\in (0,c)$ that
\begin{align*}
m(c)=&\frac{c-\mu}{c}m(c)+\frac{\mu}{c}m(c)<\frac{c-\mu}{c}\cdot \frac{c}{c-\mu}m( c-\mu)+\frac{\mu}{c}\cdot \frac{c}{\mu}m(\mu)=m(c-\mu)+m(\mu).
\end{align*}

Now to verify \eqref{3.135}, we let $\{u_n\}\subset \bar{S}(c)$ be such that $m(c)=E_b(u_n)+o_n(1)$. Then by \eqref{xiajie estimate} and \eqref{shangjie estimate} we could prove that
$$\liminf_{n\to \infty}\|u_n\|_{L^{p+2}}\geq \delta, $$ for some constant $\delta>0$. Thus for all $t>1$, $\sqrt{t}u_n\in \bar{S}(tc)$, and
\begin{align*}
m(tc)\leq  \liminf_{n\to \infty} E_b(\sqrt{t}u_n)
\leq & t \liminf_{n\to \infty} E_b(u_n) -\liminf_{n\to \infty}\frac{\lambda_3}{p+2}(t^{(p+2)/2}-t)\|u_n\|_{L^{p+2}}^{p+2}\\
\leq & t m(c) - \frac{\lambda_3}{p+2}(t^{(p+2)/2}-t)\delta^{p+2}<tm(c).
\end{align*}
\end{proof}

\begin{proposition}\label{prop-globalmini}
Assume that $\lambda_1<0, \lambda_2\leq 0, \lambda_3>0$. Then for any $c>0$ if $0<p<\frac{4}{3}$ or $0<c<\lambda_3^{-\frac{3}{2}}\|R\|_{L^2}^2$ if $p=\frac{4}{3}$, any minimizing sequence of $m(c)$ is pre-compact in $X$. In particular, $\mathcal{M}_c\neq \emptyset$.
\end{proposition}

\begin{proof} Let $\{u_n\}$ be an arbitrary minimizing sequence of $m(c)$, namely
\begin{equation}\label{o31}
E_b(u_n)\rightarrow m(c)~~\mbox{ and }~~\|u_n\|_{L^2}^2=c.
\end{equation}
Then by \eqref{energy estimate}, $\{u_n\}$ is bounded in $X$, and up to a subsequence, there exists a $u\in X$, such that $u_n \rightharpoonup u $ in $X$. In particular, since $X\hookrightarrow H^1(\mathbb{R}^3)$ continuously, then $u_n \rightharpoonup u $ in $H^1$ and $u_n \to u$ in $L^q_{loc}(\mathbb{R}^3), \forall q\in [2,6)$. Now we claim that $u\neq 0$. Indeed, if so, then by the fact that $u_n \to 0$ in $L^q_{loc}(\mathbb{R}^3)$ one could easily prove that
\begin{eqnarray}\label{3.136}
\int_{\mathbb{R}^3}\frac{|u_n|^2}{|x|}dx=o_n(1),\ \mbox{ as }\ n\to \infty.
\end{eqnarray}
Thus as $n \to \infty$, we have
$$m(c)=E_b(u_n)+o_n(1)=E^{\infty}_b(u_n)+o_n(1).$$
Since $u_n\in \bar{S}(c)$, we then deduce that
$$m^{\infty}(c)\leq E^{\infty}_b(u_n)=m(c)+o_n(1),$$
which implies $m^{\infty}(c)\leq m(c)$. This contradicts with \eqref{3.133} in Lemma \ref{lm-limit-nonlimit}. Hence $u\neq 0$.

Now we claim that $u$ is a minimizer of $m(c)$. Indeed, since $u\neq 0$, by the Brezis-Lieb Lemma, we have
\begin{align}\label{3.137}
E_b(u_n)=E_b(u_n-u)+E_b(u)+o_n(1)\geq m(\|u_n-u\|_{L^2}^2)+m(\|u\|_{L^2}^2)+o_n(1).
\end{align}
Denote $c_0:=\|u\|_{L^2}^2, c_n:=\|u_n-u\|_{L^2}^2$, then $c_n=c-c_0+o_n(1)$. Choosing a subsequence of $\{u_n\}$ (still denote by $\{u_n\}$) being such that $c_n\to (c-c_0)^{-}$, then by \eqref{3.135} and \eqref{3.137},
\begin{align}\label{3.1371}
m(c)\geq m(c_n)+m(c_0)+o_n(1)\geq& \frac{c_n}{c-c_0} m(c-c_0)+ m(c_0)+o_n(1)\nonumber\\
\Longrightarrow \qquad  m(c)\geq & m(c-c_0)+ m(c_0).
\end{align}
If $\|u\|_{L^2}^2<c$, then \eqref{3.1371} contradicts with \eqref{3.134} in Lemma \ref{lm-limit-nonlimit}.  Hence $u_n\to u$ in $L^2$, and by the interpolation, $u_n\to u$ in $L^{q}(\mathbb{R}^3)$, for all $q\in [2,6)$. Then by \eqref{Kulun estimate} and \eqref{Hartree estimate}, we have
\[
\int_{\mathbb{R}^3}\frac{|u_n|^{2}}{|x|}dx\rightarrow\int_{\mathbb{R}^3} \frac{|u|^{2}}{|x|}dx
~~\mbox{and}~~\int_{\mathbb{R}^3} (|x|^{-1}\ast|u_n|^{2})|u_n|^{2}dx\rightarrow
\int_{\mathbb{R}^3} (|x|^{-1}\ast|u|^{2})|u|^{2}dx,
\]
and further,
\[
E_b(u)\leq \lim_{n\rightarrow \infty}E_b(u_{n})=m(c).
\]
Then $E_b(u)=m(c)$ and  $\|u_{n}\|_{X}\rightarrow \|u\|_{X}$, which implies that $u_n\rightarrow u$ strongly in $X$. In particular, $u\in \bar{S}(c)$ is a minimizer of $m(c)$. Thus $\mathcal{M}_c\neq \emptyset$.
\end{proof}

\begin{proof}[\textbf{Proof of Theorem \ref{Theorem globalmini}}.]
Having Proposition \ref{prop-globalmini} at hand, the rest of Theorem \ref{Theorem globalmini} can be proved standardly.
\end{proof}

\subsection{Solutions as local minimizers}

In this subsection, we prove Theorem \ref{th-localmini}. We first establish a local minima structure for $E_b(u)$ on $\Bar{S}(c)$.
\begin{lemma}\label{lmwelldefine}
Assume that $\lambda_1\in \mathbb{R}, \lambda_2\leq 0, \lambda_3>0$ and $\frac{4}{3}\leq p<4$, all being fixed, then there exists a $r_0>0$, such that for every given $r>r_0$, there exists a $c_r$ with $0<c_r<1$, we have
\begin{eqnarray}\label{3.15}
\Bar{S}(c)\cap B(\frac{rc}{2})\neq \emptyset, \ \forall c>0,
\end{eqnarray}
\begin{eqnarray}\label{3.16}
\inf_{u\in \Bar{S}(c)\cap B(\frac{rc}{2})}E_b(u) < \inf_{u\in \Bar{S}(c)\cap (B(r)\setminus B(rc^{\frac{1}{4}}) ) }E_b(u), \ \forall c<c_r.
\end{eqnarray}
\end{lemma}
\begin{proof}
Let $u_0\in X$ be such that $\|u_0\|_{L^2}^2=2, \|u_0\|_{\tilde{X}}^2=r_0$. Then for all $c>0$, letting $u_c:=\sqrt{\frac{c}{2}}u_0$, we have
$$\|u_c\|_{L^2}^2=c\quad \mbox{and}\quad \|u_c\|_{\tilde{X}}^2=\frac{r_0c}{2}< \frac{rc}{2},\ \forall r>r_0,$$
namely $u_c\in \Bar{S}(c)\cap B(\frac{rc}{2})$. Thus \eqref{3.15} is verified.


To verify \eqref{3.16}, we note that $\Bar{S}(c)\cap B(\frac{rc}{2})\subset \Bar{S}(c)\cap B(r)$ if $c<1$. Thus for any $u\in \Bar{S}(c)\cap B(r)$,
we have
\begin{align}\label{5.14}
\int_{\mathbb{R}^3} |u|^2dx =& \frac{1}{2}\int_{\mathbb{R}^3} |u|^2 div (x_1,x_2)dx =-\frac{1}{2} \int_{\mathbb{R}^3}\Big(\partial_{x_1}(|u|^2)x_1 +\partial_{x_2}(|u|^2)x_2\Big)dx,\nonumber\\
\leq& \frac{1}{2}\Big(\|\nabla u\|_{L^2}^2+\int_{\mathbb{R}^3}(x_1^2+x_2^2)|u|^2dx\Big )\leq \max\{\frac{1}{2b^2},\frac{1}{2}\}\|u\|_{\tilde{X}}^2.
\end{align}
This, together with the Hardy inequality (see e.g. \cite[Lemma 7.6.1]{ca2003}), implies that
 \begin{eqnarray}\label{5.143}
\int_{\mathbb{R}^3} \frac{|u|^{2}}{|x|}dx\leq \|u\|_{L^2}\|\nabla u\|_{L^2} \leq \sqrt{\max\{\frac{1}{2b^2},\frac{1}{2}\}}\|u\|_{\tilde{X}}^2.
\end{eqnarray}
Thus, by the Gagliardo-Nirenberg inequality, \eqref{Hartree estimate} and \eqref{5.143}, we have
$$\left\{
  \begin{array}{ll}
    E_b(u)\geq \frac{1}{2}\|u\|_{\tilde{X}}^2-|\lambda_1|c^{\frac{1}{2}}\|u\|_{\tilde{X}} -\frac{C_1|\lambda_2|}{4}c^{\frac{3}{2}}\|u\|_{\tilde{X}}- \frac{\lambda_3C_2}{p+2} c^{\frac{4-p}{4}}\|u\|_{\tilde{X}}^{\frac{3p}{2}}, \\
    E_b(u) \leq \frac{1}{2}\Big (1+|\lambda_1|\cdot\sqrt{\max\{\frac{1}{2b^2},\frac{1}{2}\}}\Big)\|u\|_{\tilde{X}}^2,
  \end{array}
\right.$$
where $C_i>0,i=1,2$ depending only on $p$. Denote $K_0:=\frac{1}{2}\Big (1+|\lambda_1|\cdot\sqrt{\max\{\frac{1}{2b^2},\frac{1}{2}\}}\Big), K_1:=\frac{C_1|\lambda_2|}{4}$ and $K_2:=\frac{\lambda_3C_2}{p+2},$ then we define the following functions:
$$
\left\{
  \begin{array}{ll}
    i(t):=\frac{1}{2}t^2-|\lambda_1|\cdot c^{\frac{1}{2}}t-K_1\cdot c^{\frac{3}{2}}t-K_2 \cdot c^{\frac{4-p}{4}}t^{\frac{3p}{2}}, \ t>0, \\
    j(t):=K_0t^2, \ t>0.
  \end{array}
\right.
$$

Notice that for any $r>r_0$, there exists a $c_r\ll 1$, such that we have not only $B(\frac{rc}{2})\subset B(rc^{\frac{1}{4}}), \forall c<c_r$, but also that
for all $t\in (rc^{\frac{1}{4}},r)$ with $c<c_r$,
\begin{align*}
i(t)=&t^2\Big(\frac{1}{2}-|\lambda_1|\cdot c^{\frac{1}{2}}t^{-1}-K_1 \cdot c^{\frac{3}{2}}t^{-1}-K_2 \cdot c^{\frac{4-p}{4}}t^{\frac{3p}{2}-2} \Big)
    \\\geq& r^2c^{\frac{1}{2}}\Big(\frac{1}{2}-|\lambda_1|\cdot c^{\frac{1}{4}}r^{-1}- K_1 \cdot c^{\frac{5}{4}}r^{-1}-K_2 \cdot c^{\frac{4-p}{4}}r^{\frac{3p}{2}-2}\Big)\\
    >&\frac{3r^2c^{\frac{1}{2}}}{8}>j(\frac{rc}{2}).
\end{align*}
This implies that
$$\inf_{t\in (rc^{\frac{1}{4}},r)}i(t)>j(\frac{rc}{2}),\quad \forall c<c_r,$$
which yields to \eqref{3.16}.
\end{proof}

Indeed, using the same idea in proving Lemma \ref{lmwelldefine}, we could prove a similar local minima structure for the case $\lambda_2>0$. Denote
$$B_1(r):=\{u\in X:~ \|u\|_{X_1}^2\leq r \},$$
and for $\lambda_1\in \mathbb{R}, \lambda_2> 0$,
\begin{eqnarray}\label{Xnorm1}
\|u\|_{X_1}^2:=\|\nabla u\|_{L^2}^2+b^2\int_{\mathbb{R}^3} (x_1^2+x_2^2)|u|^{2}dx+|\lambda_1|\int_{\mathbb{R}^3} \frac{|u|^{2}}{|x|}dx+\frac{\lambda_2}{2}\int_{\mathbb{R}^3} (|x|^{-1}\ast|u|^{2})|u|^{2}dx.
\end{eqnarray}
Then
\begin{lemma}\label{lmwelldefine1}
Assume that $\lambda_1\in \mathbb{R}, \lambda_2> 0, \lambda_3>0$ and $\frac{4}{3}\leq p<4$, all being fixed, then there exists a $r_0>0$, such that for every given $r>r_0$, there exists a $c_r$ with $0<c_r<1$, we have
\begin{eqnarray}\label{3.150}
\Bar{S}(c)\cap B_1(\frac{rc}{2})\neq \emptyset, \ \forall c>0,
\end{eqnarray}
\begin{eqnarray}\label{3.160}
\inf_{u\in \Bar{S}(c)\cap B_1(\frac{rc}{2})}E_b(u) < \inf_{u\in \Bar{S}(c)\cap (B_1(r)\setminus B_1(rc^{\frac{1}{4}}) ) }E_b(u), \ \forall c<c_r.
\end{eqnarray}
\end{lemma}

To show the local minima is attained, we need to consider the following limit problem:
\begin{eqnarray}\label{limitlocmini1}
(\bar{m}_{c}^r)^{\infty}:=\inf_{u\in \bar{S}(c)\cap B(r)}E^{\infty}_{b}(u).
\end{eqnarray}

Then following the argument in the proof of Theorem \ref{th-nonexistence1}, we could prove that
\begin{lemma}\label{lm-limit-nonlimit02}
Assume that $\lambda_1>0, \lambda_2\leq 0, \lambda_3>0$ and $\frac{4}{3}\leq p<4$, all being fixed, then for all $c<c_r$ (where $r$ and $c_r$ are given in Lemma \ref{lmwelldefine}, and letting $r$ be larger, $c_r$ smaller if necessary), we have
\begin{eqnarray}\label{5.2812}
\bar{m}_{c}^r = (\bar{m}_{c}^r)^{\infty},
\end{eqnarray}
and $\bar{m}_{c}^r$ has no any minimizers.
\end{lemma}

Hence in what follows we consider the case $\lambda_1\leq 0$.
\begin{lemma}\label{lm-limit-nonlimit2}
Assume that $\lambda_1\leq0, \lambda_2\leq 0, \lambda_3>0$ and $\frac{4}{3}\leq p<4$, all being fixed, then for all $c<c_r$ (where $r$ and $c_r$ are given in Lemma \ref{lmwelldefine}, and letting $r$ be larger, $c_r$ smaller if necessary), we have
\begin{eqnarray}\label{5.2811}
\bar{m}_{c}^r < \bar{m}_{c-\mu}^r + \bar{m}_{\mu}^r,\quad \forall \mu\in (0,c),
\end{eqnarray}
\end{lemma}
\begin{proof}
We first prove the following strict monotonicity:
\begin{eqnarray}\label{smonotonicity}
t\bar{m}_{s}^r <  s\bar{m}_{t}^r,\quad \forall~~~ 0<t<s<\min\{1,c_r\}.
\end{eqnarray}
Indeed, let $\{v_n\}\subset \Bar{S}(t)\cap B(r)$ with $E_{b}(v_n)\to \bar{m}_{t}^r$, then by Lemma \ref{lmwelldefine}, for $n$ large enough, we have
$v_n\in \Bar{S}(t)\cap B(\frac{rt}{2})$, and by \eqref{xiajie estimate} and \eqref{shangjie estimate}, there exists a $\delta_1>0$ such that $$\liminf_{n\to \infty}\|v_n\|_{L^{p+2}}\geq \delta_1. $$
Thus letting $w_n:=\sqrt{\frac{s}{t}}v_n$, we have $w_n\in \Bar{S}(s)\cap B(\frac{rs}{2})\subset \Bar{S}(s)\cap B(r)$, and
\begin{align}\label{smonotonicity1}
\bar{m}_{s}^r \leq  \lim_{{n\rightarrow \infty}}E_b(w_n)
            \leq &\frac{s}{t} \lim_{{n\rightarrow \infty}}E_b(v_n)+\liminf_{{n\rightarrow \infty}}\frac{\lambda_2(\frac{s}{t})[(\frac{s}{t})^{p/2}-1]}{4}\int_{\mathbb{R}^3} (|x|^{-1}\ast|v_n|^{2})|v_n|^{2}dx\nonumber\\
            & - \liminf_{{n\rightarrow \infty}}\frac{\lambda_3(\frac{s}{t})[(\frac{s}{t})^{p/2}-1]}{p+2}\|v_n\|_{L^{p+2}}^{p+2}\nonumber\\
            \leq &\frac{s}{t} \bar{m}_t^r - \frac{\lambda_3(\frac{s}{t})[(\frac{s}{t})^{p/2}-1]}{p+2}\delta_1^{p+2}<\frac{s}{t}\bar{m}_t^r, \quad \mbox{since }\ \lambda_2\leq 0, \lambda_3>0,
\end{align}
which implies \eqref{smonotonicity}. Then for all $\mu\in (0,c)$,
$$\bar{m}_{c}^r =\frac{c-\mu}{c}\bar{m}_{c}^r +\frac{\mu}{c}\bar{m}_{c}^r < \frac{c-\mu}{c}\cdot \frac{c}{c-\mu}\bar{m}_{c-\mu}^r  + \frac{c}{c}\cdot \frac{c}{\mu}\bar{m}_{\mu}^r= \bar{m}_{c-\mu}^r + \bar{m}_{\mu}^r.$$
\end{proof}

\begin{proposition}\label{prop-localmini1}
Assume that $\lambda_1\leq 0, \lambda_2\leq 0, \lambda_3>0$ and $\frac{4}{3}\leq p<4$, all being fixed, then there exists a $r_0>0$, such that for every given $r>r_0$, there exists a $c_r$ with $0<c_r<1$, we have for any $c<c_r$ that any minimizing sequence of $\bar{m}_{c}^r$ is pre-compact in $X$. In particular, $\mathcal{M}_c^r\neq \emptyset$.
\end{proposition}

\begin{proof} We use the idea in the proofs of Proposition \ref{prop-limit1} and Proposition \ref{prop-globalmini} to prove this proposition. Let $\{u_n\}$ be an arbitrary minimizing sequence of $\bar{m}_c^r$. Namely,
\begin{equation}\label{0311}
E_b(u_n)\rightarrow \bar{m}_c^r,\ \|u_n\|_{L^2}^2=c~~\mbox{and}~~\|u_n\|_{\tilde{X}}^2\leq r.
\end{equation}
Then clearly $\{u_n\}$ is bounded in $X$. We first claim that if $\lambda_1=0$, then up to a subsequence, there exist a sequence $\{z_n\}\subset \mathbb{R}$ and $u\in X\backslash \{0\}$, such that
$$u_n(x_1,x_2, x_3-z_n)\to u\ \mbox{in}\  X.$$
Note that if $\lambda_1=0$, then $E_b(u_n)=E^{\infty}_b(u_n)$ and $\bar{m}_{c}^r=(\bar{m}_{c}^r)^{\infty}$. From \eqref{xiajie estimate} and \eqref{shangjie estimate}, we deduce that there exists a $\delta_2>0$ such that $$\liminf_{n\to \infty}\|u_n\|_{L^{p+2}}\geq \delta_2. $$
Thus by Lemma \ref{lemma compactness lemma I}, there exist a sequence $\{z_n\}\subset \mathbb{R}$ and $u\in X\backslash \{0\}$, such that
$$v_n \rightharpoonup u \quad  \mbox{in}~ X,$$
where $v_n:=u_n(x_1,x_2, x_3-z_n)$. Denote $c_1:=\|u\|_{L^2}^2$ and $c_n:=\|v_n-u\|_{L^2}^2$, thus by
$$\|v_n\|_{L^2}^2=\|v_n-u\|_{L^2}^2+\|u\|_{L^2}^2+o_n(1),\quad \|v_n\|_{\tilde{X}}^2=\|v_n-u\|_{\tilde{X}}^2+\|u\|_{\tilde{X}}^2+o_n(1),$$
we see that $u\in \bar{S}(c_1)\cap B(r)$ and $v_n-u\in \bar{S}(c_n)\cap B(r)$. Choosing a subsequence of $\{v_n\}$ (still denote by $\{v_n\}$) being such that $c_n\to (c-c_1)^{-}$, then by \eqref{smonotonicity},
\begin{align}\label{5.33}
(\bar{m}_{c}^r)^{\infty} = E^{\infty}_b(v_n)+o_n(1)= E^{\infty}_b(v_n-u)+ E^{\infty}_b(u)+o_n(1)&\geq (\bar{m}_{c_n}^r)^{\infty} + (\bar{m}_{c_1}^r)^{\infty} +o_n(1)\nonumber\\
              &\geq \frac{c_n}{c-c_1}(\bar{m}_{c-c_1}^r)^{\infty} + (\bar{m}_{c_1}^r)^{\infty} +o_n(1)\nonumber\\
              &=(\bar{m}_{c-c_1}^r)^{\infty} + (\bar{m}_{c_1}^r)^{\infty},
\end{align}
which is a contradiction with \eqref{5.2811}, if $\|u\|_{L^2}^2<c$. Therefore $\|u\|_{L^2}^2=c$, and $v_n\to u$ in $L^2$. Thus following the arguments in the proof of Proposition \ref{prop-globalmini}, we prove further that $v_n\to u$ in $X$, namely $u_n(x_1,x_2, x_3-z_n)\to u$ in $X$.

Now when $\lambda_1<0$, since $\{u_n\}$ is bounded in $X$, then up to a subsequence, there exists a $u\in X$, such that $u_n \rightharpoonup u $ in $X$. In particular, since $X\hookrightarrow H^1(\mathbb{R}^3)$ continuously, then $u_n \rightharpoonup u $ in $H^1$ and $u_n \to u$ in $L^q_{loc}(\mathbb{R}^3), \forall q\in [2,6)$. We claim that $u\neq 0$. Indeed, if so, then by the fact that $u_n \to 0$ in $L^q_{loc}(\mathbb{R}^3)$,  one could easily prove that
\begin{eqnarray}\label{0312}
\int_{\mathbb{R}^3}\frac{|u_n|^2}{|x|}dx=o_n(1),\ \mbox{ as }\ n\to \infty.
\end{eqnarray}
Thus as $n \to \infty$, we have
$$\bar{m}_c^r=E_b(u_n)+o_n(1)=E^{\infty}_b(u_n)+o_n(1).$$
Since $u_n\in \bar{S}(c)\cap B(r)$, we then deduce that
$$(\bar{m}_c^r)^{\infty}\leq E^{\infty}_b(u_n)=\bar{m}_c^r+o_n(1),$$
which implies $(\bar{m}_c^r)^{\infty}\leq \bar{m}_c^r$. However, since $(\bar{m}_c^r)^{\infty}$ is attained at some $u_0\in \bar{S}(c)\cap B(r)$ and $\lambda_1<0$, then
\begin{align*}
\bar{m}_c^r\leq E_b(u_0)= E^{\infty}_b(u_0)+ \frac{\lambda_1}{2}\int_{\mathbb{R}^3}\frac{|u_0|^2}{|x|}dx
=(\bar{m}_c^r)^{\infty} + \frac{\lambda_1}{2}\int_{\mathbb{R}^3}\frac{|u_0|^2}{|x|}dx < (\bar{m}_c^r)^{\infty},
\end{align*}
and a contradiction occurs. Hence $u_n \rightharpoonup u\neq 0 $ in $X$. Then since \eqref{5.2811} in Lemma \ref{lm-limit-nonlimit2}, we could argue as analogue as the case $\lambda_1=0$, to deduce that
$\|u\|_{L^2}^2=c$ and $u_n\to u$ in $L^2$, and further that $u_n\to u $ in $X$. At this point, we have already proved that any minimizing sequence of $\bar{m}_{c}^r$ is pre-compact in $X$. In particular, $u$ is a minimizer of $\bar{m}_{c}^r$.
\end{proof}

\begin{proof}[\textbf{Proof of Theorem \ref{th-localmini}}.] By Proposition \ref{prop-localmini1}, this theorem follows immediately.
\end{proof}

\section{Strong instability of standing waves with partial confine}
In this section, we study the existence and strong instability of standing waves for \eqref{equation partial} with $\lambda_1=0$.
We first prove Proposition \ref{proposition ground state partial}.
In order to solve the minimization problem \eqref{minimization ground state partial}, we consider the following minimization problem:
\begin{equation}\label{minimization1}
\widetilde{d}_K(\omega):=\inf\{\widetilde{S}_{b,\omega}(v):~v\in X\backslash \{0\},~~K_{b,\omega}(v)\leq 0\},
\end{equation}
where
\begin{align}\label{s1u}
\widetilde{S}_{b,\omega}(v):=&S_{b,\omega}(v)-\frac{K_{b,\omega}(v)}{3p+3}= \frac{3p-2}{6p+6}\|\nabla u\|_{L^2}^2+\frac{3p\omega}{6p+6}\int_{\mathbb{R}^3} |u|^{2}dx\nonumber\\&+b^2\frac{3p+2}{6p+6}\int_{\mathbb{R}^3} (x_1^2+x_2^2)|u|^{2}dx
+\lambda_2 \frac{3p-4}{12p+12}\int_{\mathbb{R}^3} (|x|^{-1}\ast|u|^{2})|u|^{2}dx.
\end{align}
If $K_{b,\omega}(v)<0$, then we have $K_\omega(\lambda v)>0$
for sufficiently small $\lambda>0$. Thus, there exists $\lambda_0\in (0,1)$ such that
$K_{b,\omega}(\lambda_0v)=0$. Moreover, it follows that
\[
\widetilde{S}_{b,\omega}(\lambda_0v)<\widetilde{S}_{b,\omega}(v).
\]
This implies that
\begin{equation}\label{minimization problemx partial}
\widetilde{d}_K(\omega)=\inf \{\widetilde{S}_{b,\omega}(v):~~v\in X\backslash \{0\},~~ K_{b,\omega}(v)= 0\}=d_K(\omega).
\end{equation}

\begin{lemma}\label{lemma minimization problem}
Let $\lambda_1=0$, $\lambda_2\geq 0$, $\lambda_3>0$, $\omega>0$ and $\frac{4}{3}\leq p<4$. Then there exists $u\in X\backslash \{0\}$, $K_{b,\omega}(u)=0$ and $\widetilde{S}_{b,\omega}(u)=\widetilde{d}_K(\omega)$.
\end{lemma}
\begin{proof}
We first show that $\widetilde{d}_K(\omega)>0$. For any $v\in X\setminus\{0\}$ satisfying $K_{b,\omega}(v)\leq 0$, we have
\begin{align}\label{kufl}
& \frac{5}{2}\|\nabla v\|_{L^2}^2+\frac{3\omega}{2}\int_{\mathbb{R}^3} |v|^{2}dx+\frac{b^2}{2}\int_{\mathbb{R}^3} (x_1^2+x_2^2)|v|^{2}dx
\nonumber\\&+\frac{7\lambda_2}{4}\int_{\mathbb{R}^3} (|x|^{-1}\ast|v|^{2})|v|^{2}dx \leq\frac{\lambda_3(3p+3)}{p+2}\|v\|_{L^{p+2}}^{p+2},
\end{align}
which implies that
\[
\frac{1}{2}H_{b,\omega}(v)\leq \frac{\lambda_3(3p+3)}{p+2}H_{b,\omega}(v)^{\frac{p}{2}+1},
\]
where $H_{b,\omega}(v)=\|\nabla v\|_{L^2}^2+\omega\| v\|_{L^2}^2+b^2\int_{\mathbb{R}^3} (x_1^2+x_2^2)|v|^{2}dx$. Thus, there exists $C_0>0$ such that $H_{b,\omega}(v)>C_0$ for all $K_{b,\omega}(v)\leq 0$.
This implies that there exists $C_1>0$ such that
\begin{align*}
\widetilde{S}_{b,\omega}(v)\geq& \frac{3p-2}{6p+6}\|\nabla v\|_{L^2}^2+\frac{3p\omega}{6p+6}\int_{\mathbb{R}^3} |v|^{2}dx+\frac{3p+2}{6p+6}\int_{\mathbb{R}^3} (x_1^2+x_2^2)|v|^{2}dx\geq C_1.
\end{align*}
Taking the infimum over $v$, we get $\widetilde{d}_K(\omega)>0$.

We now show the minimizing problem \eqref{minimization1} is attained.
Let $\{v_n\}$ be a minimizing sequence for \eqref{minimization1}, i.e.,
$\{v_n\}\subset X\backslash\{0\}$,  $K_{b,\omega}(v_n)\leq 0$ and $\widetilde{S}_{b,\omega}(v_n)\rightarrow \widetilde{d}_K(\omega)$ as $n\rightarrow \infty$.
  Thus, there exists $N$ such that $\widetilde{S}_{b,\omega}(v_n)\leq \widetilde{d}_K(\omega)+1$ for all  $n>N$.
This implies that $\{v_n\}$ is bounded in $X$.

Next, we claim that $\liminf_{n\rightarrow \infty}\|v_n\|_{L^{p+2}}^{p+2}>0$.
Indeed, if there exists a subsequence, still denoted by $\{v_n\}$, such that $\|v_n\|_{L^{p+2}}^{p+2}\rightarrow 0$ as $n\rightarrow \infty$, then by $K_{b,\omega}(v_n)\leq 0$, we deduce from \eqref{kufl} that $\widetilde{S}_{b,\omega}(v_n) \rightarrow 0$. This contradicts with the fact that $\widetilde{S}_{b,\omega}(v_n)\to \widetilde{d}_K(\omega)>0$.

Thus, we obtain
\[
\liminf_{n\rightarrow \infty}\|v_n\|_{L^{p+2}}^{p+2}>0,
\]
for all $\frac{4}{3}\leq p<4$.
Therefore, applying Lemma \ref{lemma compactness lemma I}, there exist a sequence $\{z_n\}\subset \mathbb{R}$ and $u\in X\backslash \{0\}$, such that, up to a subsequence,
$$v_n(x_1,x_2, x_3-z_n) \rightharpoonup u, \ \mbox{in } X. $$

Moreover, we deduce from the Brezis-Lieb lemma(Lemma 2.2) and Lemma 2.3 that
\begin{equation}\label{bl1}
K_{b,\omega}(v_n)-K_{b,\omega}(v_n-u)-K_{b,\omega}(u)\rightarrow 0,
\end{equation}
and
\begin{equation}\label{bl2}
\widetilde{S}_{b,\omega}(v_n)-\widetilde{S}_{b,\omega}(v_n-u)
-\widetilde{S}_{b,\omega}(u)\rightarrow 0.
\end{equation}
Now, we claim that $K_{b,\omega}(u)\leq 0$. If not, it follows from \eqref{bl1} and
$K_\omega(v_n)\leq 0$ that $K_{b,\omega}(v_n-u)\leq0$ for sufficiently large $n$. Thus, by the definition of
$\widetilde{d}_K(\omega)$, it follows that
\[
\widetilde{S}_{b,\omega}(v_n-u)\geq \widetilde{d}_K(\omega),
\]
which, together with $\widetilde{S}_{b,\omega}(v_n)\rightarrow \widetilde{d}_K(\omega)$, implies that
\[
\widetilde{S}_{b,\omega}(u)\leq 0,
\]
which is a contradiction with $\widetilde{S}_{b,\omega}(u)>0$. We thus obtain $K_{b,\omega}(u)\leq 0$.

Furthermore, we deduce from the definition of
$\widetilde{d}_K(\omega)$ and the weak lower semicontinuity of norm that
\[
\widetilde{d}_K(\omega)\leq \widetilde{S}_{b,\omega}(u)\leq \liminf_{n\rightarrow \infty}\widetilde{S}_{b,\omega}(v_n)=\widetilde{d}_K(\omega).
\]
This yields that
\[
\widetilde{S}_{b,\omega}(u)=\widetilde{d}_K(\omega).
\]

Finally, by a similar argument as \eqref{minimization problemx partial}, we can show that $K_{b,\omega}(u)=0$. This completes the proof.
\end{proof}
By the fact $d_K(\omega)=\widetilde{d}_K(\omega)$ and this lemma, we can obtain the existence of the minimizing problem \eqref{minimization ground state partial}.
\begin{corollary}\label{corollary}
Let $\lambda_1=0$, $\lambda_2\geq 0$, $\lambda_3>0$, $\omega>0$ and $\frac{4}{3}\leq p<4$. Then there exists $u\in X\backslash \{0\}$, $K_{b,\omega}(u)=0$ and $S_{b,\omega}(u)=d_K(\omega)$.
\end{corollary}
By a similar argument as Lemma \ref{lemma daoshuweiling}, we can obtain that the minimizer of \eqref{minimization ground state partial} is indeed the solution of \eqref{elliptic partial}.
\begin{lemma}\label{lemma daoshuweiling partial}
Let $\lambda_1=0$, $\lambda_2\geq 0$, $\lambda_3>0$, $\omega>0$ and $\frac{4}{3}\leq p<4$. Assume that $u\in X\backslash \{0\}$ such that $K_{b,\omega}(u)=0$ and $S_{b,\omega}(u)=d_K(\omega)$. Then $S'_{b,\omega}(u)=0$.
\end{lemma}

\begin{proof}[\textbf{Proof of Proposition \ref{proposition ground state partial}.}]
We now denote the set of all minimizers of \eqref{minimization ground state partial} by
\[
\mathcal{M}_{b,\omega}:=\{u\in X\backslash \{0\},~~S_{b,\omega}(u)=d_K(\omega),~K_\omega(u)=0\}.
\]
Then, by a standard argument, it easily follows that $\mathcal{M}_{b,\omega}=\mathcal{G}_{b,\omega}$, where $\mathcal{G}_{b,\omega}$ is the set of ground states of \eqref{elliptic partial}. We consequently obtain Proposition \ref{proposition ground state partial}.
\end{proof}

Next, we prove that the cross-manifold $\mathcal{N}$ defined by \eqref{cross-manifold} is not empty.
\begin{lemma}\label{cross-manifold feikong}
Let $\lambda_1=0$, $\lambda_2\geq 0$, $\lambda_3>0$, $\omega>0$, $\frac{4}{3}\leq p<4$ and $u$ be the ground
state related to \eqref{elliptic partial}. When $\frac{4}{3}\leq p<2$, assume further that $\partial_\lambda^2S_{b,\omega}(u_\lambda)|_{\lambda=1}\leq 0$ with $u_\lambda=\lambda^{\frac{3}{2}}u(\lambda x)$. Then $\mathcal{N}$ is not empty.
\end{lemma}
\begin{proof}
We firstly consider the case $2\leq p<4$.
Let $u\in X$ be the ground
state related to \eqref{elliptic partial}, it follows from Lemma
\ref{lemma pohozaev identity} that $K_{b,\omega}(u)=0$ and
$Q_b(u)=0$.
We see from \eqref{pohozaev identity1} that
\[
\frac{5\lambda_2}{4}\int_{\mathbb{R}^3} (|x|^{-1}\ast|u|^{2})|u|^{2}dx<\frac{3\lambda_3}{p+2}\|u\|_{L^{p+2}}^{p+2}
<\frac{15p\lambda_3}{2(p+2)}\|u\|_{L^{p+2}}^{p+2}.
\]
This, together with $Q_b(u)=0$, implies that
\[
\|\nabla u\|_{L^2}^2-b^2\int_{\mathbb{R}^3} (x_1^2+x_2^2)|v|^{2}dx>0.
\]
We consequently obtain that
\[
K_{b,\omega}(\lambda u)<0~~ \mbox{ and }~~
Q_b(\lambda u)<0,
\]
for all $\lambda>1$.
Now, we let $v=\lambda u$ and $v_\mu(x)=\mu^{\frac{2}{p}}v(\mu x)$. It easily follows that
\begin{align}\label{iuscaling}
K_{b,\omega}&(v_\mu)=\mu^{\frac{4}{p}-1}\left(\frac{5}{2}\|\nabla v\|_{L^2}^2-\frac{\lambda_3(3p+3)}{p+2}\|v\|_{L^{p+2}}^{p+2}\right)
+\frac{3\omega}{2}\mu^{\frac{4}{p}-3}\int_{\mathbb{R}^3} |v|^{2}dx\nonumber\\&+\frac{b^2}{2}\mu^{\frac{4}{p}-5}\int_{\mathbb{R}^3} (x_1^2+x_2^2)|v|^{2}dx+\frac{7\lambda_2}{4}\mu^{\frac{8}{p}-5}\int_{\mathbb{R}^3} (|x|^{-1}\ast|v|^{2})|v|^{2}dx\leq \mu^{\frac{4}{p}-1}K_{b,\omega}(v)<0,
\end{align}
for all $\mu>1$.
Next, we define
\begin{align}\label{quscaling}
f(\lambda,\mu):=Q_b(v_\mu)=\mu^{\frac{4}{p}-1}(g(\lambda,\mu)
-\mu^{-4}\lambda^2g(1,1)),
\end{align}
where
\[
g(\lambda,\mu)=\lambda^2\|\nabla u\|_{L^2}^2+\frac{\lambda_2\mu^{\frac{4}{p}-4}\lambda^4}{4}\int_{\mathbb{R}^3} (|x|^{-1}\ast|u|^{2})|u|^{2}dx
-\frac{\lambda_3p^*\lambda^{p+2}}{p+2}\|u\|_{L^{p+2}}^{p+2},
\]
with $p^*=\frac{3p}{2}$.
Due to $g(1,1)>0$, there exists $\delta>0$ such that $g(\lambda,\mu)>0$ for all $1\leq \lambda <1+\delta$ and $1\leq \mu <1+\delta$. This, together with $f(1,1)=0$ implies that there exists $(\lambda_0,\mu_0)$ with $\lambda_0 > 1,\mu_0>1$, such that $f(\lambda_0,\mu_0)>0$. On the other hand, we have $f(\lambda_0,1)<0$.
Therefore, there exists $\mu_*>1$ such that $Q_b(v_{\mu_*})=0$.
Thus, $v_{\mu_*}\in \mathcal{N}$.

We firstly consider the case $\frac{4}{3}\leq p<2$.
Let $u\in X$ be the ground
state related to \eqref{elliptic partial}, it follows from Lemma
\ref{lemma pohozaev identity} that $K_{b,\omega}(u)=0$ and
$Q_b(u)=0$. Set $u_\lambda=\lambda^{\frac{3}{2}}u(\lambda x)$, we define
\begin{align*}
f(\lambda):=S_{b,\omega}(u_\lambda)=& \frac{\lambda^2}{2}\|\nabla u\|_{L^2}^2+\frac{\omega}{2}\int_{\mathbb{R}^3} |u|^{2}dx+\frac{\lambda^{-2}b^2}{2}\int_{\mathbb{R}^3} (x_1^2+x_2^2)|u|^{2}dx
\nonumber\\&+\frac{\lambda_2\lambda}{4}\int_{\mathbb{R}^3} (|x|^{-1}\ast|u|^{2})|u|^{2}dx -\frac{\lambda_3\lambda^{p^*}}{p+2}\|u\|_{L^{p+2}}^{p+2},
\end{align*}
 \begin{align}\label{scaling12}
g(\lambda):=K_{b,\omega}(u_\lambda)=& \frac{5\lambda^2}{2}\|\nabla u\|_{L^2}^2+\frac{3\omega}{2}\int_{\mathbb{R}^3} |u|^{2}dx+\frac{\lambda^{-2}b^2}{2}\int_{\mathbb{R}^3} (x_1^2+x_2^2)|u|^{2}dx
\nonumber\\&+\frac{7\lambda_2\lambda}{4}\int_{\mathbb{R}^3} (|x|^{-1}\ast|u|^{2})|u|^{2}dx -\frac{\lambda_3(3p+3)\lambda^{p^*}}{p+2}\|u\|_{L^{p+2}}^{p+2}.
\end{align}
 Note that
 \begin{equation}\label{scaling14}
f''(\lambda)= \|\nabla u\|_{L^2}^2+3\lambda^{-4}b^2\int_{\mathbb{R}^3} (x_1^2+x_2^2)|u|^{2}dx -\frac{\lambda_3p^*(p^*-1)\lambda^{p^*-2}}{p+2}\|u\|_{L^{p+2}}^{p+2},
\end{equation}
which, together with $f''(1)\leq 0$, implies that $f''(\lambda)<0$ for all $\lambda>1$. We consequently obtain that
\[
\frac{1}{\lambda}h(\lambda)= f'(\lambda)< f'(1)=0,
\]
 for all $\lambda>1$, where $h(\lambda):=Q_b(u_\lambda)$.
 On the other hand, from
 \begin{equation}\label{scaling15}
g''(\lambda)=5\|\nabla u\|_{L^2}^2+3\lambda^{-4}b^2\int_{\mathbb{R}^3} (x_1^2+x_2^2)|u|^{2}dx -\frac{\lambda_3(3p+3)p^*(p^*-1)\lambda^{p^*-2}}{p+2} \|u\|_{L^{p+2}}^{p+2},
\end{equation}
$K_{b,\omega}(u)=0$ and $f''(1)\leq 0$, we can obtain that $g''(1)<0$. It consequently follows that $g''(\lambda)<0$ for all $\lambda\geq1$.
 Moreover, we deduce from \eqref{pohozaev identity0} and $Q_b(u)=0$ that
 \begin{align*}
g'(1)=&5\|\nabla u\|_{L^2}^2-b^2\int_{\mathbb{R}^3} (x_1^2+x_2^2)|u|^{2}dx+\frac{7\lambda_2}{4}\int_{\mathbb{R}^3} (|x|^{-1}\ast|u|^{2})|u|^{2}dx-\frac{\lambda_3(3p+3)p^*}{p+2}\|u\|_{L^{p+2}}^{p+2}
\nonumber\\\leq &4\|\nabla u\|_{L^2}^2+4\lambda_2\int_{\mathbb{R}^3} (|x|^{-1}\ast|u|^{2})|u|^{2}dx-\frac{\lambda_3(3p+2)p^*}{p+2}\|u\|_{L^{p+2}}^{p+2}<0.
\end{align*}
Combining the above estimates, we obtain
 \begin{equation}\label{scaling17}
K_{b,\omega}(u_\lambda)<0~~\mbox{and}~~Q_b(u_\lambda)<0,
\end{equation}
for all $\lambda>1$.

Now, we let $v=u_\lambda$
 and $v_\mu(x)=\mu^{\frac{2}{p}}v(\mu x)$. By a similar argument as the case $2\leq p<4$, there exists $\mu_*>1$, $Q_b(v_{\mu_*})=0$ and $K_{b,\omega}(v_{\mu_*})<0$.
Thus, $v_{\mu_*}\in \mathcal{N}$.
\end{proof}

\begin{lemma}\label{lemma dayuling}
Let $\lambda_1=0$, $\lambda_2\geq 0$, $\lambda_3>0$, $\omega>0$, $\frac{4}{3}\leq p<4$. Then, $\alpha(\omega)>0$, where $\alpha(\omega)$ is defined by \eqref{minimization cross-manifold}.
\end{lemma}
\begin{proof}
Let $v\in \mathcal{N}$. From $K_{b,\omega}(v)<0$, we have $v\neq 0$. From $Q_b(v)=0$, we have
\begin{align}\label{scaling20}
S_{b,\omega}(v)=S_{b,\omega}(v)-\frac{2}{3p}Q_b(v)=&\frac{3p-4}{6p}
\|\nabla v\|^2_{L^2}+\frac{\omega}{2} \|v\|^2_{L^2}+b^2\frac{3p+4}{6p}\int_{\mathbb{R}^3} (x_1^2+x_2^2)|v|^{2}dx \nonumber\\
&+ \frac{\lambda_2(3p-2)}{12p} \int_{\mathbb{R}^3} (|x|^{-1}\ast|v|^{2})|v|^{2}dx.
\end{align}
Since $\frac{4}{3}\leq p<4$, it follows from \eqref{scaling20} and $v\neq0$ that $S_\omega(v)>0$ for all $v\in \mathcal{N}$.
Thus $\alpha(\omega)\geq 0$. In the following, we shall show by contradiction that $\alpha(\omega)\neq 0$, which then completes the proof.
%

Indeed, if $\alpha(\omega)=0$ happened, then from \eqref{scaling20} there were a sequence $\{v_n\}\subset \mathcal{N}$ such that $Q_b(v_n)=0$,
$K_{b,\omega}(v_n)<0$, and $S_{b,\omega}(v_n)\rightarrow 0$ as $n\rightarrow \infty$. Since $p\geq \frac{4}{3}$, \eqref{scaling20} implies that
\begin{align}\label{coscaling21}
 \|v_n\|^2_{L^2}\rightarrow 0,~~~~~~~~~~\int_{\mathbb{R}^3} (x_1^2+x_2^2)|v_n|^{2}dx\rightarrow 0,
~~\mbox{and}~~ \int_{\mathbb{R}^3} (|x|^{-1}\ast|v_n|^{2})|v_n|^{2}dx\rightarrow 0,
\end{align}
as $n\rightarrow \infty$.
On the other hand, it follows from $K_\omega(v_n)<0$ that
\[
\frac{1}{2}\left(\|\nabla v_n\|_{L^2}^2+\omega\| v_n\|_{L^2}^2+b^2\int_{\mathbb{R}^3} (x_1^2+x_2^2)|v_n|^{2}dx\right)< \frac{\lambda_3(3p+3)}{p+2}\|v_n\|_{L^{p+2}}^{p+2}
\leq C_3\|\nabla v_n\|_{L^2}^{\frac{3p}{2}}\|v_n\|_{L^2}^{\frac{4-p}{2}}.
\]
This implies that
$$\frac{1}{2}\|\nabla v_n\|_{L^2}^2\leq C_3\|\nabla v_n\|_{L^2}^{\frac{3p}{2}}\|v_n\|_{L^2}^{\frac{4-p}{2}},$$
or equivalently
$$\frac{1}{2} \leq C_3\|\nabla v_n\|_{L^2}^{\frac{3p-4}{2}}\|v_n\|_{L^2}^{\frac{4-p}{2}},$$
which is an contradiction with \eqref{coscaling21}, since $\frac{4}{3}\leq p<4$.
\end{proof}

\begin{proof}[\textbf{Proof of Lemma \ref{lemma keyestimate partial}.}]
Firstly, by a similar argument as Theorem \ref{theorem sharp}, we can prove that the set $\mathcal{K}_\omega$ is invariant under the flow of \eqref{equation partial}, so we omit it.


Next, we prove the key estimate \eqref{keyestimate partial}.
Fix $t\in[0,T^*)$, and denote $\psi=\psi(t)$. Thus $\psi$ satisfies that $S_{b,\omega}(\psi)<S_\omega(u)$, $K_{b,\omega}(\psi)<0$ and $Q_b(\psi)<0$.
For $\lambda>0$, let $\psi_\lambda=\lambda^{\frac{3}{p+2}}\psi(\lambda x)$, it easily follows that
\begin{align*}
K_{b,\omega}(\psi_\lambda)=& \frac{5}{2}\lambda^{\frac{4-p}{p+2}}\|\nabla u\|_{L^2}^2+\frac{3\omega}{2}\lambda^{\frac{-3p}{p+2}}\int_{\mathbb{R}^3} |u|^{2}dx+\frac{b^2}{2}\lambda^{\frac{-5p-4}{p+2}}\int_{\mathbb{R}^3} (x_1^2+x_2^2)|u|^{2}dx\\&
+\frac{7\lambda_2}{4}\lambda^{\frac{2-5p}{p+2}}\int_{\mathbb{R}^3} (|x|^{-1}\ast|u|^{2})|u|^{2}dx -\frac{\lambda_3(3p+3)}{p+2}\|u\|_{L^{p+2}}^{p+2},
\end{align*}
\begin{align*}
Q_b(\psi_\lambda)=&\lambda^{\frac{4-p}{p+2}}\|\nabla \psi\|_{L^2}^2-\lambda^{\frac{-5p-4}{p+2}}b^2\int_{\mathbb{R}^3} (x_1^2+x_2^2)|\psi|^{2}dx
\\&
+\frac{\lambda_2\lambda^{\frac{2-5p}{p+2}}}{4}\int_{\mathbb{R}^3} (|x|^{-1}\ast|\psi|^{2})|\psi|^{2}dx-
\frac{3\lambda_3p}{2(p+2)}\|\psi\|_{L^{p+2}}^{p+2}.
\end{align*}
Thus, $Q_b(\psi)<0$ implies that there exists $\lambda_*>1$ such that $Q_b(\psi_{\lambda_*})=0$, and $Q_b(\psi_{\lambda})<0$ for all
$\lambda\in [1,\lambda_*)$. For $\lambda\in [1,\lambda_*]$, due to $K_\omega(\psi)<0$, $K_\omega(\psi_\lambda)$ has the following
two possibilities:

(i) $K_{b,\omega}(\psi_\lambda)<0$ for all $\lambda\in [1,\lambda_*]$.

(ii) There exists $\mu\in (1,\lambda_*]$ such that $K_{b,\omega}(\psi_\mu)=0$.

For the case (i), we have $Q_b(\psi_{\lambda_*})=0$ and $K_{b,\omega}(\psi_{\lambda_*})<0$. This implies that $S_{b,\omega}(\psi_{\lambda_*})\geq \alpha(\omega)\geq S_{b,\omega}(u)$. Moreover, we have
\begin{align}\label{suscaling par}
S_{b,\omega}(\psi)&-S_{b,\omega}(\psi_\lambda)= \frac{1}{2}(1-\lambda^{\frac{4-p}{p+2}})\|\nabla \psi\|_{L^2}^2+\frac{\omega}{2}(1-\lambda^{\frac{-3p}{p+2}})\int_{\mathbb{R}^3} |\psi|^{2}dx\nonumber\\&+\frac{b^2}{2}(1-\lambda^{\frac{-5p-4}{p+2}})\int_{\mathbb{R}^3} (x_1^2+x_2^2)|\psi|^{2}dx
+\frac{\lambda_2}{4}(1-\lambda^{\frac{2-5p}{p+2}})\int_{\mathbb{R}^3} (|x|^{-1}\ast|\psi|^{2})|\psi|^{2}dx,
\end{align}
and
\begin{align}\label{quscaling003par}
Q_b(\psi)-Q_b(\psi_\lambda)=&(1-\lambda^{\frac{4-p}{p+2}})\|\nabla \psi\|_{L^2}^2-(1-\lambda^{\frac{-5p-4}{p+2}})b^2\int_{\mathbb{R}^3} (x_1^2+x_2^2)|\psi|^{2}dx\nonumber\\&
+\frac{\lambda_2}{4}(1-\lambda^{\frac{2-5p}{p+2}})\int_{\mathbb{R}^3} (|x|^{-1}\ast|\psi|^{2})|\psi|^{2}dx.
\end{align}
Thus, it follows from $\frac{4}{3}\leq p<4$ and $\lambda_*>1$ that
\begin{equation}\label{key estimate31}
S_{b,\omega}(\psi)-S_{b,\omega}(\psi_{\lambda_*})\geq \frac{1}{2}(Q_b(\psi)-Q_b(\psi_{\lambda_*}))=\frac{1}{2}Q_b(\psi).
\end{equation}
For the case (ii), we have $K_{b,\omega}(\psi_\mu)=0$ and $Q_b(\psi_\mu)\leq0$. Applying Proposition \ref{proposition ground state partial}, it follows
that $S_{b,\omega}(\psi_\mu)\geq S_{b,\omega}(u)$. And referring to \eqref{suscaling par} and \eqref{quscaling003par}, we can obtain
\begin{equation}\label{key estimate32}
S_{b,\omega}(\psi)-S_{b,\omega}(\psi_\mu)\geq \frac{1}{2}(Q_b(\psi)-Q_b(\psi_\mu))\geq\frac{1}{2}Q_b(\psi).
\end{equation}
Since $S_{b,\omega}(\psi_{\lambda_*})\geq S_\omega(u)$ and $S_{b,\omega}(\psi_\mu)\geq S_{b,\omega}(u)$, from \eqref{key estimate31} and \eqref{key estimate32}, we all get
\begin{equation*}
Q_b(\psi(t))\leq 2(S_{b,\omega}(\psi_0)-S_{b,\omega}(u)),
\end{equation*}
for all $t\in [0,T^*)$. This completes the proof.
\end{proof}

\begin{proof}[\textbf{Proof of Theorem \ref{Theorem instability partial}.}]
When $\frac{4}{3}\leq p<2$,  let $u$ be the ground state related to \eqref{elliptic partial}, under the assumption $\partial_\lambda^2S_{b,\omega}(u_\lambda)|_{\lambda=1}\leq 0$ with $u_\lambda(x)=\lambda^{\frac{3}{2}}u(\lambda x)$,
we see from \eqref{scaling17} that
\[
S_{b,\omega}(u_{\lambda})<S_{b,\omega}(u),~~K_{b,\omega}(u_{\lambda})<0,~~Q_b(u_{\lambda})<0,
\]
 for all $\lambda>1$. By the same argument as Theorem \ref{th instability no partial}, we can obtain the initial data $\chi_{M_0}u_{\lambda_0}\in \Sigma\cap \mathcal{K}_\omega$ such that the corresponding solution $\psi_{\lambda_0}(t)$ of \eqref{equation partial}
 blows up in finite time. Hence, we can prove this theorem.

  When $2\leq p<4$, we see from Lemma 6.2 that
 \[
 S_{b,\omega}(\lambda u)<S_{b,\omega}(u),~~
 K_{b,\omega}(\lambda u)<0~~ and~~
Q_b(\lambda u)<0,
\]
for all $\lambda>1$.
Therefore, we can prove the strong instability along the above lines by replacing $u_{\lambda_0}$ by $\lambda_0u$.
 This completes the proof.
\end{proof}

\vspace{0.5cm}

{\bf Acknowledgments}

Binhua Feng is supported by the National Natural Science Foundation of China (No. 11601435). Tingjian Luo would like to thank Professor Louis Jeanjean for helpful discussion on this manuscript, and he is supported by the National Natural Science Foundation of China (11501137), and the Guangdong Basic and Applied Basic Research Foundation (2016A030310258, 2020A1515011019).

\end{document}